%% file: IPAAL-SIAM-04-06-20.tex
\documentclass{siamart0516}

\usepackage{amsfonts}
\usepackage{multirow}
\usepackage{float}
\usepackage{array}
\usepackage[labelfont=bf]{caption}
\usepackage{textcomp}
\usepackage{amsmath}
\usepackage{amssymb}
\usepackage{algorithmic}
\usepackage{mathtools}
\usepackage{enumitem}

\usepackage{xcolor}
\usepackage{colortbl}
\usepackage{booktabs}

\setitemize{noitemsep,topsep=0pt,parsep=0pt,partopsep=0pt}

\newcommand{\cP}{{\cal{P}}}

\newcommand{\lm}{p}
\newcommand{\pen}{c}
\newcommand{\z}{z}

\input{def}

\input{Teste_shared}

\ifpdf
\hypersetup{
  pdftitle={IPAAL},
  pdfauthor={}
}
\fi

\begin{document}

\maketitle

% REQUIRED
\begin{abstract}
This paper proposes and establishes the iteration-complexity of an inexact proximal accelerated augmented Lagrangian (IPAAL) method 
			for solving linearly constrained smooth  nonconvex composite optimization problems.
			Each IPAAL iteration consists of inexactly solving a proximal augmented Lagrangian subproblem  by an accelerated composite gradient (ACG) method 
			followed by a suitable Lagrange multiplier update. It is shown that IPAAL generates an approximate  stationary solution 
			in at most ${\cal O}(\log(1/\rho)/\rho^{3})$ ACG iterations,  where $\rho>0$ is the given tolerance. 
	        It is also shown that the previous complexity bound can be sharpened to ${\cal O}(\log(1/\rho)/\rho^{2.5})$ under
    	additional mildly stronger assumptions.	The above bounds are derived assuming  that
    	the initial point is neither feasible nor the domain of the composite term of the objective function is bounded. 
    	Some preliminary numerical results are presented to illustrate the performance of the IPAAL method.
\end{abstract}

% REQUIRED

% REQUIRED

\section{Introduction} \label{sec:int}
	
	This paper presents an inexact proximal accelerated augmented Lagrangian
	(IPAAL) method for solving
	the linearly constrained smooth  nonconvex composite optimization problem
	\begin{equation} \label{optl0}
	\phi^*:=\min \{ \phi(z) :=  f(z) + h(z) : A z  =b\},
	\end{equation}
	where  $A:\Re^n \mapsto \Re^l$ is a linear operator, $b\in\Re^l$, $h:\Re^n \to (-\infty,\infty]$ is a closed proper convex function, and $f$ is a real-valued differentiable (possibly nonconvex) function whose gradient is $L$--Lipschitz and which, for some $0<m\leq L$, satisfies
\begin{equation}\label{eq:PenaltyProb-introa}
 f(u) \geq f(z)+\left\langle \nabla f(z),u-z\right\rangle -\frac{m}{2}\|u-z\|^{2} 
\quad \forall \, z,u \in \dom\, h.
\end{equation}
For a given tolerance pair $(\hat \rho,\hat \eta) \in \Re^2_{++}$,
its goal is to find a triple
$(\hat z,\hat p,\hat v)$ satisfying
\begin{align}\label{eq:approx_stationary1}
		\hat v\in\nabla f(\hat z)+\partial h(\hat z)+A^*\hat \lm,\qquad\|\hat v\|\leq\hat{\rho}, \qquad 
		\|A\hat z-b\|\leq \hat \eta.
		\end{align}
More specifically, the $\theta$-IPAAL method % for solving \eqref{optl0}
is based on the
	$\theta$-augmented Lagrangian ($\theta$-AL) function ${\cal L}^\theta_c(z;p)$ defined as
	\begin{equation}\label{lagrangian2}
	{\cal L}^\theta_c(z;\lm):=f(z)+h(z)+(1-\theta)\left\langle\lm,Az-b\right\rangle+\frac{\pen}{2}\|Az-b\|^2,
	\end{equation}
	where $\theta \ge 0$ is a given parameter.
	Note that when $\theta=0$,
	${\cal L}^\theta_\pen(\cdot,\cdot)$
	reduces to the well known quadratic augmented
	Lagrangian function which has been thoroughly studied
	in the literature (see for example \cite{AybatAugLag,Ber1,LanMonteiroAugLag,ShiqiaMaAugLag16,MR0418919}). Moreover,
	when $\theta=1$, ${\cal L}^\theta_\pen(\cdot,\cdot)$ does not depend on $\lm$ and reduces to the quadratic penalty function frequently used in penalty methods for solving
	\eqref{optl0}.
	Roughly speaking, for a fixed
	%tolerance $\bar \rho>0$ and
	stepsize $\lam>0$,
	the static version of $\theta$-IPAAL method repeatedly
	performs the following iteration:
	given $(z_{k-1},\lm_{k-1})  \in \dom h \times \Re^l$, it computes $(z_k,\lm_{k})$ as
	\begin{align}
	z_k &\approx \argmin_z \left \{ \lam {\cal L}^\theta_\pen(z,\lm_{k-1}) + \frac12 \| z-z_{k-1}\|^2 \right\} \quad \label{exactzk0} \\
	\lm_k &= (1-\theta) \lm_{k-1} + c (Az_k-b).\label{exactpk0}
	\end{align}
	where $z_k$ in \eqref{exactzk0} should be understood as
	a suitable approximate solution of the underlying
	$\theta$-prox-AL subproblem. To complete the above
	outline of the $\theta$-IPAAL method, we now describe
	how $z_k$ is computed without elaborating on its
	inexactness. It can be easily seen that
	\eqref{eq:PenaltyProb-introa} implies that the objective
	function of  \eqref{exactzk0} is
	strongly convex whenever $\lambda<1/m$.
	The $\theta$-IPAAL method then sets $\lam=\tau/m$ for
	some suitably chosen $\tau \in (0,1)$  and then
	approximately solves the
	corresponding subproblem \eqref{exactzk0} by a
	strongly convex version of an
	accelerated composite gradient (ACG) method (see for example \cite{beck2009fast,MontSvaiter_fista,nesterov1983method}) to obtain $z_k$. Also, it is shown that each pair $(z_k,\lm_k)$ obtained in the
	above manner can be refined to a
	triple $(\hat z,\hat v, \hat \lm )=(\hat z_k,\hat v_k, \hat \lm_k )$ satisfying
	the inclusion in \eqref{eq:approx_stationary1}
	%$ \hat v_k\in\nabla f(\hat z_k)+\partial h(\hat z_k)+A^*\hat \lm_k$,
	and the static
	$\theta$-IPAAL method is then stopped whenever
	the first inequality in \eqref{eq:approx_stationary1} is also satisfied.
% 	the residual $\hat v_k$ satisfies
% 	$\|\hat v_k\| \le \hat \rho$ for some fixed tolerance
% 	$\hat \rho >0$.
	Finally, the static $\theta$-IPAAL method
	is shown to satisfy
	the following properties:
	%are shown for the above static $\theta$-AAL method:
% 	1) for every $k \ge 1$, the iterate $z_k$ satisfies
% 	$\|Az_k-b\| = {\cal O}(1/(\sqrt{c})$; and
	1) it stops in
	${\cal O}(\sqrt{c}\log(c)/\hat \rho^2)$ ACG iterations;
	2) for every $k \ge 1$, the refined iterate $\hat z_k$ satisfies
	$\|A\hat z_k-b\| = {\cal O}(1/\sqrt{c})$.
	
	Observe that property 2) guarantees that $\hat z_k$
	is a near feasible point, i.e., satisfies the second inequality
	in \eqref{eq:approx_stationary1}, only when $c$ is sufficiently large.
	Based on this remark,
	a dynamic version of the 
	$\theta$-IPAAL method for
	finding a triple $(\hat z,\hat v,\hat p)$
	satisfying \eqref{eq:approx_stationary1}
	is also considered. More specifically, it chooses an initial
	penalty parameter $c$ and it repeatedly:
	a) invokes the static $\theta$-IPAAL with the current $c$ to
	obtain a triple $(\hat z,\hat v,\hat p)$ satisfying the inclusion
	and the first
	inequality in \eqref{eq:approx_stationary1};
	and b) doubles $c$ whenever the second inequality in \eqref{eq:approx_stationary1}
	is violated,
	until a triple $(\hat z,\hat v,\hat p)$ satisfying \eqref{eq:approx_stationary1} is obtained.
% 	a) 
% 	doubles $c$ whenever the last computed $\hat z_k$
% 	satisfying $\|\hat v_k\| \le \hat \rho$ violates
% 	the feasibility criterion,
% 	and b) invokes the static $\theta$-IPAAL
% 	with the new updated value of $c$.
	It is then shown that
	the ACG iteration-complexity of this
	dynamic variant is
	${\cal O}([1/(\hat \eta \hat \rho^2)]\log(1/\hat\eta))$.
	It is also shown that
	the previous complexity can be sharpened to
	 ${\cal O}([1/(\sqrt{\hat \eta} \hat \rho^2)]\log(1/\hat\eta))$ under the mildly stronger assumptions that:
	a) $\inte(\dom h) \cap \{ z : Az = b\}$ is nonempty;
	b) for some $\bar c \ge 0$,
	the quadratic penalty function
	${\cal L}^1_{\bar c}$ has bounded level set, and;
	c)	$h$ belongs to a special class of
	closed convex functions which contains all indicator
	functions
	of closed convex sets.
	Finally, it is worth emphasizing that all the results
	mentioned above are derived without assuming that the
	initial point $z_0 \in \dom h$ is feasible, i.e.,
	satisfies $Az_0=b$.

{\it Related works.}
We first discuss papers dealing with related algorithms for solving the convex version of \eqref{optl0} and other related monotone problems.
Iteration-complexity analysis of quadratic penalty methods for solving  \eqref{optl0} under the assumption that
$f$ is convex and  $h$  is a convex indicator function was first studied in \cite{LanRen2013PenMet} and further explored  in
\cite{Aybatpenalty,IterComplConicprog}. Iteration-complexity of first-order augmented Lagrangian methods for solving
the latter class  of linearly constrained convex programs was studied in  \cite{AybatAugLag,LanMonteiroAugLag,ShiqiaMaAugLag16,zhaosongAugLag18,Patrascu2017,YangyangAugLag17}.
Inexact proximal point methods using  accelerated gradient algorithms to solve their prox-subproblems were previously considered in \cite{GlanPDaccel2014,YHe2,YheMoneiroNash,OliverMonteiro,MonteiroSvaiterAcceleration}
 in the setting of convex-concave saddle point problems and monotone variational inequalities.

We now discuss papers dealing with related algorithms for solving \eqref{optl0}
when $f$ is nonconvex and $A=0$, i.e., the unconstrained version of \eqref{optl0}.
Paper \cite{nonconv_lan16}  proposed an accelerated gradient framework to solve an unconstrained problem with better iteration-complexity than the usual composite gradient method. Since then, many authors have proposed other accelerated frameworks for solving the unconstrained counterpart of \eqref{optl0} under different assumptions on the functions $f$ and $h$
(see, for example, \cite{Aaronetal2017,Paquette2017,Ghadimi2019,Li_Lin2015,CatalystNC}).
In particular, by exploiting the lower curvature $m$, \cite{Aaronetal2017,Paquette2017,CatalystNC} proposed some algorithms which improve the iteration-complexity bound of \cite{nonconv_lan16} in terms of the dependence on  the Lipschitz constant. Finally,  there has been a growing interest in the iteration-complexity  of methods for solving optimization problems using  second order information (see, for example,  \cite{Aaronetal2017,CartToint,MonteiroSvaiterNewton,NesterovSec_ord}).

There are only a few papers analyzing  iteration-complexity of quadratic penalty and/or augmented Lagrangian type methods for solving \eqref{optl0} in its general form, i.e., $A\neq 0$ and $f$  nonconvex. 
This paragraph discusses the quadratic penalty type methods
while the one below discusses the augmenetd Lagrangian type
methods. Quadratic penalty type methods were studied in 
\cite{WJRproxmet1,WJRComputQPAIPP,PPmetNonconvex2019}. More specifically, paper \cite{WJRproxmet1}
proposed a quadratic penalty  accelerated inexact proximal point (QP-AIPP) method
which can be viewed as an instance of the $\theta$-IPAAL method outlined above
with $\theta=1$, and hence with $\eqref{lagrangian2}$ being
the usual quadratic penalty function for \eqref{optl0}.
The QP-AIPP and the $\theta$-IPAAL methods share  similar ACG iteration-complexity bounds. More specifically, the ACG iteration-complexity of
the QP-AIPP method is 
$\mathcal{O}(1/(\hat \rho^2\hat \eta))$ which is similar
(up to a logarithm multiplicative term) to
the first bound derived  for the $\theta$-IPAAL method
mentioned above.
Paper \cite{WJRComputQPAIPP} proposes a more computationally efficient
variant of QP-AIPP, namely, R-QP-AIPP,
which adaptively chooses the prox-stepsize $\lambda$ in
a more aggressive manner and as a result generates possibly nonconvex
subproblems \eqref{exactzk0} which are tentatively solved by
a standard strongly convex version of an ACG variant.
% is adaptively and aggressively
% A drawback of the QP-AIPP method that is that its computational may affect its computational performance  is that it requires the prox-stepsize $\lam$ (discussed above for the  $\theta$-IPAAL) to be, possibly, very small.
% In \cite{WJRComputQPAIPP}, the authors introduced a variant of the  QP-AIPP method  which does not have the aforementioned drawback. Namely, the prox-stepsize $\lambda$  is adaptively and aggressively updated in order to improve the computational efficiency of the scheme.
% The latter scheme shares similar iteration-complexity bounds as that of the QP-AIPP method. 
More recently, an alternative
proximal quadratic penalty method for \eqref{optl0} whose subproblems are also
approximately solved by a strongly convex version of an ACG variant
is studied in \cite{PPmetNonconvex2019}.
An ${\cal O}(\log(1/\hat \eta)/(\hat \rho^{2}\sqrt{\hat \eta}))$ ACG-iteration-complexity
is established for the method
under the assumption that $\dom h$ is bounded and the Slater condition
mentioned above holds.
Finally, paper \cite{MinMax-RenWilliam} studies the complexity of
a quadratic penalty based method
for solving \eqref{optl0} under the assumption that
$f(\cdot)= \max \{\Phi(\cdot,y) : y \in Y\}$ where $Y$ is a compact convex set,
$-\Phi(x,\cdot)$ is proper lower semi-continuous convex for every
$x \in \dom h$, and $\Phi(\cdot,y)$ is nonconvex differentiable
on $\dom h$ and its gradient is uniformly Lipschitz continuous
on $\dom h$ for every $y \in Y$.

Paper \cite{ProxAugLag_Ming} studies the iteration-complexity of
a linearized version of the augmented Lagrangian method to solve \eqref{optl0}
but assumes the strong condition (among a few others) that $h=0$,
which most important problems arising in applications do not satisfy.
Paper \cite{HongPertAugLag} studies an unaccelerated augmented Lagrangian
inexact proximal method
for \eqref{optl0} based on the $\theta$-AL function \eqref{lagrangian2} with
$\theta$ chosen in $(0,1]$
and establishes an ${\cal O}(1/(\hat \eta^4+\hat\rho^4))$ iteration-complexity
where each iteration exactly solves a subproblem
of the form \eqref{exactzk0} except that the function
$f$ that appears \eqref{lagrangian2}
is replaced by its linearization at $z_{k-1}$ and the prox term
$\|z-z_{k-1}\|^2$ is replaced by $\|z-z_{k-1}\|^2_{B^*B}$ for some
matrix $B$ such that $B^*B+A^*A - I$ is positive semidefinite.
In contrast to the $\theta$-IPAAL method studied in this paper,
their method requires the restrictive condition
that its initial point $x_0$ be feasible, i.e.,
satisfy $Ax_0=b$ and $x_0 \in \dom h$.

    %The  study in \cite{PPmetNonconvex2019} was further explored in \cite{HybridPenaltyAugLag19} by using some hybrid strategies based on penalty and augmented Lagrangian methods. More specifically, the scheme is divided into three parts. First, some  iterations of a proximal  augmented Lagrangian method is applied  to estimate a  Lagrange multiplier $\lm_0$. Second, a proximal quadratic penalty scheme similar to the one proposed in \cite{PPmetNonconvex2019} is applied, where the usual quadratic penalty function for \eqref{optl0} is replaced by ${\cal L}^\theta_\pen(\cdot,\lm_0)$ with $\theta=0$. In a final process, the proximal augmented Lagrangian method is invoked again in order to improve the quality of the obtained  solution as well as the associated Lagrange multiplier. The authors showed that this hybrid scheme possesses the same iteration-complexity as that of the penalty method studied in \cite{PPmetNonconvex2019} to obtain a stationary point to \eqref{optl0}. 

We now discuss three other alternative works dealing with complexity
of first-order algorithms for solving \eqref{optl0}.
The first one \cite{SZhang-Pen-admm} 
presents a
% penalty based approach
% which uses the ADMM
% methdod to solve the two-block linearly constrained subproblems.
% More specifically, \cite{SZhang-Pen-admm} proposed a
penalty ADMM approach which
introduces an artificial variable $y$ in (1) and then penalizes $y$ to obtain
the penalized problem
\begin{equation} \label{eq:penpr}
\min \left \{ f(z) + h(z) + \frac{c}2 \|y\|^2 : Ax+ y =b \right \},
\end{equation}
which is then solved by a two-block  ADMM. Since \eqref{eq:penpr} satisfies the assumption that
its $y$-block objective function component has Lipschitz continuous gradient everywhere and its $y$-block coefficient matrix
is the identity, an iteration-complexity of the two-block ADMM for solving \eqref{eq:penpr},
and hence \eqref{optl0},  can be established. More specifically, it has been shown in Remark 4.3 of \cite{SZhang-Pen-admm}  that the overall number of composite gradient steps performed by the aforementioned two-block ADMM penalty scheme to obtain a triple $(\hat z,\hat v,\hat p)$
satisfying \eqref{eq:approx_stationary1} is bounded by ${\cal O}(\hat \rho^{-6})$
under the assumptions that $\hat \eta = \hat \rho$, the level sets of $f+h$ are bounded and the initial triple
$(z_0,y_0,p_0)$ is such that $(y_0, p_0)=(0, 0)$ and $z_0$ is feasible.

The second one \cite{HybridPenaltyAugLag19} studies a hybrid penalty based and AL based
method whose penalty iterations are the ones which guarantee
its convergence and whose  AL iterations are included with the purpose
of improving its computational efficiency. More specifically, 
the latter ones are performed:
1) at the initial stage of the method in order to
provide an initial value for the
constant multiplier used during the penalty phase;
2) at the final stage of the method in order
to provide a better final multiplier estimate.

 Finally, the third one
 \cite{Lan-ConstrainedStocasticProxMetNonconvex2019} studies
 a  primal-dual proximal point type method for computing approximate stationary solution to a  constrained smooth  nonconvex composite optimization problem and  establishes its iteration-complexity bounds under different sets
 of assumptions.

\vspace{5mm}
{\it Organization of the paper.} 
Subsection~\ref{sec:bas} provides some basic definitions and notation. Section~\ref{modifiedProxAugLag} contains two subsections. The first one presents our main problem of interest and the assumptions made on it. It also presents a  refinement   procedure used in the $\theta$-IPAAL method. The second subsection reviews an ACG variant which will be used to approximately solve $\theta$-prox-AL subproblems  of the $\theta$-IPAAL method. Section~\ref{sec-algorithms} contains two subsections. The first one states  the static $\theta$-IPAAL method as well as its iteration-complexity bounds. The second  subsection is devoted to the study of a dynamic variant of the  static $\theta$-IPAAL method. The  iteration-complexity analysis of this dynamic scheme is also presented in this subsection. 
Section~\ref{sec:Technical-Results}	 contains the proofs of two main results of this paper, namely, Theorems~\ref{mainprop1} and \ref{mainprop2}. It is divided into two subsections. The proof of Theorem \ref{mainprop1} is given in the first subsection while the one of Theorem \ref{mainprop2} is given in the second subsection.
Section~\ref{Sec: proofAuxlemma} contains  the proof of an auxiliary technical result.  Secion~\ref{sec:numerical} is devoted to  some preliminary numerical results. Section~\ref{sec:ConcludingRemarks} presents  some  concluding remarks. Finally, some basic auxiliary results are considered in Appendix.

	\subsection{Notation and basic definitions}
	\label{sec:bas}
	This subsection presents   notation and basic definitions used in this paper.

	Let $\Re^n$ denote the $n$-dimensional Euclidean space with inner product and associated norm denoted by $\inner{\cdot}{\cdot}$ and $\|\cdot\|$, respectively. We use $\Re^{l\times n}$ to denote the set
	of all $l\times n$ matrices.  
	The image space of a matrix $Q\in \Re^{l \times n}$ is defined as  ${\rm Im}(Q):=\{Qx: x \in \Re^n\}$
	and $\cP_Q$ denotes the Euclidean projection onto $\mbox{Im}\, (Q)$.    The smallest positive eigenvalue of $(Q^*Q)^{1/2}$ is denoted by $\sigma^+(Q)$.
	If $Q$ is a symmetric and positive semidefinite matrix, the seminorm induced by $Q$ on $\Re^n$, denoted by $\|\cdot\|_{Q}$, 
	is defined as $\|\cdot\|_{Q}:= \langle Q (\cdot), \cdot\rangle ^{1/2}$. The distance of a point $x$ to a closed convex set $X$ is denoted by $\dist_X(x)$. For $t>0$, define $\log_1^+(t):=\max\{\log t, 1\}$.
	
	The domain of a function $h :\Re^n\to (-\infty,\infty]$ is the set   $\dom h := \{x\in \Re^n : h(x) < +\infty\}$.
	Moreover, $h$ is said to be proper if  $h(x) < \infty$ for some $x \in \Re^n$. The set of closed proper convex functions defined in $\Re^n$ is denoted by $\bConv{n}$. The $\varepsilon$-subdifferential of a function $h\in \bConv{n}$ is defined by 
	\begin{equation}\label{def:epsSubdiff}
	\partial_\varepsilon h(z):=\{u\in \Re^n: h(z')\geq h(z)+\inner{u}{z'-z}-\varepsilon, \quad \forall z' \in \Re^n\}
	\end{equation}
	for every $z\in \Re^n$.	For a scalar $\alpha\in \Re$ and a function $h$, we define the  sublevel set
	\[L_h(\alpha):=\{z\in \Re^n:  h(z)\leq \alpha\}.
	\]
	If $\psi:\Re^n\to \Re$ is differentiable at $\bar z \in \Re^n$, then its affine   approximation $\ell_\psi(\cdot;\bar z)$ at $\bar z$ is defined as
	\begin{equation}\label{eq:defell}
	\ell_\psi(z;\bar z) :=  \psi(\bar z) + \inner{\nabla \psi(\bar z)}{z-\bar z} \quad \forall  z \in \Re^n.
	\end{equation}

	\section{Problem of interest and background materials}\label{modifiedProxAugLag}
	This section contains two subsections. The first one presents our main problem of interest and the assumptions made on it.
	It also presents a procedure to refine an approximate
	solution  $(z_k,v_k)$ of \eqref{exactzk0} to a pair $(\hat z_k,\hat v_k)$ which together with an approximate Lagrange multiplier $\hat\lm_k$ is such that the triple $(\hat z,\hat v,\hat\lm):=(\hat z_k,\hat v_k,\hat\lm_k)$ 
	satisfies the inclusion in \eqref{eq:approx_stationary1}.
	The second  subsection reviews an accelerated composite gradient variant which will be used
	to approximately solve the subproblems generated by the $\theta$-IPAAL method.

	\subsection{Problem of interest and refinement procedure }\label{subsec:assump-Motivation}
	This subsection formally states  our problem of interest as well as the main assumptions and  the concept of approximate stationary point to it. 
	It also contains a refinement procedure which will be used in the $\theta$-IPAAL method.
	
	The main problem of interest in this paper is \eqref{optl0}
	where  $f, h:  \Re^{n} \to (-\infty,\infty]$,
	${A} : \Re^n \to \Re^l$ and $b \in \Re^l$ satisfy the following assumptions:
	\begin{itemize}
		\item[{\bf(A1)}] $A$ is a nonzero linear operator and the feasible set $\mathcal{F}:=\{z\in \dom h:Az=b\} \ne \emptyset$;
		\item[{\bf(A2)}] $h$ is a proper convex
		lower semi-continuous function;
		\item[{\bf(A3)}]
		$f$ is nonconvex and differentiable on $\dom h$, and there exist $L\geq m > 0$ such that for every $z,z' \in \dom h$,
		\begin{align}\label{gradLips}
		\| \nabla f(\z') -  \nabla f(\z) \| \le L \|z'-z\|, \\
		% \quad \forall z,z' \in \dom h, \\
		%			\item[{\bf(A3)}]
		%			there exists $m \ge 0$ such that the function $f(\cdot)+ m \|\cdot\|^2/2$ is convex on $\dom h$, or equivalently,
		f(z') -\ell_f(z';z) \ge - \frac{m}2 \|z'-z \|^2; \label{lowerCurvature-m}%\quad \forall z,z' \in \dom h;
		\end{align}
		\item[{\bf(A4)}] 
		there exists $\bar{\pen} \ge 0$  such that 	$\phi^*_{\bar c}:=\inf_{z\in \Re^n} \phi_{\bar c}(z)>-\infty$, where
		$\phi_c$ is defined as
		\begin{equation}\label{def:phic}
		\phi_c(\cdot) := \phi(\cdot) + \frac{c}2 \|A\cdot-b\|^2 \quad \forall c \ge 0.
		\end{equation} 
	\end{itemize}

	Some comments are in order. First, it can be easily shown that the condition in \eqref{gradLips} implies that 
	$L\|z'-z \|^2/2 \geq f(z') -\ell_f(z';z) \ge -  L\|z'-z \|^2/2$ for every $z,z' \in \dom h$, and hence that \eqref{lowerCurvature-m} holds with $m=L$. However our analysis covers the case in which \eqref{lowerCurvature-m}  holds with a scalar $m< L$ in order to obtain
	better iteration-complexity bounds. Second,  \eqref{lowerCurvature-m} implies that the function $f(\cdot)+ m \|\cdot\|^2/2$ is convex on $\dom h$. Moreover, since $f$ is nonconvex on $\dom h$, we have that the smallest $m$ satisfying \eqref{lowerCurvature-m} is positive. Third, {\bf (A3)} implies that $\dom h \subseteq \dom f$,
	and hence that $\dom h = \dom \phi$.
	Fourth, {\bf (A4)} is used to obtain our iteration-complexity bounds. It trivially holds if $\phi$ is bounded below, which is always the case when $\dom h$ is bounded and $f$ is lower semi-continuous on $\cl( \dom h)$. Fifth, it is easy to see that  \eqref{lowerCurvature-m} implies that ${\cal L}^\theta_\pen(\cdot,\lm)$  is strongly convex for every $\lam<1/m$. 
	%This fact is essential in order to apply the ACG variant of Subsection~\ref{sec-nesterov} to compute an inexact solution of \eqref{exactzk0}.?????????Fifth, \eqref{gradLips} and \eqref{lowerCurvature-m} are satisfied  with $f={\cal L}^\theta_\pen(\cdot,\lm)$ and $(L,m)=(L+c\|A\|^2,m)$, and hence it follows  that the objective function of \eqref{exactzk0} satisfies \eqref{gradLips} and
	%\eqref{lowerCurvature-m} with $(L,m)=(1+ \lam[L+c\|A\|^2], \lam m-1)$, and, 
	%as a consequence, is strongly convex for every $\lam<1/m$. This fact is essential in order to apply the ACG variant of Subsection~\ref{sec-nesterov} to compute an inexact solution of \eqref{exactzk0}.
	%Hence, if $\lambda$ is sufficiently small, i.e., $\lam<1/m$,  problem \eqref{exactzk0} has a unique global minimum $z_k$ which
%in view of the stationary condition and \eqref{exactpk0} satisfies
Finally, for the sake of future reference, we note that the definitions of ${\cal F}$  and $\phi^*_{\bar c}$
	immediately imply that
	\begin{equation}\label{R-lam>=0}
	\phi^*_{\bar c} \le \phi^*=\inf_{z\in {\cal F} }  \phi(z).
	\end{equation}

	It is well known that, under some mild conditions, if $\bar{z}$ is a local minimum of \eqref{optl0}, then there exists
	$\bar{\lm} \in \Re^l$ such that $ (\bar{z}, \bar{\lm})$ is a stationary point of \eqref{optl0}, i.e., 
	\begin{equation}\label{estacionarypoint}
	0\in  \nabla f(\bar{z})+\partial h(\bar{z})+A^*\bar{\lm},
	\quad A\bar{z}-b=0.
	\end{equation}
	The main complexity results of this paper are stated in terms of the following notion of approximate stationary point
	which is a natural relaxation of \eqref{estacionarypoint}.
	
	\begin{definition}\label{def:stationarypoint}
		Given a tolerance pair  $(\hat\rho,\hat\eta)\in \Re_{++}\times\Re_{++}  $, a triple $(\hat z, \hat v,\hat\lm)\in \Re^n\times\Re^n\times\Re^l$ is said to be a $(\hat\rho,\hat\eta)$-approximate stationary point of \eqref{optl0} if it satisfies \eqref{eq:approx_stationary1}.
% 		\begin{align}\label{eq:approx_stationary1}
% 		\hat v\in\nabla f(\hat z)&+\partial h(\hat z)+A^*\hat \lm,\qquad\|\hat v\|\leq\hat{\rho}\\[3mm]
% 		&\|A\hat z-b\|\leq \hat \eta \label{eq:approx_stationary2}.
% 		\end{align}
	\end{definition}

	%In  Section~\ref{subsec:dynamicIPAAL} will  for computing $(\hat\rho,\hat\eta)$-approximate stationary point of \eqref{optl0}, which is based on
	%$\theta$-augmented Lagrangian ($\theta$-AL) function ${\cal L}^\theta_c(z;p)$ defined in \eqref{lagrangian2}.
	%We have also observed that
	%when $\theta=0$, ${\cal L}^\theta_\pen(\cdot,\cdot)$ reduces to the well known quadratic augmented Lagrangian function, and
	%when $\theta=1$, ${\cal L}^\theta_\pen(\cdot,\cdot)$ does not depend on $\lm$ and reduces to the quadratic penalty function $\phi_c$ defined in \eqref{def:phic}.
	
	Subsection~\ref{subsec:dynamicIPAAL} formally describes the  $\theta$-IPAAL method for finding a
	$(\hat\rho,\hat\eta)$-approximate stationary point of \eqref{optl0}. Each one of its outer iteration generates
	a triple $(z_k,v_k,\lm_k)$ satisfying the criterion \eqref{GIPPinc_ineq} and then
    updates the multiplier according to \eqref{def:multiplier}. The method
    computes the latter triple by applying 
    the ACG variant of Subsection~\ref{sec-nesterov} to the subproblem started
    from $z_{k-1}$ and stops when  \eqref{GIPPinc_ineq} is satisfied.
    
    In what follows, we discuss how the quadruple
    $(z_k,v_k,\lm_k,\varepsilon_k)$ obtained above can be refined to a suitable triple
    $(\hat z_k,\hat v_k,\hat \lm_k)$ satisfying the inclusion
    $\hat v_k \in \nabla f(\hat z_k)+\partial h(\hat z_k)+A^*\hat\lm_k$,
    %in \eqref{eq:approx_stationary1},
    which will later be shown
    in Subsections~\ref{subsec-mainAlg} and  \ref{subsec:dynamicIPAAL}
    to possess the property that
    $\|\hat v_k\|$  
    converges to zero with a well-established convergence rate
    while the feasibility gap $\|A\hat z_k-b\|$ is shown
    to be either
    ${\cal O}(1/\sqrt{c})$ (see Theorem~\ref{mainprop1}) or
    ${\cal O}(1/c)$ (see Theorem~\ref{mainprop2}).
    Hence, if $c$ is chosen sufficiently large,
    it follows that $(\hat z,\hat v,\hat \lm)=(\hat z_k,\hat v_k,\hat \lm_k)$
    will eventually satisfy \eqref{eq:approx_stationary1}, or equivalently,
    be a $(\hat\rho,\hat\eta)$-approximate solution of \eqref{optl0}.
    
    First, we mention that $\theta$-IPAAL assumes that
    $\lam<1/m$. In view of \eqref{lowerCurvature-m}, this implies that
    \eqref{exactzk0} is a strongly convex subproblem,
    and hence has a unique optimal solution.
    Consider first the case
    in which $z_k$ is the exact solution of
    \eqref{exactzk0}. Indeed, in this case, it follows from
    \eqref{exactpk0} and the
    optimality condition of \eqref{exactzk0} that 
    $0 \in \nabla f(z_k)+\partial h(z_k) + A^*\lm_k+ (z_k-z_{k-1})/\lam$,
    and hence that $(\hat z_k, \hat v_k,\hat \lm_k):=(z_k,(z_{k-1}-z_k)/\lam,\lm_k)$
    satisfies the above inclusion.
    Assume now that $z_k$ is an approximate solution of \eqref{exactzk0}
    in the sense that there exists a residual pair
    $(v_k,\varepsilon_k)$ which together with $z_k$ satisfies
    the approximation criterion \eqref{GIPPinc_ineq} below.
    In this case, it is shown 
    in Proposition~\ref{prop:approxsol}
    below that the following refinement procedure
    with inputs  $(g,h)=({\cal L}^\theta_c(\cdot;\lm_{k-1})-h,h)$  and $(\lam,z^-,z,v)=(\lam,z_{k-1},z_k,v_k)$ obtains a pair    $(\hat z_k, \hat v_k)$ which together with
    $\hat\lm_k$ as in \eqref{def:wkpkhat} satisfies the above
    inclusion and such that
    $(\|\hat v_k\|, \|\hat z_k - z_k\|)$ is conveniently
    bounded by $\|v_k+z_{k-1}-z_k\|$ and $\varepsilon_k$.
    Since $\varepsilon_k$ is bounded by $\|v_k+z_{k-1}-z_k\|$
    in view of the inequality in \eqref{GIPPinc_ineq}, and
    it follows from the second inequality in \eqref{ineqs:Feas-rk-pk} combined with \eqref{def:Delta}  that
    $\|v_k+z_{k-1}-z_k\|$  approaches
    zero, we will then be able to conclude that
    $\|\hat v_k\|$, as well as
    $\|\hat z_k - z_k\|$, approaches zero.

	\vspace{2mm}
	\noindent\rule[0.5ex]{1\columnwidth}{1pt}
	
	\textbf{Refinement Procedure}
	
	\noindent\rule[0.5ex]{1\columnwidth}{1pt}
	
		\noindent {\bf Input:} A pair of functions $(g,h)$ such that  $h \in \bConv{n}$,
		and $g$ is differentiable on $\dom h$ and its gradient  is $M$-Lipschitz continuous, and a quadruple $(\lam,z^-,z,v) \in \Re_{++}\times\Re^n\times \dom h\times\Re^n$. 
	\begin{itemize}
		\item[(1)] set % satisfying 
		\begin{equation}\label{eq:def_f}
		g_\lam:=\lambda g +\frac{1}{2} \|\cdot-z^-\|^2-\langle v, \cdot \rangle, \quad h_{\lam} := \lam h;
		\end{equation}
		\item[(2)]
		compute
		\begin{align}
		&\hat  z := \argmin_u \left\{ \inner{\nabla g_\lam(z)}{u-z} + \frac{\lambda M+1}2 \|u-z\|^2 + h_\lam(u) \right \}, \label{eq:z_pRefProc} \\
		& \hat v:= \frac{1}\lam \left[( v+ z^--z ) +(\lambda M+1)(z-\hat z) \right] +  \nabla g(\hat z)-\nabla g(z), \label{eq:ref_vp} \\
		& \Delta:= (g_\lam+h_\lam)(z) - (g_\lam+h_\lam)(\hat z); \label{eq:ref_var}
		\end{align}
	\end{itemize}
		\noindent {\bf Output:}  $(\hat z,\hat v, \Delta) \in \dom h \times \Re^n \times \Re_{++}$.
	
	\vspace{2mm}
\noindent\rule[0.5ex]{1\columnwidth}{1pt}

	We now state a proposition summarizing some  important properties of the above procedure.
	
	\begin{proposition}\label{prop:approxsol}
		Under the assumptions stated at the beginning of the above refinement procedure,
		the following statements about its output $(\hat z,\hat v,\Delta)$ hold:
		\begin{itemize}
			\item[a)] $\Delta \geq 0$ and 
			\begin{equation}\label{eq:inclv'}
			\hat v \in \nabla g(\hat z) + \partial h (\hat z),\quad \lam \|\hat v\|\leq \|v +z^--z\| +2\sqrt{2 (\lambda M+1) \Delta},\qquad \|\hat z-z\| \leq \sqrt{2(\lambda M+1)^{-1} \Delta};
			\end{equation}
			\item[b)] if there exists $\varepsilon\geq 0$ such that the input  $(\lambda,z^-,z,v)$  satisfies
			\begin{equation} \label{incl:proc}
			v \in \partial_\varepsilon \left(\lambda(g+ h) + \frac{1}{2}\|\cdot-z^-\|^2 \right)(z),  
			\end{equation}
			then $\Delta \leq \varepsilon.$
			
		\end{itemize}
	\end{proposition}
	\begin{proof}
	The proof of this proposition can be found, for instance, in \cite[Proposition~2.1 and Lemma~2.3]{WJRComputQPAIPP}. 
	\end{proof}
	The above proposition shows that: 1) the pair $(\hat z, \hat v)$, computed as in \eqref{eq:z_pRefProc} and \eqref{eq:ref_vp},  satisfies an inclusion closely related to  \eqref{eq:approx_stationary1}, in view of the definition $\hat \lm$ given in Step~3 of the $\theta$-IPAAL method of Subsection~\ref{subsec-mainAlg}; 2) the quantities
	$\lambda \|\hat v\|$ and $\|\hat z-z\|$ have upper  bounds  expressed in terms of  $\lambda M$ and the two quantities: $\|v+z^- -z\|$ and $\sqrt{\Delta}$. Finally, the ACG method of Subsection~\ref{sec-nesterov} will be invoked in Step~1 of the $\theta$-IPAAL method to compute a triple $(\varepsilon,z,v)$ which  together with a previously computed pair $(\lambda,z^-)$  satisfy the inclusion in \eqref{incl:proc} as well as some extra conditions bounding, in particular, the scalar $\varepsilon$ in terms of  the quantity $\|v+z^- -z\|$. In view of the latter discussion and  Proposition~\ref{prop:approxsol} b), the scalar $\Delta$ will also be controlled by $\|v+z^- -z\|$. The latter result will be essential to establish our iteration-complexity bounds.

	\subsection{An ACG variant}\label{sec-nesterov}
	This subsection reviews the ACG variant that will be
invoked by	the $\theta$-IPAAL method  for solving  the subproblems \eqref{exactzk0} which
	arise during its implementation. It also discusses the  iteration-complexity for finding a
	certain type of approximate solutions.
	
	With the above goal in mind, we
	 consider the composite optimization problem
	(for the purpose of this subsection only)
	\begin{equation}\label{mainprob:nesterov1}
	\min \{\psi(x):=\psi_s(x)+\psi_n(x) :  x \in \Re^n\}
	\end{equation}
	where the following conditions are assumed to hold:
	\begin{itemize}
		\item [{\bf (C1)}]$\psi_n:\Re^n\rightarrow (-\infty,+\infty]$ is a proper, closed and $\mu$-strongly convex  function with $\mu \ge 0$;
		\item [{\bf (C2)}]$\psi_s$ is a convex differentiable function on $\dom \psi_n$
		and there exists $M_s >0$ satisfying
		$
		\psi_s(u)-\ell_{\psi_s}(u; x)
		%[ \psi_s(x)+\langle \nabla\psi_s(x),u-x\rangle]
		\leq  M_s\|u-x\|^2/2$ for every $ x, u \in \dom \psi_n$
		where $\ell_{\psi_s}(\cdot \,;\cdot)$ is defined in \eqref{eq:defell}.
	\end{itemize}
	
	We now state the aforementioned ACG variant %(\cite{Attouch2016,YHe2,nesterov2012gradient,nesterov1983method,tseng2008accmet})
	for solving \eqref{mainprob:nesterov1}.
	We remark that other ACG variants such as the ones
	in \cite{Attouch2016,YHe2,nesterov2012gradient,nesterov1983method,tseng2008accmet}
	could also have been used in the development of
	the $\theta$-IPAAL method.
	
	\vgap

	\noindent\rule[0.5ex]{1\columnwidth}{1pt}
	
	\textbf{ACG Method}
	
	\noindent\rule[0.5ex]{1\columnwidth}{1pt}
	\begin{itemize}
		\item[(0)]  Let function and parameter pairs $(\psi_s,\psi_n)$ and $(M_s,\mu)$ satisfying assumptions {\bf (C1)} and {\bf (C2)}
		and initial point $x_{0}\in \dom  \psi_n $ be given, and set $y_{0}=x_{0}$, $A_{0}=0$, $\Gamma_0\equiv0$ and $j=0$; 
		\item[(1)]  compute
		\begin{align*}
		A_{j+1}  &=A_{j}+\frac{\mu A_{j}+1+\sqrt{(\mu A_{j}+1)^2+4M_s(\mu A_{j}+1)A_{j}}}{2M_s},\\
		\tilde{x}_{j}  &=\frac{A_{j}}{A_{j+1}}x_{j}+\frac{A_{j+1}-A_{j}}{A_{j+1}}y_{j},\quad\Gamma_{j+1}=\frac{A_{j}}{A_{j+1}}\Gamma_j+\frac{A_{j+1}-A_{j}}{A_{j+1}}\ell_{\psi_s}(\cdot;\tilde x_j),\\
		y_{j+1} &=\argmin_{y}\left\{ \Gamma_{j+1}(y)+\psi_n(y)+\frac{1}{2A_{j+1}}\|y-y_{0}\|^{2}\right\},\\
	x_{j+1} & =\frac{A_{j}}{A_{j+1}}x_{j}+\frac{A_{j+1}-A_{j}}{A_{j+1}}y_{j+1};
		\end{align*}
		\item[(2)] compute 
		\begin{align*}
		u_{j+1}&=\frac{y_0-y_{j+1}}{A_{j+1}},\\[2mm]
		\eta_{j+1}&= \psi(x_{j+1})-\Gamma_{j+1}(y_{j+1})- \psi_n(y_{j+1})-\langle u_{j+1},x_{j+1}-y_{j+1}\rangle;   
		\end{align*}
		\item[(3)]  set $j\leftarrow j+1$ and go to (1). 
	\end{itemize}
	\noindent\rule[0.5ex]{1\columnwidth}{1pt}
	
	Some remarks about the ACG method follow. First, the main core and usually the common way of describing an iteration of the ACG method is as in step~1.
	Second, the
	extra sequences $\{u_j\}$ and $\{\eta_j\}$ computed  in step~2 will be used to develop a stopping criterion for the ACG method when it 	is called as a subroutine for solving the subproblems of the $\theta$-IPAAL method in Subsection~\ref{subsec-mainAlg}.
	Third, the ACG method in which  $\mu=0$  is a special case of a slightly more general one studied by Tseng in \cite{tseng2008accmet} (see Algorithm~3
	of \cite{tseng2008accmet}).
	The analysis of the general case of the ACG method in which $\mu\geq0$ was studied  in  \cite[Proposition~2.3]{YHe2}. 
	The sequence $\{A_k\}$ has the following increasing property
	\begin{equation}
	A_{j}\geq\frac{1}{M_s}\max\left\{\frac{j^{2}}{4},\left(1+\sqrt{\frac{\mu}{4M_s}}\right)^{2(j-1)}\right\}. \label{ineq:increasingA_k}
	\end{equation}

	The next proposition summarizes the main properties
	of the ACG method that will be needed in our analysis.
	
	\begin{proposition}\label{lem:nest_complex} 
		Let $\{(A_j,x_j, u_j,\eta_j)\}$ be the sequence generated by the ACG method applied to \eqref{mainprob:nesterov1},
		where $(\psi_s,\psi_n)$ is a given pair of  functions  satisfying {\bf (C1)} and {\bf (C2)} with $4 M_s \ge \mu> 0$.
		Then, the following statements hold: 
		\begin{itemize}
		    \item [a)] for every $j\geq 1$, we have 
		    $u_j\in  \partial_{\eta_j}(\psi_s+\psi_n)(x_j)$;
		   \item[b)] for any $\sigma>0$, the ACG method obtains  a triple $(x,u,\eta)=(x_j,u_j,\eta_j)$ satisfying  
		   \begin{equation}\label{mainprob:nesterov}
		u\in  \partial_{\eta}(\psi_s+\psi_n)(x) \quad \|u\|^{2}+2\eta\le\sigma^2\|x_{0}-x+u\|^{2}
		\end{equation}
		in at most 
		\begin{equation}\label{boundACG}
		\left\lceil1+\sqrt{\frac{M_s}{\mu}}\log^+_1\left( (1+\sigma^{-1})\sqrt{2M_s}\right)\right\rceil 
		\end{equation}
		iterations.
		\end{itemize}
	\end{proposition}
	\begin{proof} The inclusion in a) follows immediately from \cite[Proposition~8(c)]{WJRproxmet1}. The proof of b) follows from the first statement in \cite[Lemma~9]{WJRproxmet1} and by noting that if  $j$ is larger than or equal to the number in \eqref{boundACG}, then $A_j\geq 2(1+\sigma^{-1})^2$, in view of \eqref{ineq:increasingA_k}. It should be noted that the parameter $\sigma$ in \eqref{mainprob:nesterov} corresponds to 
	$\sqrt{\sigma}$ in \cite[Proposition~8 and Lemma~9]{WJRproxmet1}.\end{proof}
	
	As we have already mentioned earlier, the ACG method
	of this subsection will be used to approximately
	solve the subproblems \eqref{exactzk0} during the course
	of the $\theta$-IPAAL method.
	 The results of
	this subsection will then be invoked 
	with $\psi$ chosen as the objective
	function of \eqref{exactzk0} and condition
	\eqref{mainprob:nesterov} on $(x,u,\eta)$ will then be used 
	to declare $z_k=x$ as an acceptable
	approximate solution of \eqref{exactzk0}.

	\section{The $\theta$-IPAAL method and main complexity results}\label{sec-algorithms}
	This section contains two subsections. The first one states  the static $\theta$-IPAAL method as well as its iteration-complexity bounds.
	The second  subsection presents  a dynamic variant of the  
	static $\theta$-IPAAL method and its  iteration-complexity analysis.
	
	\subsection{The static $\theta$-IPAAL method and its iteration-complexity}\label{subsec-mainAlg}
	This subsection presents the static $\theta$-IPAAL method and its iteration-complexity bounds.
    
    The static $\theta$-IPAAL method assumes that the parameter
    $\theta \in (0,1]$ is specified by the user and then starts
    by computing the following two intermediate parameters $\tau_\theta$ and $\sigma_\theta$ uniquely
    determined by $\theta$:
    \begin{equation}\label{def:tau}
            \tau_\theta:=\left\{\begin{array}{cc} \frac{\theta}{16-17\theta},&\mbox{if $\theta\leq \frac{16}{19}$}, \\[3mm]
            \frac{1}{2}, & \mbox{otherwise},
            \end{array} \right.                   
            \end{equation}
            which is used to define the prox-stepsize $\lambda$ that appears in the
    prox $\theta$-AL subproblems \eqref{exactzk0}, and
    $\sigma_\theta$ defined as the
    the only positive
	root of the second-order equation
	\begin{equation}\label{def:sigmatheta}
	\left(\frac{3}{4} +\frac{2(1-\theta)\left(3\tau_\theta+1\right)}{\theta\tau_\theta}\right) \sigma^2 + \left( \frac{8-7\theta}{2\theta}\right)\sigma  -\frac{1}{8}= 0.
	\end{equation}
	For the sake of future reference, we note that
	$\sigma_\theta \le 1/2$.

%	$$
%	3\sigma^2/4+\sigma/2-1/8=0
%	$$
%	$$
%\frac{\sigma=(-0.5+\sqrt{1/4+3/8})}{6/4}
%	$$
	
    % The latter one is used to measure the inexactness of $z_k$
    % as an approximate solution of subproblem \eqref{exactzk0}.
    % $\sigma$ in \eqref{definition of sigma}  which controls the inexactness of the approximate solution computed by the ACG variant of Section~\ref{sec-nesterov}. 
%    The aforementioned parameters are  defined as follow:
% 	\begin{equation}\label{def:tau}
%             \tau_\theta:=\left\{\begin{array}{cc} \frac{\theta}{16-17\theta},&\mbox{if $\theta<\frac{16}{17}$} \\[3mm]
%             1, & \mbox{otherwise};
%             \end{array} \right.                   
%             \end{equation}

% 	\begin{equation}\label{def:sigmatheta}
% \sigma_\theta:=\left[- \frac{8-7\theta}{2\theta} + \sqrt{\left( \frac{8-7\theta}{2\theta}\right)^2+ \frac{1}{2}\left[\frac{3}{4} + \frac{2(1-\theta)}{\theta}\left( 2+ \frac{\tau_\theta+1}{\tau_\theta}\right)\right]}\right]{\left[\frac{3}{2} + \frac{4(1-\theta)}{\theta}\left( 2+ \frac{\tau_\theta+1}{\tau_\theta}\right)\right]^{-1}}    
% \end{equation}

% 	We remark that $\sigma_\theta$ corresponds to  the only positive
% 	root of the second-order equation
% 	\begin{equation}\label{def:sigmatheta}
% 	\left[\frac{3}{4} + \frac{2(1-\theta)}{\theta}\left( 2+ \frac{\tau_\theta+1}{\tau_\theta}\right)\right] \sigma^2 + \left( \frac{8-7\theta}{2\theta}\right)\sigma  -\frac{1}{8}= 0.
% 	\end{equation}
	
We now state the $\theta$-IPAAL Method.
\vspace{4mm}	
	\hrule
\vspace{3mm}
	{\bf Static $\theta$-IPAAL Method}
	\vspace{2mm}
	\hrule
	\begin{itemize}
		\item[(0)] Let parameter $\theta \in (0,1]$, initial point $z_0 \in \dom h$,
		scalar $\bar c$ satisfying {\bf (A4)},
		tolerance $\hat \rho >0$, and penalty parameter $c \in \Re$ satisfying
		
		\begin{equation}\label{assump:penaltyparameter}
		\pen > 2\bar \pen
		\end{equation}
     be given, and set $p_0=0$, $k=1$ and
		\begin{equation}\label{definition of sigma}
		L_c:=L+c\|A\|^2,\qquad  \lambda = \frac{\tau_\theta}m,\qquad   \sigma:= \min\left\{\frac{1}{\sqrt{\lambda L_c + 1}}, \sigma_\theta\right\};
		\end{equation}
		
		%\begin{equation}\label{definition of sigma}
		%\lambda = \frac{\tau_\theta}m, \quad \sigma=\min \left\{ %\sigma_\theta, \frac{1}{ \sqrt{\lam(L + c \|A\|^2) + 1  } } %\right\};
		%\end{equation}
		
		\item [(1)] let
		\begin{equation}\label{def:ghtilde}
		g_k := f + (1-\theta) \Inner{\lm_{k-1}}{A\cdot -b}+\frac{\pen}{2}\|A\cdot -b\|^2
		\end{equation}
		and apply the ACG method with inputs
		\begin{equation}\label{def:psi-s e psi-n}
		\psi_s := 	\lam g_k + \frac{\tau_\theta}2 \|\cdot-\z_{k-1}\|^2, \quad \psi_n := \lam h +\frac{1-\tau_\theta}{2}\|\cdot-z_{k-1}\|^2,
		    \end{equation}
		\begin{equation}\label{def:Ms-muIPAAL}
		(M_s,\mu)= (\lam L_c+ \tau_\theta, 1-\tau_\theta), \quad x_0=z_{k-1}
		\end{equation}
		
		to obtain a triple $(z_k,v_k,\varepsilon_k)\in \Re^n\times\Re^n\times\Re_{+}$ satisfying 
		
		\begin{equation}\label{GIPPinc_ineq}
		v_k \in \partial_{\varepsilon_k} \left ( \lam {\cal L}^\theta_c(\cdot;\lm_{k-1}) + \frac12 \|\cdot-\z_{k-1}\|^2 \right) (\z_k),
		\quad \| v_k\|^2 + 2 \varepsilon_k \le \sigma^2 \| v_k + z_{k-1}- z_k\|^2;
		\end{equation}

		\item[(2)] compute a pair $(\hat z_k,\hat v_k)$ via  the refinement procedure with input $(g, h) =(g_k, h)$ and
		$(\lam,z^-,z, v) = (\lam,z_{k-1},z_k,v_k)$
		and set 
		\begin{equation}\label{def:wkpkhat}
		\hat \lm_k = (1-\theta)\lm_{k-1}+\pen\left(A\hat z_k-b\right);
		\end{equation}

		\item[(3)] if $\|\hat v_k \| \le \hat \rho$,
		then output $(\hat z,\hat v,\hat \lm):=(\hat z_k,\hat v_k,\hat \lm_k)$ and {\bf stop}; otherwise,  set
		\beq \label{def:multiplier}
		\lm_k = (1-\theta)\lm_{k-1}+\pen\left(Az_k-b\right)
		\eeq
		and  $k \leftarrow k+1$ and return to (1);
	\end{itemize}
	
	\hrule
	\vgap

	We now make a few remarks about  the static $\theta$-IPAAL method.
	First, it makes two types of iterations, namely,
	the outer ones indexed by $k$ and the ACG ones performed
	during its calls to the ACG method in step~1.
	Second, upon termination the method finds
	a pair satisfying the inclusion and the first inequality in \eqref{eq:approx_stationary1} but not necessarily the second inequality in \eqref{eq:approx_stationary1}. However, it is shown in Theorems~\ref{mainprop1} and \ref{mainprop2} below that the feasibility gap
	$\|A \hat z_k-b\|$
	of $\hat z_k$
	is either ${\cal O}(1/\sqrt{c})$ or ${\cal O}(1/c)$ depending on the
	assumptions made on problem \eqref{optl0}.
	Hence,  the static $\theta$-IPAAL method should not
	be viewed as an algorithm for
	solving \eqref{optl0} but rather an intermediary step in this
	direction.
	Indeed, its dynamic version described
	in Subsection \ref{subsec:dynamicIPAAL}, which:
	a) doubles $c$ whenever the aforementioned feasibility gap
	at the end of the static $\theta$-IPAAL method is not small enough, and; b) repeats the static $\theta$-IPAAL method with the updated $c$,
	is guaranteed to obtain an
	approximate solution of \eqref{optl0} in the sense of Definition~\ref{def:stationarypoint}.
% 	Third, condition \eqref{assump:penaltyparameter} is used mainly to  simplify the complexity bounds for the static $\theta$-IPAAL method, indeed it can be easily seen that the complexity bounds established for the  $\theta$-IPAAL method hold for any $c>\bar{c}$.
	%The conditions on  $\tau_\theta,\sigma_\theta,$ and $\sigma$ in  \eqref{def:tau}, \eqref{def:sigmatheta}, and   \eqref{definition of sigma}, respectively, are used to ensure  that the scalar $C_\theta$ defined in \eqref{def:BCtheta} is strictly positive.
% 	The latter general condition is essential to establish the iteration complexity bound of the static $\theta$-IPAAL method.
	Finally, only the inequality
	in \eqref{GIPPinc_ineq} needs to be checked
	for terminating the call to the ACG method
	in step 1 since the inclusion is always satisfied by
	any ACG iterate (see Proposition~\ref{lem:nest_complex}~(a), \eqref{lagrangian2}, \eqref{def:ghtilde}, and  \eqref{def:psi-s e psi-n}). 

	Our goal now is to state an iteration-complexity result for
	the static $\theta$-IPAAL method under the assumptions
	introduced in Subsection~\ref{subsec:assump-Motivation}. Its bounds are described
	in terms of the quantity
	\begin{equation}\label{def:Rlamb}
	R_0 = R(z_0;\lam,\bar c) := \inf\left\{  \lambda [ (f+h)(z) - \phi^*_{\bar \pen} ]  + \|z-z_0\|^2: z\in {\mathcal F}\right\}
	\end{equation}
	where  ${\mathcal F}$ and $\phi^*_{\bar \pen}$ are as {\bf (A1)} and {\bf (A4)},  respectively.
	Note that 
	if \eqref{optl0} has an optimal solution,
	then we easily see by using \eqref{R-lam>=0} that
	\begin{equation}\label{ineq:majorizing-Rlambda}
    	0 \le R_0 \leq\lambda [ \phi^* - \phi^*_{\bar \pen} ]  + d_0^2,
		\end{equation}
where $\phi^*$ is as in \eqref{optl0} and
	\[
	d_0:=\inf \{\|z^*-z_0\|: z^* \; {\rm is\; an\; optimal  \; solution\; of\;     }\eqref{optl0}\}.
	\]
	Hence, under the assumption above, $R_0$ is majorized by
	a quantity expressed in terms of
	$d_0^2$ and the functional gap
	$\phi^* - \phi^*_{\bar \pen} \ge 0$.

	We are now ready to state the aforementioned iteration-complexity result for the static $\theta$-IPAAL method. Its proof will be given in Subsection~\ref{TechSubsec1}.

	\begin{theorem}\label{mainprop1}
		Let  $\tau_\theta$ and $R_0$ be  as in \eqref{def:tau}  and  \eqref{def:Rlamb}, respectively, and consider $\kappa_\theta$  as
		\begin{equation}\label{def:constante0}
		\kappa_\theta:= 
		1+\frac{16(1-\theta)}{\theta\tau_\theta}.
		\end{equation}
		Then, the following statements about the static $\theta$-IPAAL method hold: 
		\begin{itemize}
			\item[a)] its number of outer iterations
			is bounded by
			$
			\left\lceil 108\kappa_\theta R_0/(\lambda^2\hat \rho^2)\right\rceil;
			$
			\item[b)] the number of ACG iterations performed at each outer iteration is bounded by
			\begin{equation}\label{def:Thetac}
		\left\lceil 1+\sqrt{\Theta_c}\log^+_1\left(\frac{2\Theta_c}{\sigma_\theta}\right)\right\rceil
		\end{equation}
			where 	$\Theta_c:=2\lambda L_c + 1$  and  $L_c$ is as  in \eqref{definition of sigma};
			\item[c)] its output  $(\hat z,\hat v,\hat \lm)$ satisfies
			\begin{align}\label{Main_ineq-feasibility}
			\hat v\in\nabla f(\hat z)&+\partial h(\hat z)+A^*\hat \lm,\qquad\|\hat v\|\leq\hat{\rho}, \\[3mm]
			& \|A \hat z -b\| \leq \frac{4}{\lambda}\sqrt{\frac{\kappa_\theta R_0}{\pen}}. \label{eq:feasib}\end{align}
		\end{itemize}
\end{theorem}	
% 	\begin{proof}
% 		The proof  will be presented at the end of Subsection~\ref{TechSubsec1}.
% 	\end{proof}
	
	We will now state an alternative iteration-complexity
	result for the static $\theta$-IPAAL method which holds under the same assumptions
	stated in Subsection~\ref{subsec:assump-Motivation} plus a new set of mild
	conditions which include the existence of a Slater point
	for \eqref{optl0}. The main extra consequence of 
	this result is that
	the feasibility gap of $\hat z$ is now shown to
	be ${\cal O}(1/c)$ instead of ${\cal O}(1/\sqrt{c})$.
	The extra conditions assumed on problem \eqref{optl0} 
	are as follows:

	\begin{itemize}
		\item[{\bf (B1)}]
		there exists $\bar z \in \inte(\dom h)$ such that $A\bar z=b$;
		\item[{\bf (B2)}] for any $\alpha \in \Re$,  
		\[
		D(\alpha) := \sup \{ \|x-x'\| : x, x' \in L_{\phi_{\bar{\pen}}}(\alpha) \} < \infty,
		\]
		where $\bar \pen$ is as in (A4); hence, the sublevel set $L_{\phi_{\bar{\pen}}}(\alpha)$ is bounded for any $\alpha \in \Re$;
		\item[{\bf (B3)}] for any $\alpha \in \Re$, 
		\[
		S(\alpha) := \inf \left\{ s \in \Re_+ : \partial h( z) \subseteq \cball{0}{s} + N_{\dom h}(z), \ \forall z \in L_{\phi_{\bar{\pen}}}(\alpha)) \right\} < \infty.
		\]
	\end{itemize}
	
	We now make a few remarks about conditions
	{\bf (B1)}--{\bf(B3)}. First, {\bf (B1)} is the so called
	Slater condition for \eqref{optl0}. Second,
	{\bf (B2)} trivially
	holds if $\dom h$ is bounded.
	Third, the latter condition  plus
	conditions {\bf (B1)} and
	{\bf (B3)} are exactly the ones assumed
	in \cite{PPmetNonconvex2019} where the complexity
	stated for $\theta$-IPAAL method in Theorem~\ref{mainprop1} was established
	for the penalty scheme studied there.

	We are now ready to state the
	aforementioned alternative
	iteration-complexity for the  static $\theta$-IPAAL method.
Its proof will be given only
in Subsection~\ref{proof of main theo2}.
	\begin{theorem}\label{mainprop2}
	In addition to {\bf (A1)}--{\bf (A4)}, assume that conditions {\bf (B1)}--{\bf (B3)}  hold. Let $R_0$, $\kappa_\theta$, and $\bar{z}$ be as in   \eqref{def:Rlamb}, \eqref{def:constante0}, and {\bf (B1)}, respectively. 
	 Define $N_0 = N(z_0;\theta,\lambda,\bar c, \bar z)$
	 as
\begin{equation}\label{N0}
%	\small	N(z_0;\theta,\lambda,\bar c, \bar z):=
	N_0 := \frac{1}{\sigma^+(A)}
		\left( 1+ \frac{\max\{ \sqrt{3\kappa_\theta R_0} , D(\alpha)\} }{\mbox{\rm dist}_{\partial(\dom h )} (\bar z)} \right) \left[
		\|\nabla f(z_0)\| +L D(\alpha)+ S(\alpha)  + \left(L+\frac{14}{\lambda \sqrt{\theta}}\right) \sqrt{3\kappa_\theta R_0}  \right], 
		\end{equation}
		where $D(\alpha)$ and $S(\alpha)$ are as in
		 {\bf (B2)} and {\bf (B3)} with
		 $\alpha:= \kappa_\theta + \max\{\phi_{\bar{c}}(z_0),\phi_{\bar{c}}(\bar z)\}.$
			Then, the output $(\hat z,\hat v,\hat \lm)$ of the static $\theta$-IPAAL method satisfies the relations in \eqref{Main_ineq-feasibility} and 
		\begin{equation}
		\|A \hat z -b\| \leq \frac{2N_0}{c}. \label{eq:feasib3}
		\end{equation}
	\end{theorem}

		We now make a few remarks about Theorems~\ref{mainprop1} and \ref{mainprop2}. First, given a tolerance $\hat \eta$, define
		\[
		c(\hat \eta):= \max\left \{2\bar{c},\frac{16\kappa_\theta R_0}{\lambda^2\hat\eta^2}\right\}.
		\]
		Clearly, in view of \eqref{Main_ineq-feasibility} and \eqref{eq:feasib}, it follows that for any $c > c(\hat \eta)$, the output
		$(\hat z,\hat v,\hat \lm)$ of the static $\theta$-IPAAL method is a $(\hat \rho,\hat \eta)$-approximate stationary point of \eqref{optl0}
		(see Definition~\ref{def:stationarypoint}).
		Second,  under assumptions {\bf (A1)}--{\bf (A4)} and {\bf (B1)}--{\bf (B3)},   Theorem~\ref{mainprop2} implies that
		$\|A\hat z-b\| \leq \hat \eta$
		for any $c$		satisfying 
		\begin{equation}\label{secCond-penaltyc}
		c > \max\left\{2\bar{c},\frac{2 N_0}{\hat \eta}\right\},
		\end{equation}
		and hence  the output $(\hat z,\hat v,\hat \lm)$ of the static $\theta$-IPAAL method is a $(\hat\rho,\hat\eta)$-approximate stationary point of \eqref{optl0}.
		Third, since either $c(\hat \eta)$ or
		the right-hand side of \eqref{secCond-penaltyc}
		is difficult to compute
		due to its dependence on $\kappa_\theta$ or $N_0$,
		it is usually not possible to have at our immediate disposal a scalar $c$ as in the previous two remarks,
		and hence to find a $(\hat \rho,\hat \eta)$-approximate stationary point of \eqref{optl0}
		by just solving a single penalized subproblem. The next subsection presents a dynamic version of the static $\theta$-IPAAL method which dynamically
		updates the penalty parameter $c$ and whose
		overall ACG iteration-complexity is similar to
		either one
		the complexities obtained above with
		$c$ chosen as $c(\hat \eta)$ or the
		right-hand side of \eqref{secCond-penaltyc}.

	\subsection{The $\theta$-IPAAL method and its iteration-complexity}\label{subsec:dynamicIPAAL}
	This subsection presents a dynamic version of the $\theta$-IPAAL method  to obtain 
	a $(\hat \rho,\hat \eta)$-approximate stationary point to \eqref{optl0}. The method basically consists of 
	applying the static $\theta$-IPAAL method  and repeatedly
	doubling the penalty parameter until the second inequality in   \eqref{eq:approx_stationary1}
	is satisfied. The main iteration-complexity result of this scheme is also presented and proved in this subsection.

	We start by stating the dynamic version of
	the $\theta$-IPAAL method, which will be simply
	referred to the $\theta$-IPAAL method for shortness. 
	\vgap
	\hrule
	\noindent
	\\
	{\bf $\theta$-IPAAL method}
	\\
	\hrule
	\begin{itemize}
		\item[(0)] Let  an initial point $z_0 \in \dom h$, a scalar $\theta\in (0,1]$, and  a pair of tolerances $(\hat \rho, \hat \eta)\in \Re_{++}\times\Re_{++}$ be given, choose an initial  penalty parameter $\pen_1>2\bar c$ and set $c=c_1$; 
		\item[(1)] execute the static $\theta$-IPAAL method with input $(z_0,\theta,c,\hat \rho)$, and  let $(\hat z, \hat v,\hat p)$ be the output.
		\item[(2)] if $\|A \hat z-b\|\leq \hat \eta$, stop and output $(\hat z, \hat v,\hat p)$; otherwise, set $c\leftarrow 2c$ and return to step~1.  
	\end{itemize}
	\hrule
	\vgap
	
	The next result establishes the overall ACG
	iteration-complexity of
	the  $\theta$-IPAAL method for obtaining
	a  $(\hat \rho,\hat \eta)$-approximate stationary point of \eqref{optl0}.

\begin{theorem}\label{mainTheo2} Assume that conditions {\bf (A1)}--{\bf(A4)}  of
	Subsection \ref{subsec:assump-Motivation} hold. Then,
	the following statements about
	the $\theta$-IPAAL method hold:
	\begin{itemize}
		\item[a)]	it obtains an approximate stationary point $(\hat z,\hat v,\hat \lm)$ of problem~\eqref{optl0} in the sense of Definition~\ref{def:stationarypoint} in at most 
		\begin{equation}\label{bound001}
		{\mathcal O}\left(\left\lceil\frac{R_0}{\lambda^2\hat\rho^2}\right\rceil\left[\sqrt{\lambda\max\left\{ c_1,\frac{R_0}{\lambda^2\hat\eta^2}\right\}}\|A\|	
		+ \sqrt{\lambda L+1}\log^+_1\left(\frac{R_0}{c_1\lambda^2\hat\eta^2}\right)\right]\log^+_1\left(T_{\hat{\eta}}\right)\right),
		\end{equation}	
		%			$$ \mathcal{O}\left(\frac{R_0\sqrt{T_{\hat \eta}}}{\theta\lambda^2\hat\rho^2}\log^+_1\left(T_{\hat{\eta}}\right)\right)
		%			$$
		%			$$ \mathcal{O}\left(\frac{\left(m R_0 + \| \lm_0\|^2\right)\sqrt{\Theta(\hat \eta)}}{\hat\rho^2}\log\left(\Theta(\hat{\eta})+1\right)\right)
		%			$$
		ACG iterations,	where  $R_0$  is as in \eqref{def:Rlamb}, and 
		\begin{equation}\label{def:bartheta}
		T_{\hat \eta}:= \lambda\left(L +\max\left\{c_1,\frac{R_0}{\lambda^2\hat\eta^2}\right\}\|A\|^2\right)+1;		
		\end{equation}
		
		\item[b)] if, in addition, 	  the conditions  {\bf (B1)}--{\bf (B3)}
		of Subsection \ref{subsec-mainAlg} hold, then its overall
		ACG iteration-complexity is	
		\begin{equation} \label{eq:ttt1}
		{\mathcal O}\left(\left\lceil\frac{R_0}{\lambda^2\hat\rho^2}\right\rceil\left[\sqrt{\lambda\max\left\{c_1,\frac{ N_0}{\hat \eta}\right\}}\|A\|	
		+ \sqrt{\lambda L+1}\log^+_1\left(\frac{N_0}{c_1\hat\eta}\right)\right]\log^+_1\left(\tilde T_{\hat{\eta}}\right)\right),
		\end{equation}
		where
		$$
		\tilde T_{\hat\eta}:=\lambda\left(L +\max\left\{c_1,\frac{ N_0 }{\hat\eta}\right\}\|A\|^2\right)+1
		$$
		and $N_0$ is as in \eqref{N0}.
	\end{itemize}
\end{theorem}

\begin{proof}
	
	a)	First note that the $l$-th loop of the $\theta$-IPAAL method invokes the static $\theta$-IPAAL method with  penalty parameter $c=c_l$ where $c_l :=2^{l-1}c_1> 2\bar{c}$ for all $l \ge 1$.
	Hence, in view  of the stopping criterion in step~2 of the $\theta$-IPAAL method  and Theorem~\ref{mainprop1} c),  we conclude that the  $\theta$-IPAAL method  performs at most $\bar{l}$ iterations  where  
	\begin{equation}\label{ineq:cl}	
	\bar{l}:=\min\left\{ l:  c_l\geq \frac{16 \kappa_\theta R_0}{\lambda^2\hat\eta^2}\right\}
	\end{equation}
	and its output $(\hat z,\hat v,\hat \lm)$ is an approximate stationary point  of problem~\eqref{optl0}.
	Moreover, it follows from  statements a) and b) of Theorem~\ref{mainprop1}  that the total number of ACG iterations is bounded by 
	\begin{equation}\label{aux:totalacg}
	\left(\sum_{l=1}^{\bar{l}} \left\lceil1+\sqrt{\Theta_{c_ l}}\log^+_1\left(\frac{2\Theta_{c_l}}{\sigma_\theta}\right)\right\rceil\right)\left\lceil  \frac{108\kappa_\theta R_0}{\lambda^2\hat\rho^2}\right\rceil ,
	\end{equation}
	where 
	$\Theta_{c_l}=2\lambda L_{c_l}+1$.
	%Note that $\lambda L/\tau_\theta \geq 1$ in view of the definition of  $\lambda$ in \eqref{definition of sigma} and the fact that $L\geq m$, see {\bf (A2)}. 
	%	Note that if $\bar{l}=1$, the statement (a) trivially holds. Assume then that $\bar{l}>1$. 
In view of the  above definition of $c_l$ and \eqref{ineq:cl}, we have
	\begin{equation}\label{ineq:aux-clbar}
	c_l \leq \max\left\{c_1,\frac{32 \kappa_\theta R_0}{\lambda^2\hat\eta^2}\right\}, \qquad \forall l=1,\ldots, {\bar l}.
	\end{equation} 
	Hence,  \eqref{def:bartheta},  \eqref{ineq:aux-clbar},   and the definitions of $\Theta_{c_l}$ and  $L_{c_l}$ (see \eqref{definition of sigma}) imply that
	\begin{equation}\label{estimatingThetal}
	\Theta_{c_l}= \left[2\lambda \left(L +c_l\|A\|^2\right)+1\right]\leq  \left[2\lambda \left(L +\max\left\{c_1,\frac{32 \kappa_\theta R_0}{\lambda^2\hat\eta^2}\right\}\|A\|^2\right)+1\right]= {\mathcal O}\left(T_{\hat \eta}\right).
	\end{equation}
	It also  follows from  the definitions of $c_l$, $\Theta_{c_l}$, and  $L_{c_l}$ that
	\begin{align*}
	\sum_{l=1}^{\bar{l}} \sqrt{\Theta_{c_l}} 
	&= \sum_{l=1}^{\bar{l}}\sqrt{\left[2\lambda\left(L +2^{l-1}\pen_1\|A\|^2\right)+1\right]}\leq  \sum_{l=1}^{\bar{l}}\left(\sqrt{\lambda\pen_1\|A\|^22^{l}}+\sqrt{2\lambda L+1}\right)\\
	&\leq  8\sqrt{\lambda\pen_1} \|A\|\sqrt{2}^{\bar{l}}+\bar{l}\sqrt{2\lambda L+1}=8\sqrt{2 \lambda\pen_{\bar l}} \|A\|+\bar{l}\sqrt{2\lambda L+1}.
	\end{align*}
From the above inequalities,  \eqref{ineq:aux-clbar} and definition of $\bar{l}$ in \eqref{ineq:cl}, we have
		\begin{align*}
	\sum_{l=1}^{\bar{l}} \sqrt{\Theta_{c_l}} 
	&\leq 8\sqrt{2\lambda} \|A\|\sqrt{\max\left\{c_1,\frac{32 \kappa_\theta R_0}{\lambda^2\hat\eta^2}\right\}}+ \sqrt{2\lambda L+1}\log^+_1\left(\frac{64\kappa_\theta R_0}{c_1\lambda^2\hat\eta^2}\right)\\
	&= {\mathcal O}\left(\|A\|\sqrt{\lambda\max\left\{c_1,\frac{R_0}{\lambda^2\hat\eta^2}\right\}}+ \sqrt{\lambda L+1}\log^+_1\left(\frac{R_0}{c_1\lambda^2\hat\eta^2}\right)\right).
	\end{align*}
	%	where the last equality is due to the definition  of $T_{\hat \eta}$ given in \eqref{def:bartheta}.
	Since $\sqrt{\Theta_{c_l}}\log^+_1(t)\geq 1$ for every $t>0$,  statement a) follows from the above  conclusions, \eqref{aux:totalacg} and \eqref{estimatingThetal}.	
	
	b) The proof of this statement is similar
	to the one of a), by noting that  it uses Theorem~\ref{mainprop2} instead of Theorem~\ref{mainprop1} c), and the definition of $\bar l$ in \eqref{ineq:cl}	should be modified to
	% 	and by taking into account the definition of $N_0$ in \eqref{N0} and that, under the additional assumptions {\bf (B1)-(B3)}, the number of iterations performed by  the dynamic $\theta$-IPAAL method is bounded by  $\tilde{l}$, where  
	\[\bar{l}:=\min\left\{ l: \pen_l \geq 
	\frac{2N_0}{\hat \eta}\right\},
	\]
	and then \eqref{ineq:aux-clbar} is replaced by
	$$
	c_l \leq \max\left\{c_1,\frac{4 N_0}{\hat\eta}\right\}, \qquad  \forall l=1\ldots \bar l.
	$$
\end{proof}

We now make two remarks about Theorem \ref{mainTheo2}.
First, it follows from its statements a) and b)
that the ACG iteration-complexity of  the 
$\theta$-IPAAL method expressed only in terms of
the tolerance pair $(\hat \rho,\hat \eta)$ is 
${\cal O}([1/(\hat \eta \hat \rho^2)]\log(1/\hat\eta))$
under the assumption that  conditions {\bf (A1)--(A4)} hold and
can actually be sharpened to
 ${\cal O}([1/(\sqrt{\hat \eta} \hat \rho^2)]\log(1/\hat\eta))$ if  conditions {\bf (B1)--(B3)} also hold.
Second, if the initial penalty parameter $c_1$ is
chosen so as to satisfy
$c_1= 2 \bar c+ \Theta(L/\|A\|^2)$ and the
logarithms are  ignored in
the complexity bounds \eqref{bound001} and \eqref{eq:ttt1}, then
these bounds reduce to
% reduce toit can be
% shown that  the total number of ACG iterations of the $\theta$-IPAAL method is bounded by
% $$
% {\mathcal O}\left(\left\lceil\frac{R_0}{\lambda^2\hat\rho^2}\right\rceil\sqrt{\hat T_{\hat \eta}}\log^+_1\left(\hat T_{\hat{\eta}}\right)\right),
% $$
$$
		{\mathcal O}\left(\left\lceil\frac{m^2R_0}{\hat\rho^2}\right\rceil\left[ \|A\| \sqrt{\max\left\{\frac{\bar c}{m},\frac{m R_0}{\hat\eta^2}\right\}}	
		+ \sqrt{\frac{L}{m}+1}\right]\right)
		$$
		and
		$$
		{\mathcal O}\left(\left\lceil\frac{m^2R_0}{\hat\rho^2}\right\rceil\left[\|A\|\sqrt{\frac{1}{m}\max\left\{\bar c,\frac{ N_0}{\hat \eta}\right\}}	
		+ \sqrt{\frac{L}{m}+1}\right]\right),
		$$
		respectively, in view of the fact that by \eqref{definition of sigma} the stepsize $\lam$ satisfies
		$\lam={\Theta}(1/m)$.
		Third, when \eqref{optl0} has an optimal solution,
		then the above bounds can be majorized
		by ones involving the upper bound
on $R_0$ described in \eqref{ineq:majorizing-Rlambda}.

	\section{Proofs of Theorems~\ref{mainprop1} and \ref{mainprop2}}\label{sec:Technical-Results}
	
	This section contains the proofs of Theorems~\ref{mainprop1} and \ref{mainprop2}. It is divided into two subsections. The proof of
	Theorem \ref{mainprop1} is given in the first subsection while
	the one of Theorem \ref{mainprop2} is given in the second subsection.

	\subsection{Proof of  Theorem~\ref{mainprop1}}\label{TechSubsec1}
	This subsection is devoted to the proof of  Theorem~\ref{mainprop1}.
% 	Before presenting this proof, we need to consider several technical results.

We start by stating a result which describes
some relations which are frequently
used in our analysis.

\begin{lemma} Let $\{(z_k,v_k,\lm_k)\}$ be generated by the static $\theta$-IPAAL method and, for every $k\geq 1$,  define
	\begin{equation}\label{def:Delta}
	\Delta \lm_k:=\lm_k-\lm_{k-1},\quad  \Delta z_k:=z_k-z_{k-1},\quad r_k:= \Delta z_k - v_k.
		\end{equation}
	Then, for every $k\geq 1$, the following relations hold
	\begin{equation}\label{eq:deltalambda}
	\Delta \lm_{k+1} = (1-\theta) \Delta \lm_{k} + \pen A\Delta z_{k+1},
	\end{equation}	
	\begin{equation}\label{ineq:deltazkrk}
	\|\Delta z_k\| \leq (1+\sigma)\|r_k\|,\quad \|r_k\|\leq \frac{1}{1-\sigma}\|\Delta z_k\|,
	\end{equation}
 where $\sigma$ is as in \eqref{definition of sigma}.
	\end{lemma}
	\begin{proof}
	The relation in \eqref{eq:deltalambda} follows immediately  from \eqref{def:multiplier}  and  the first two relations in \eqref{def:Delta}. 
	Now, in view of  the last two relations in \eqref{def:Delta},  it is easy to see that the inequality in \eqref{GIPPinc_ineq} implies that
	$\|v_k\|\leq \sigma\|r_k\|$. It follows from  the latter inequality, the last relation in \eqref{def:Delta},  and  the triangle inequality that 
	$$\|\Delta z_k\|\leq \|\Delta z_k-v_k\|+\|v_k\|\leq (1+\sigma)\|r_k\|,\qquad
	\|r_k\|\leq \|\Delta z_k\|+\|v_k\|\leq\|\Delta z_k\|+\sigma \|r_k\|.$$  
The above  inequalities clearly imply that the ones in  \eqref{ineq:deltazkrk} hold.\end{proof}
It follows from the inequalities  in \eqref{ineq:deltazkrk} that the sequence of displacements $\{\Delta z_k\}$ goes to zero if and only if the residual sequence $\{r_k\}$ goes to zero.  %The next result presents a useful inequality which combined with the above lemma allows us to construct a decreasing sequence that controls in a certain manner the residual sequence $\{r_k\}$.
			
The next lemmas describe how the sequence $\{(z_k,\lm_k)\}$ generated by the static $\theta$-IPAAL method  affects  the value of  ${\cal L}^\theta_\pen(z,\lm)$ defined in \eqref{lagrangian2}.
	\begin{lemma}\label{aux:lemma-auglagb}
	Let $\{(z_k,v_k,\lm_k)\}$ be generated by the static $\theta$-IPAAL method and let $\{\Delta \lm_k\}$, $\{\Delta z_k\}$, and $\{r_k\}$ be as in \eqref{def:Delta}.
		Then,  for every $k\geq 1$, the following relations  hold:
		\begin{align}
		{\cal L}^\theta_\pen(z_k,\lm_k)-{\cal L}^\theta_\pen(z_k,\lm_{k-1}) &=\frac{(1-\theta)(2-\theta)}{2\pen} \|\Delta\lm_k\|^2+\frac{(1-\theta)\theta}{2\pen}\left(\|\lm_k\|^2 - \|\lm_{k-1}\|^2 \right),\label{auglagDecreasing1}\\[3mm] 
		{\cal L}^\theta_\pen(z_k,\lm_{k})-{\cal L}^\theta_\pen(z_{k-1},\lm_{k-1})&\leq
		-\frac{1-{\sigma}^2}{2\lambda}\|r_k\|^2+\frac{(1-\theta)(2-\theta)}{2\pen} \|\Delta\lm_k\|^2+\frac{(1-\theta)\theta}{2\pen}\left(\|\lm_k\|^2 - \|\lm_{k-1}\|^2 \right).\label{auglagDecreasing2}
		\end{align}
	\end{lemma}
	\begin{proof} In view of  the definition of $ {\cal L}^\theta_\pen$ in \eqref{lagrangian2}, relation \eqref{def:multiplier} and the first relation in \eqref{def:Delta}, we obtain
		\begin{align*} 
		&{\cal L}^\theta_\pen(z_k,\lm_k)-{\cal L}^\theta_\pen(z_k,\lm_{k-1})
		=(1-\theta)\left\langle\Delta\lm_k,Az_k-b\right\rangle
		=(1-\theta)\left\langle\Delta\lm_k,\frac{\lm_k-(1-\theta)\lm_{k-1}}{c}\right\rangle\\
		&=\frac{1-\theta}{c}\left[\|\Delta\lm_k\|^2+\theta\Inner{\Delta\lm_k}{\lm_{k-1}}\right]=\frac{1-\theta}{c}\left[\|\Delta\lm_k\|^2+\frac{\theta}{2}\left([\|\Delta\lm_k+\lm_{k-1}\|^2-\|\Delta\lm_k\|^2-\|\lm_{k-1}\|^2\right)\right]
		\end{align*}
which immediately implies \eqref{auglagDecreasing1}
upon using the identity $\lm_k=\Delta \lm_k+\lm_{k-1}$.
	Now,  it follows from  \eqref{GIPPinc_ineq}, \eqref{def:Delta}, and the Cauchy-Schwarz inequality, that
		\begin{align*}
		& \lambda{\cal L}^\theta_\pen(z_k,\lm_{k-1}) - \lambda{\cal L}^\theta_\pen(z_{k-1},\lm_{k-1})
		\leq -\frac{1}{2} \|z_{k}-z_{k-1}\|^2 +
		\inner{v_{k}}{z_k-z_{k-1}} +\varepsilon_k \\
		&= -\frac{1}{2} \|\Delta z_k - v_k \|^2 + \left( \frac{\|v_k\|^2}2 + \varepsilon_k \right)  \leq-\frac{1-\sigma^2}{2} \|\Delta z_k-v_k\|^2 = -\frac{1-\sigma^2}{2} \|r_k\|^2.
		\end{align*}
		The inequality in \eqref{auglagDecreasing2}  
		follows immediately by adding  the previous inequality and \eqref{auglagDecreasing1}. \end{proof}

The next result shows that the term $\|\Delta\lm_k\|^2$ in  \eqref{auglagDecreasing2} is majorized  by  $\|r_k\|^2$  and some summable terms. Its proof is postponed to Section~\ref{Sec: proofAuxlemma}.

	\begin{lemma}\label{pr:aux-new2}  Let $\{\Delta \lm_k\}$, $\{\Delta z_k\}$, and $\{r_k\}$ be as in  \eqref{def:Delta}. Then,
	for every $k \ge 2$, we have
	\begin{align*}
			\frac{(2-\theta)}{2c} \|\Delta \lm_{k}\|^2\leq& \frac{2}{\lambda\theta }\left[ \frac{2\tau_\theta(1+\sigma)^2}{\tau_\theta+1} + 2\sigma(1+\sigma) +\frac{(\tau_\theta+1)\sigma^2}{\tau_\theta}\right] \|r_{k}\|^2+\frac{(1-\theta)^2}{\theta c}\left[ \|\Delta \lm_{k-1}\|^2-\|\Delta \lm_{k}\|^2 \right]\\[3mm]
			&+\frac{1}{\lambda\theta }\left[\sigma(1+\sigma) +\frac{(\tau_\theta+1)\sigma^2}{\tau_\theta}\right]\left[\|r_{k-1}\|^2-\|r_{k}\|^2\right]+\frac{1}{\lambda \theta}\left[\|\Delta z_{k-1}\|^2-\|\Delta z_{k}\|^2\right].
			\end{align*}		
		\end{lemma}
Next result is a first step in order to show that the residual sequence $\{r_k\}$ is controlled by  a certain decreasing sequence  associated to $\{{\cal L}^\theta_\pen(z_k,\lm_k)\}$.  This fact is crucial to establish the iteration-complexity bounds for the $\theta$-IPAAL method. 
	\begin{lemma} \label{lem:auxdecL2} Define
	\begin{equation}\label{def:BCtheta}
		C_\theta:=    \frac{1-{\sigma}^2}{2}-\frac{2(1-\theta)}{\theta}
		\left[ \frac{2\tau_\theta(1+\sigma)^2}{\tau_\theta+1} + 2\sigma(1+\sigma) +\frac{(\tau_\theta+1)\sigma^2}{\tau_\theta}\right]
		\end{equation}
        and 
		\begin{equation}\label{def:Dtheta}
		\tilde C_\theta:=\frac{1-\theta}{\theta}\left[\sigma(1+\sigma)+\frac{(\tau_\theta+1)\sigma^2}{\tau_\theta}\right].
		\end{equation}
	Then, $C_\theta \ge 1/8$, $\tilde C_\theta\geq 0$ and, for every $k\geq 2$, we have
		\begin{align*}
		{\cal L}^\theta_\pen&(z_k,\lm_{k})-{\cal L}^\theta_\pen(z_{k-1},\lm_{k-1})\leq
		-\frac{C_\theta}{\lambda}\|r_k\|^2+\frac{(1-\theta)\theta}{2\pen}\left(\|\lm_k\|^2 - \|\lm_{k-1}\|^2 \right)
		 \\[3mm]
		&+\frac{(1-\theta)^3}{\theta c}\left[ \|\Delta \lm_{k-1}\|^2-\|\Delta \lm_{k}\|^2 \right]  +\frac{\tilde C_\theta}{ \lambda}\left[\|r_{k-1}\|^2-\|r_k\|^2\right]
		+\frac{1-\theta}{\theta\lambda}\left[\|\Delta z_{k-1}\|^2-\|\Delta z_{k}\|^2\right].	\end{align*}
	\end{lemma}
	
	\begin{proof}
The proof that $C_\theta\geq 1/8$ is simple but tedious, and hence it
is given in Appendix~\ref{sec:proof of Ctheta is positive}. It is immediate to see that $\tilde C_\theta\geq 0$ in view of $\theta, \sigma, \tau_\theta \in (0,1]$.	
Now the last statement of the lemma follows immediately by combining the inequality in Lemma~\ref{pr:aux-new2} with  \eqref{auglagDecreasing2} and the definitions of $C_\theta$ and $\tilde C_\theta$.	\end{proof}
	
	The next result summarizes some useful relations  about  $\{{\cal L}^\theta_\pen(z_k,\lm_k)\}$. 

	\begin{lemma}\label{lem:T+eta_decrease} Let $C_\theta$ and $\tilde C_\theta$ be as in \eqref{def:BCtheta} and \eqref{def:Dtheta}, and let  $\eta_k$ be defined by   
		\begin{equation}\label{def:newetak} 
		\eta_k:=\frac{(1-\theta)^3}{\theta\pen}\|\Delta \lm_{k}\|^2
		+\frac{\tilde C_\theta}{\lambda}\|r_k\|^2
		+\frac{1-\theta}{\theta\lambda}\|\Delta z_{k}\|^2-\frac{(1-\theta)\theta}{2\pen}\|\lm_k\|^2- \phi^*_{\bar \pen},\quad \quad \forall k\geq 1.
		\end{equation}
		Then, the following statements hold:
		\begin{itemize}
			\item[a)] for every $k \ge 2$,
			\begin{align} \label{aux:eq:T+eta1}
			 &{\cal L}^\theta_\pen(z_k,\lm_k)   +\eta_k  +   \frac{C_\theta}
			{\lambda}\|r_k\|^2 \leq  {\cal L}^\theta_\pen(z_{k-1},\lm_{k-1})   +\eta_{k-1},\\[3mm]
			 &{\cal L}^\theta_\pen(z_k,\lm_k)   +\eta_k  
		+ \frac{C_\theta}{\lam} \sum_{i=2}^k\| r_i\|^2\label{IneqT1+eta1}
		\leq {\cal L}^\theta_\pen(z_{1},\lm_{1})   +\eta_{1} .
		\end{align}
		\item[b)]for every $k \ge 1$,
			\begin{equation}\label{aux:eq:T+eta00}
			  \frac{\pen}{4} \|Az_k-b\|^2 + \frac{(1-\theta)\theta}{4\pen}\|\lm_k\|^2
			+\phi_{\bar{\pen}}(z_k)- \phi^*_{\bar \pen} \leq {\cal L}^\theta_\pen(z_k,\lm_k) +\eta_k; 
			\end{equation}
	\end{itemize}
	\end{lemma}	
	\begin{proof}
		a) The  inequality in \eqref{aux:eq:T+eta1}  follows immediately from Lemma~\ref{lem:auxdecL2}  and the definition of  $\eta_k$ given in \eqref{def:newetak}. The inequality in \eqref{IneqT1+eta1}  follows immediately from  \eqref{aux:eq:T+eta1}.
		
		b) From the definitions of  ${\cal L}^\theta_\pen$ and $\eta_k$ given in \eqref{lagrangian2} and \eqref{def:newetak}, respectively, and the fact that $\tilde C_\theta \geq 0$ (see Lemma~\ref{lem:auxdecL2}),  we obtain,	for every $k\geq 1$,			
		\begin{align*}
		&{\cal L}^\theta_\pen(z_k,\lm_k)  +\eta_k + \phi^*_{\bar \pen}
		\geq  	{\cal L}^\theta_\pen(z_k,\lm_k)+\frac{(1-\theta)^3}{\theta\pen}\|\Delta \lm_{k}\|^2
		+\frac{1-\theta}{\theta\lambda}\|\Delta z_{k}\|^2-\frac{(1-\theta)\theta}{2\pen}\|\lm_k\|^2\\[3mm]
		&\geq  (f+h)(z_k) + \frac{\pen}2 \|Az_k-b\|^2 + (1-\theta)\Inner{\lm_k}{Az_k-b-\frac{\theta\lm_k}{\pen}}+
		\frac{(1-\theta)\theta}{2\pen}\|\lm_k\|^2+\frac{(1-\theta)^3}{\theta\pen}\|\Delta \lm_{k}\|^2.
		\end{align*}
		Since, for every $k\geq 1$,  \eqref{def:multiplier}  implies that $(1-\theta)\Delta\lm_k =\pen (Az_k-b)-\theta\lm_k$, we conclude from the  above relations,  definition of $\phi_{\bar{\pen}}$ in \eqref{def:phic}, and the fact that  $\inner{u}{v} \geq -(1/2)(\|u\|^2/2+2\|v\|^2)$ for all $u,v\in \Re^n$,
		that
		\begin{align*}
		 &{\cal L}^\theta_\pen(z_k,\lm_k)  +\eta_k+ \phi^*_{\bar\pen} 	\ge   \phi_{\bar{\pen}}(z_k)+\frac{\pen-\bar{\pen}}2 \|A z_k-b\|^2 + \frac{1-\theta}{\pen}\left[\Inner{\lm_k}{(1-\theta)\Delta\lm_k} +\frac{\theta}{2}\|\lm_k\|^2+\frac{(1-\theta)^2}{\theta}\|\Delta \lm_{k}\|^2\right]\\
		&\geq  \phi_{\bar{\pen}}(z_k)+\frac{\pen}{4}\|Az_k-b\|^2 + \frac{1-\theta}{2\pen}\left[-\frac{\theta}{2}\|\lm_k\|^2 - \frac{2(1-\theta)^2}{\theta}\|\Delta \lm_k\|^2+\theta\|\lm_k\|^2
		+\frac{2(1-\theta)^2}{\theta}\|\Delta \lm_{k}\|^2\right],
		\end{align*}
where in the last inequality we also used  $c-\bar{c} > c-c/2=c/2$, in view of \eqref{assump:penaltyparameter}. The statement in b) easily follows. \end{proof}
	
	We now make a few comments about Lemma~\ref{lem:T+eta_decrease}.
	Its first statement shows that the $\theta$-AL function
	${\cal L}^\theta_c (z_k,\lm_k)$ plus
	the scalar $\eta_k$ defined in \eqref{def:newetak} works as a merit
	descent function for the static $\theta$-IPAAL method.
	Moreover, its second statement shows that both the
	feasibility gap $\|Az_k-b\|$ and the magnitude of
	the Lagrange multiplier $\|\lm_k\|$
	is well-controlled by this merit function.
	
	At a first  sight, it is natural to think that the left hand side of \eqref{IneqT1+eta1} grows linearly with $c$ but the following
	proposition shows that it is actually uniformly bounded with respect to $c$.
	This result plays an important role in showing that the $\theta$-IPAAL method finds a  $(\hat \rho,\hat \eta)$-approximate stationary point of \eqref{optl0}
	as in Definition~\ref{def:stationarypoint}, regardless
	of whether the initial point $z_0 \in \dom h$ is feasible or infeasible.

	%%%%%%%%%%%%%%%%%%%%%%%%%%%%%%%%%%%%%%%%%%%%%%%%%%%%%%%%%%%
	%%%%%%%%%%%%%%%%%%%%%%%%%%%%%%%%%%%%%%%%%%%%%%%%%%%%%%%%%%%

%	?????????????????????
%	???????????????????

	\begin{proposition}\label{prop:constante0} Let $ c > 2\bar{\pen}$,  let  $(z_1,\lm_1)$ be  generated by the static $\theta$-IPAAL method, and consider $R_0$, $\kappa_\theta$, $r_1$, and $\eta_1$  as in \eqref{def:Rlamb},  \eqref{def:constante0}, \eqref{def:Delta}  and \eqref{def:newetak}, respectively. Then, the following inequalities hold:
			
			\begin{equation} \label{aux:ineqTz1p0}
			\pen\|Az_1-b\|^2 +\frac{1}{\lambda}\|r_1\|^2 \leq  \frac{4R_0}{\lambda};
			\end{equation}
			\begin{equation}\label{ineq:p1po}
			 {\cal L}^\theta_\pen(z_1,\lm_1)+\eta_1\leq \frac{\kappa_\theta R_0}{\lambda}.
			\end{equation}
\end{proposition}

	\begin{proof}
		 The definitions of ${\cal L}^\theta_\pen$ and  $\mathcal{F}$ in \eqref{lagrangian2} and {\bf (A1)}, respectively, imply that  $ {\cal L}^\theta_\pen(z,\lm_0)=(f+h)(z)$ for all  $z\in\mathcal{F}$. Hence, from   \eqref{GIPPinc_ineq} and \eqref{def:Delta} with $k=1$,    and Lemma~\ref{lem:auxNewNest2} with 
		$\tilde{\phi} =\lambda {\cal L}^\theta_\pen(\cdot;\lm_0)$ and $s=1$, we obtain, 
		 for every $ z\in \mathcal{F}$,
		\begin{align*}
		    \lambda{\cal L}^\theta_\pen(z_1,\lm_0) + \frac{1-2\sigma^2}{2}\|r_1\|^2 &\leq   \lambda{\cal L}^\theta_\pen(z,\lm_0) +\|z-z_0\|^2 = \lambda \left[(f+h)(z) - \phi_{\bar {c}}^*\right] +\|z-z_0\|^2+\lambda\phi_{\bar {c}}^*.
		\end{align*} 
		The above inequality combined with the definition of $R_0$ in \eqref{def:Rlamb} and the fact that $\sigma \leq 1/2$  imply that 
		\begin{equation}\label{aux5555}
		  \frac{1}{4\lambda }\|r_1\|^2 +{\cal L}^\theta_\pen(z_1,\lm_0)-  \phi_{\bar {c}}^* \leq \frac{R_0}{\lambda}.
		\end{equation}
		Now note that the definition of $\phi^*_{\bar \pen}$ in {\bf (A4)} implies that
		$- (\bar{\pen}/2)\|Az_1-b\|^2\leq (f+h)(z_1) - \phi^*_{\bar \pen}$. Since $\lm_0=0$, it follows from the latter inequality and the definition of ${\cal L}^\theta_\pen$ in \eqref{lagrangian2} that 
		\[
		\frac{\pen-\bar \pen}2 \|Az_1-b\|^2\leq  (f+h)(z_1) 
		+ \frac{\pen}2 \|Az_1-b\|^2 - \phi^*_{\bar \pen}={\cal L}^\theta_\pen(z_1,\lm_0) - \phi^*_{\bar \pen}.
		\]
Hence, \eqref{aux:ineqTz1p0} follows by combining the above inequality with \eqref{aux5555} and by using that $c-\bar{c} > c-c/2=c/2$, in view of the fact that $c >2\bar{c}$.
	
Now, we proceed to prove \eqref{ineq:p1po}.	Since $\lm_0=0$, \eqref{def:multiplier} and \eqref{def:Delta}, both with $k=1$,  imply that
		\begin{equation}\label{eq:aux1111}
		\Delta \lm_1= c(Az_1-b).
		\end{equation}
		Also note that \eqref{def:Dtheta} and the fact that $\tau_\theta, \sigma \leq 1/2$ yield $\tilde C_\theta\leq (1-\theta)/(\theta \tau_\theta).$
Hence,	combining Lemma~\ref{aux:lemma-auglagb}, \eqref{def:newetak}, the first inequality  in  \eqref{ineq:deltazkrk}, all three with $k=1$, \eqref{eq:aux1111},   and the fact that $\sigma, \theta \in (0,1]$, we obtain
		\begin{align*}
		 & {\cal L}^\theta_\pen(z_1,\lm_1)+\eta_1 - {\cal L}^\theta_\pen(z_1,\lm_0)+\phi^*_{\bar \pen}
		 = \frac{(1-\theta)(2-\theta)}{2\pen}\|\Delta\lm_1\|^2 
		+\frac{(1-\theta)^3}{\theta\pen}\|\Delta \lm_1\|^2
		+\frac{\tilde C_\theta}{\lambda}\|r_1\|^2
		+\frac{1-\theta}{\theta\lambda}\|\Delta z_1\|^2\\[3mm]
		\leq&  \frac{1-\theta}{\theta}\left\{\frac{2}{\pen} \|\Delta\lm_1\|^2 +\frac{\|r_1\|^2}{\lambda \tau_\theta}+
 \frac{(1+\sigma)^2\|r_1\|^2}{\lambda}\right\}
		\le \frac{1-\theta}{\theta}\left\{
		2c\|A z_1-b\|^2  +\frac{4\|r_1\|^2}{\tau_\theta\lambda}\right\}
 		\end{align*}
 		which combined with \eqref{aux:ineqTz1p0}, \eqref{aux5555}, and the fact that $\tau_\theta <1$ imply that
 		\begin{align*}
		{\cal L}^\theta_\pen(z_1,\lm_1)+\eta_1	
		\le & \frac{ R_0}{\lambda} 
		+  \frac{16(1-\theta)R_0}{\lambda\theta\tau_\theta}.
		\end{align*}
Inequality \eqref{ineq:p1po} follows immediately from the latter one and the definition of $\kappa_\theta$ given in  \eqref{def:constante0}.	
\end{proof}

	The following result presents some estimates regarding the sequence  $\{(z_k,r_k,\lm_k)\}$ which are useful  to establish  iteration-complexity bounds for the $\theta$-IPAAL method.
	
		\begin{proposition}
		Let  $\{(z_k,v_k,\lm_k)\}$ be 
		generated by the static $\theta$-IPAAL method and consider $\{r_k\}$ as in \eqref{def:Delta}. Then, the following inequalities hold
		\begin{equation}\label{ineqs:Feas-rk-pk}
		\|A z_k-b\|^2\leq \frac{4\kappa_\theta R_0}{\lambda\pen}, \qquad \sum_{j=1}^k \|r_j\|^2 \leq 12 \kappa_\theta R_0 ,\qquad (1-\theta)\|\lm_k\|^2\leq \frac{4\pen \kappa_\theta R_0}{\lambda \theta},
		\end{equation}
		where  $\kappa_\theta $   is as in    \eqref{def:constante0}. 
	\end{proposition}
	\begin{proof}	The first and third inequalities in \eqref{ineqs:Feas-rk-pk} follow immediately by combining \eqref{aux:eq:T+eta00} and \eqref{ineq:p1po}.
Now, note that  \eqref{aux:eq:T+eta1} implies that  ${\cal L}^\theta_\pen(z_{k},\lm_{k})   +\eta_k\geq 0$ for every $k\geq 1$. 
		Hence, since, in view of \eqref{def:BCtheta},   $C_\theta\geq 1/8$, we obtain by combining \eqref{IneqT1+eta1}  and \eqref{ineq:p1po}  that 
		\[
		\sum_{j=2}^{k} \frac{1}{8\lambda}\|r_j\|^2 \leq  	
		{\cal L}^\theta_\pen(z_1,\lm_1)   +\eta_1-\left({\cal L}^\theta_\pen(z_{k},\lm_{k})   +\eta_k\right)\leq \kappa_\theta.
		\]
		We also obtain from the inequality in \eqref{aux:ineqTz1p0} and  the definition of $\kappa_\theta$ in    \eqref{def:constante0} that
		$$
		\|r_1\|^2\leq  4\lambda \kappa_\theta.
		$$
		The above inequalities easily imply the second inequality in \eqref{ineqs:Feas-rk-pk}. \end{proof}
	The next proposition shows some useful relations on the sequence $\{(\hat z_k, \hat v_k,\hat \lm_k)\}$. It shows that $(\hat z_k, \hat v_k,\hat \lm_k)$ satisfies the inclusion in \eqref{eq:approx_stationary1} and  that   $\{\|\hat v_k\|\}$ and $\{\|\hat z_k-z_k\|\}$ are  controlled by the residual sequence $\{\|r_k\|\}$ which has a subsequence converging to zero, in view of the second inequality in \eqref{ineqs:Feas-rk-pk}. Moreover, it also presents some useful bounds on the feasibility of $\{\hat z_k\}$ and the boundedness of the auxiliary  sequence $\{\hat \lm_k\}$ associated to the Lagrange multipliers.
	\begin{proposition}\label{mainauxprop1} 
		Consider the sequences $\{(z_k,v_k,\varepsilon_k)\}$
		and $\{(\hat z_k,\hat v_k,\hat \lm_k)\}$ be
		generated by the static $\theta$-IPAAL method and let $\{r_k\}$ be as in  \eqref{def:Delta}.  Then, the following relations hold, for every $k\geq 1$:
		\begin{equation}\label{inclusionAlgoProp}
		\hat v_k \in \nabla f(\hat z_k) + \partial h(\hat z_k) +A^* \hat \lm_k, \qquad \|\hat v_k\|	\leq \frac{\left(1+ 2\sigma \sqrt{\lambda L_c + 1}\right)\|r_k\|}{\lam},
		\end{equation}
     \begin{equation}\label{ineqzztilde}
		\|\hat z_k-z_k\|\leq \frac{\sigma \|r_k\|}{\sqrt{\lambda L_c+1}}, \qquad \|A \hat z_k -b\| \leq 4\sqrt{\frac{\kappa_\theta R_0}{\lambda\pen}},\qquad 
		\|\hat \lm_k\|\leq  6\sqrt{\frac{\pen\kappa_\theta R_0}{\theta\lambda}},
	 \end{equation}
		where  $L_c$  and $\kappa_\theta$ are as in  \eqref{definition of sigma} and \eqref{def:constante0}, respectively.
	\end{proposition}
	\begin{proof} First note that Step~2 of the $\theta$-IPAAL method computes  $(\hat z_k, \hat v_k )$ by the refinement procedure with input 
	$(g, h) =(g_k, h)$ and 	$(\lam,z^-,z, v) = (\lam,z_{k-1},z_k,v_k)$
	where $g_k$  as defined   in \eqref {def:ghtilde} has gradient  $L_c$-Lipschitz continuous with  $L_c$ as in  \eqref{definition of sigma}. Hence, the inclusion in \eqref{inclusionAlgoProp} follows from the inclusion in Proposition~\ref{prop:approxsol} a) and the definitions of $g_k$ and $\hat \lm_k$ given in \eqref{def:ghtilde} and \eqref{def:wkpkhat}, respectively. Now note that the scalar $\Delta$ output by the refinement procedure satisfies $\Delta \leq \varepsilon_k$, in view of  the inclusion in \eqref{GIPPinc_ineq}, the fact that $g+h={\cal L}^\theta_c(\cdot;\lm_{k-1})$, and  Proposition~\ref{prop:approxsol} b). Hence,  the inequalities in Proposition~\ref{prop:approxsol} a) imply that, for every $k\geq 1$,
	$$
	\lambda\|\hat v_k\|	\leq \|v_k+z_k-z_{k-1}\|+2\sqrt{2(\lambda L_c + 1)
	]\varepsilon_k}, \quad \|\hat z_k-z_k\| \leq \sqrt{2(\lambda L_c+1)^{-1} \varepsilon_k}
	$$
	which combined with  the inequality in \eqref{GIPPinc_ineq} and the definition of $r_k$ in \eqref{def:Delta} proves  the inequality in \eqref{inclusionAlgoProp} and the first inequality in \eqref{ineqzztilde}.
	
	 Using   the triangle inequality for norms, the first  inequality  in \eqref{ineqzztilde}, and the first and second  inequalities in \eqref{ineqs:Feas-rk-pk}, we obtain 
		\begin{align*}
		\|A \hat z_k -b\| &\leq \|A z_k-b\| +\|A\|\|\hat z_k-z_k\|  \leq 2\sqrt{\frac{\kappa_\theta R_0}{\lambda\pen}} +  
		\frac{2\sigma\|A\|\sqrt{3\kappa_\theta R_0}}{\sqrt{\lambda L_c +1}}\leq 2\sqrt{\frac{\kappa_\theta R_0}{\lambda\pen}} + \sqrt{\frac{3\kappa_\theta R_0}{\lambda c}},
		\end{align*}
		where the last  inequality is due to $ \sqrt{\lambda L_c +1}\geq \|A\|\sqrt{\lambda c}$ and the fact that $\sigma \leq \sigma_\theta \leq  1/2$. Hence,   the second  inequality in \eqref{ineqzztilde} follows.
		Now,  from the definition of $\hat \lm_k$ given in \eqref{def:wkpkhat}, the last inequality  in \eqref{ineqs:Feas-rk-pk}, the second  inequality in \eqref{ineqzztilde},  and the Cauchy-Schwarz inequality, we obtain
		\[ 
		\|\hat \lm_k\| \leq (1-\theta)\|\lm_{k-1}\|+\pen\|A\hat z_k-b\|\leq 2\sqrt{\frac{\pen \kappa_\theta R_0}{\theta \lambda}}+  4c\sqrt{\frac{\kappa_\theta R_0}{\lambda\pen}}=6\sqrt{\frac{\pen \kappa_\theta R_0}{\theta \lambda}}.
		\]
	 \end{proof}
	
	The  following  inequality, which follows by combining the second inequality in \eqref{ineqs:Feas-rk-pk},  the inequality in \eqref{inclusionAlgoProp} and the definition of $\sigma$ in \eqref{definition of sigma},  will be frequently used: 
	\begin{equation}\label{auxIneq:vkhat} 
	\|\hat v_k\|\leq \frac{3\|r_k\|}{\lambda} \leq \frac{6\sqrt{3\kappa_\theta R_0}}{\lambda} \qquad \forall k\geq 1.    
	\end{equation}
	We are now  ready to prove Theorem~\ref{mainprop1}.
	\vspace{1mm}
	
	\noindent{\bf Proof of Theorem \ref{mainprop1}:}	
 a) It follows from the second inequality in \eqref{ineqs:Feas-rk-pk}  that, for every $k\geq 1$,
	$$
	k\min\{\|r_j\|^2,j=1,\ldots,k\} \leq 12\kappa_\theta R_0.
	$$
Consider the index~$j\in \{1,\ldots, k\}$ that achieves the above minimum. It follows from the above inequality and  the first inequality in \eqref{auxIneq:vkhat} with $k=j$ that
	$$
	\|\hat v_j\| \leq \frac{3 \|r_j\|}{\lambda}\leq \frac{6\sqrt{3\kappa_\theta R_0}}{\lambda\sqrt{ k}},
	$$
	which combined with the stopping criterion of  the static $\theta$-IPAAL method proves the desired result.
	
	b)
	First note that the static $\theta$-IPAAL method invokes the ACG method with $\psi_s$ and $\psi_n$ as in \eqref{def:psi-s e psi-n} and then $\nabla \psi_s$ has Lipschitz constant $M_s$ and $\psi_n$ is $\mu$-strongly convex, where $M_s$ and $\mu$ are as in \eqref{def:Ms-muIPAAL}. Hence, it follows from Proposition~\ref{lem:nest_complex} and the definition of $\sigma$ given in \eqref{definition of sigma} that  the static $\theta$-IPAAL method performs at most 
	\[	
	\left\lceil 1+\sqrt{\frac{\lambda L_c+\tau_\theta}{1-\tau_\theta}}\log^+_1\left(\left( 1+\max\left\{\sqrt{\lambda L_c+1},\frac{1}{\sigma_\theta}\right\}\right)\sqrt{2(\lambda L_c +\tau_\theta)}\right)\right\rceil 
	\]
	ACG iterations at each outer iteration. Hence, the statement in b) follows in view of  the definition of $\Theta_c$ and the fact that $\tau_\theta \leq 1/2$.
	
	c)
First note that a) ensures that the static $\theta$-IPAAL method has finite termination. Hence, in view of its stopping criterion and the inclusion in \eqref{inclusionAlgoProp}, we immediately obtain the relations in  \eqref{Main_ineq-feasibility}. Moreover, the output of the $\theta$-IPAAL method combined with the second inequality in \eqref{ineqzztilde} imply that \eqref{eq:feasib} holds, concluding the proof. 
\qed

	\subsection{Proof of Theorem~\ref{mainprop2}}\label{proof of main theo2}
This subsection is devoted to the proof of Theorem~\ref{mainprop2}.

%Before presenting this proof, we consider some technical results that will be used. In particular, it is shown in Proposition~\ref{prop:auxbounded} that, under conditions {\bf(B1)--(B3)}, the sequence of multipliers $\{\lm_k\}$  as well as its associated  sequence $\{\hat\lm_k\}$ are bounded. This is a key result to prove Theorem~\ref{mainprop2}.
We start by specifying some bounds on the sequence $\{z_k\}$ and $\{\hat z_k\}$ generated by the $\theta$-IPAAL method.
	\begin{lemma}\label{lem:zkz*bounded} 	In addition to {\bf (A1)}--{\bf (A4)}, assume that conditions {\bf (B1)}--{\bf (B2)} hold and let $\bar{z}$ be as in
	{\bf (B1)}.
	Then, the sequences $\{z_k\}$ and $\{\hat z_k\}$ generated by the $\theta$-IPAAL method satisfy 
		\begin{equation}\label{ineq:zkz0}
		\max\{\|z_k- z_0\|,\|z_k- \bar z\|\}\leq D(\alpha), \qquad \max\{\|\hat z_k-z_0\|,\|\hat z_k-\bar{z}\|\}\leq  D(\alpha)+\sqrt{3\kappa_\theta R_0}, \qquad \forall k\geq 1,
		\end{equation}
		where   $\alpha$ is as  in Theorem~\ref{mainprop2} and  $D(\alpha)$  is as in {\bf (B2)}.
		\end{lemma}
	\begin{proof}
Since $\alpha=\kappa_\theta R_0/\lambda+\max\{\phi_{\bar{\pen}}(z_0),\phi_{\bar{\pen}}(\bar z)\}$,  Lemma~\ref{lem:T+eta_decrease} b) combined with   \eqref{ineq:p1po} and the definition of $\phi^*_{\bar \pen}$ given in {\bf(A4)} imply that
		\begin{equation} \label{level-alpha}
		\phi_{\bar{\pen}}(z_k) \leq \frac{\kappa_\theta R_0}{\lambda} + \phi^*_{\bar \pen}\leq   \frac{\kappa_\theta R_0}{\lambda} + \phi_{\bar \pen}(z_0)\leq\alpha, \qquad \forall k \geq 1,
		\end{equation}
		which means that $\{z_k\}$ is in the $\alpha$-sublevel set $L_{\phi_{\bar\pen}}(\alpha)$.  Since $\kappa_\theta>0$, the definition of $\alpha$ immediately yields  $\phi_{\bar{c}}(z_0) \leq \alpha$  and $\phi_{\bar{c}}(\bar z) \leq \alpha$, or equivalently $z_0,\bar{z}\in L_{\phi_{\bar \pen}}(\alpha)$. Hence, the first inequality in \eqref{ineq:zkz0} holds in view of assumption {\bf (B2)}.
		Now, using the first inequality in  \eqref{ineqzztilde}, the second inequality in \eqref{ineqs:Feas-rk-pk}, and the fact that $\sigma\leq \sigma_\theta\leq 1/2$ (see \eqref{definition of sigma}), we obtain
		$$
		\|\hat z_k-z_k\|\leq \frac{\sigma \|r_k\|}{\sqrt{\lambda L_c+1}}\leq  \sqrt{3\kappa_\theta R_0},\quad \forall k\geq 1.
		$$
		Hence, in view of the first inequality in 
		\eqref{ineq:zkz0} and  the triangle inequality, we conclude that, for every $k\geq 1$, 
		$$
		\max\{\|\hat z_k-z_0\|,\|\hat z_k-\bar{z}\|\}\leq  \max\{\|z_k-z_0\|,\|z_k-\bar{z}\|\}+\|\hat z_k-z_k\|\leq  D(\alpha)+\sqrt{3\kappa_\theta R_0},
		$$
		which proves the last inequality in \eqref{ineq:zkz0}.
		\end{proof}

	Next we state a  technical result which will be used in the proof of the subsequently lemma. Its proof can be found, for instance,  in \cite[Lemma~1]{PPmetNonconvex2019}.
	
	\begin{lemma}\label{lem:bound_xiN}
		Assume that $X$ is a convex set and $\bar x \in \inte(X)$, and let $\partial X$ denote the boundary of $X$.
		Then, $\dist_{\partial X}(\bar x)>0$ and
		% for every $x \in X$ and $\xi \in N_X(x)$, we have
		\[
		\|\xi\| \le \frac{\inner{\xi}{x-\bar x}}{\dist_{\partial X}(\bar x)} \quad \forall x \in X, \ \forall \xi \in N_X(x).
		\]
		%where $\partial X$ denotes the boundary of $X$.
	\end{lemma}

	The following result shows that the component of the inclusion in \eqref{inclusionAlgoProp} lying in
	$\partial h(\hat z_k)$ is bounded.
	It is worth noting that its proof strongly relies
	on the bound for $\|\hat \lm_k\|$ derived in
	Proposition~\ref{mainauxprop1}. 
	
	\begin{lemma}\label{lem:xikbounded} In addition to {\bf (A1)}--{\bf (A4)}, assume that  conditions {\bf (B1)}--{\bf (B3)} hold and let $ M_\theta(\alpha)$ be defined by
		\begin{equation}\label{def:Mtheta}
		M_\theta(\alpha):=\frac{2\max\left\{D(\alpha),\sqrt{3\kappa_\theta R_0}\right\}}{\mbox{\rm dist}_{\partial (\dom h)} (\bar z)}\left[L D(\alpha)+\left(L+\frac{25}{\sqrt{\theta}\lambda}\right){\sqrt{3\kappa_\theta R_0}}+\|\nabla f(z_0)\|+ S(\alpha)\right].
		\end{equation}
	Let $\{(\hat z_k,\hat v_k,\hat \lm_k)\}$ be generated by the static $\theta$-IPAAL method and consider  the sequence $\{\hat\xi_k\}$ given by 
		\begin{equation}\label{def:xi}
		\hat \xi_k:= \hat v_k-  \nabla f(\hat z_k) - A^*\hat \lm_k, \qquad \forall k\geq 1.
		\end{equation}
Then,  $\hat \xi_k\in \partial h(\hat z_k)$ and $\|\hat \xi_k\|\leq M_\theta(\alpha)+S(\alpha),$	for every $k\geq 1$. 
	\end{lemma}
	\begin{proof}
		The first statement of the lemma immediately follows from the inclusion in  \eqref{inclusionAlgoProp} and the definition of  $\hat \xi_k$ given in \eqref{def:xi}. 
		Now note that by using the Cauchy-Schwarz inequality and the last two inequalities in \eqref{ineqzztilde},  we obtain
		\begin{align}\label{inez:auxxik2}
		\|\inner{\hat \lm_k}{A\hat z_k-b}\| &\leq  \|\hat \lm _k\|  \|A\hat z_k-b\|\leq  24\sqrt{\frac{\pen\kappa_\theta R_0}{\lambda\theta}}\sqrt{\frac{ \kappa_\theta R_0}{\lambda\pen}}\leq\frac{24\kappa_\theta R_0}{\lambda\sqrt{\theta}}, \qquad \forall k\geq 1,
		\end{align}
		where the last inequality is due to the fact that $c-\bar{c}\geq c/2$ (in view of  \eqref{assump:penaltyparameter}).
		On the other hand,  since   $\hat \xi_k \in \partial h(\hat z_k)$ for every $k\geq 1$, assumption {\bf (B3)} yields 
		\begin{equation}\label{decomposition-xik}
		\hat \xi_k= \hat \xi^s_k+\hat \xi^N_k, \quad \|\hat \xi^s_k\|\leq S(\alpha), \quad \hat \xi^N_k\in N_{\dom h}(\hat z_k).
		\end{equation}
		where $\alpha$ is  as in Theorem~\ref{mainprop2}.
		Let  $\bar{z}$ be as in ${\bf (B1})$ and note that $A\bar z=b$.	Hence, it follows from Lemma~\ref{lem:bound_xiN} with $x=\hat z_k$,  $\bar{x}=\bar{z}$ and $X= \dom h$, the Cauchy-Schwarz inequality, the triangle inequality, \eqref{inez:auxxik2} and  \eqref{decomposition-xik}, the Lipshichtz continuity of $\nabla f$ and \eqref{auxIneq:vkhat}
  that, for every $k\geq 1$, 
		\begin{align*} 
		\mbox{\rm dist}_{\partial (\dom h)} (\bar z)\|\hat \xi^N_k\| &\leq\inner {\hat \xi_k^N} {\hat z_k-\bar z} 
		= \inner{\hat \xi_k-\hat \xi^s_k}{\hat z_k-\bar z}=\inner{\hat v_k-\nabla f(\hat z_k) -\hat \xi_k^s}{\hat  z_k-\bar z}- \inner{\hat \lm_k}{A\hat z_k-b}\\
		&\le \left(\|\nabla f(\hat z_k)-\nabla f(z_0)\|+\|\nabla f(z_0)\|+ \|\hat v_k\|+\|\hat \xi_k^s\|\right)\|\bar z-\hat  z_k\|+\frac{24\kappa_\theta R_0}{\lambda\sqrt{\theta}}\\
		&\le \left(L\|\hat z_k-z_0\|+\|\nabla f(z_0)\|+ \frac{6\sqrt{3\kappa_\theta R_0} }{\lam}+S(\alpha)\right)\|\bar z-\hat  z_k\|+\frac{24\kappa_\theta R_0}{\lambda\sqrt{\theta}},
		\end{align*}
		%
		%--------------------------------
		%Drafts for the case in which \theta=0 
		%\begin{align*} 
		%\mbox{\rm dist}_{\partial (\dom h)} (\bar z)\|\hat \xi^N_k\| &\leq \frac{1}k \sum_{i=1}^k \inner {\hat \xi_k} {z_i-\bar z}
		%= \frac{1}k \sum_{i=1}^k \inner{\hat v_k-\nabla f(\hat z_k) -A^* p_k}{z_i-\bar z} \\
		%&= \frac{1}k \sum_{i=1}^k \inner{\hat v_k-\nabla f(\hat z_k)}{z_i-\bar z} - \frac{1}k \sum_{i=1}^k
		%\inner{p_k}{Az_i-b} \\
		%&= \frac{1}k \sum_{i=1}^k \inner{\hat v_k-\nabla f(\hat z_k)}{z_i-\bar z} - \frac{1}{kc} \|p_k\|^2
		%\end{align*}
		%
		%--------------------------------
		which combined with the  last inequality in \eqref{ineq:zkz0} and the fact that $\theta<1$, imply that
		\begin{align}
\mbox{\rm dist}_{\partial (\dom h)} (\bar z)\|\hat \xi^N_k\| 
		&\le \left[L D(\alpha)+\left(L+\frac{6}{\lambda}\right){\sqrt{3\kappa_\theta R_0}}+\|\nabla f(z_0)\|+ S(\alpha)\right]\left(D(\alpha)+\sqrt{3\kappa_\theta R_0}\right)+\frac{24\kappa_\theta R_0}{\lambda\sqrt{\theta}}\nonumber\\
		&\le2 \left[L D(\alpha)+\left(L+\frac{14}{\sqrt{\theta}\lambda}\right){\sqrt{3\kappa_\theta R_0}}+\|\nabla f(z_0)\|+ S(\alpha)\right]\max\left\{D(\alpha),\sqrt{3\kappa_\theta R_0}\right\}\label{eqaux:xiN}		
		\end{align}
		The above inequalities and  \eqref{def:Mtheta} imply that 
		$$
		\|\hat \xi^N_k\|\leq  M_\theta(\alpha), \qquad \forall k\geq 1. 
		$$
		Hence, \eqref{decomposition-xik}  and  the triangle inequality imply that 
		\[
		\|\hat \xi_k\| \le \|\hat \xi^N_k\|+\|\hat \xi^s_k\|\leq  M_\theta(\alpha)+S(\alpha), \qquad  \forall k\geq 1,
		\]
		proving the last statement of the lemma. 	\end{proof}
	
	In the following,  we state  a basic result that will be used in the proof of the next proposition. Its proof can be found, for instance, in \cite[Lemma~1.4]{MaxJeffRen-admm}. 
	
	\begin{lemma}\label{lem:sigma+} 
		Let $S \in \Re^{l\times n}$ be a non-zero matrix and let $\sigma^+(S)$ denote  the smallest positive eigenvalue of $(S^*S)^{1/2}$. Then, 
		for every $u \in \Re^l$, there holds
		\[
		\|\cP_{S}(u)\|\leq \frac{1}{\sigma^+(S)}\|S^*u\|.
		\]
	\end{lemma}

	The next result shows that the sequences of multipliers $\{\lm_k\}$ and $\{\hat\lm_k\}$ are bounded by a quantity
	which does not depend on the penalty parameter $\pen$.
	This fact in turn is easily seen to imply
	an ${\cal O}(1/c)$ bound on $\|A \hat z_k-b\|$.

	\begin{proposition}\label{prop:auxbounded}
		Under the assumptions of Theorem~\ref{mainprop2}, the following inequalities hold 
		\begin{equation}\label{ineqlmk}
		\|\hat \lm_k\| \leq   N_0, \qquad
		\|\lm_k\|\leq  N_0 \qquad \forall k\geq 1,
		\end{equation}
		where $N_0$ is  as in $\eqref{N0}$.
	\end{proposition}
	\begin{proof}
		It follows from  \eqref{def:xi},  the triangle inequality for norms,  the Lipschitz continuity of $\nabla f$, the relation in \eqref{auxIneq:vkhat},
		and the last statement of	Lemma~\ref{lem:xikbounded},  that
		\begin{align}\label{ineq:auxbouded0}
		\|A^*\hat \lm_k\|&=\|\hat v_k-\nabla f(\hat z_k)- \hat \xi_k\|	\leq  \|\nabla f(\hat z_k)-\nabla f( z_0)\|+\|\nabla f(z_0)\|+\|\hat v_k\|+\|\hat \xi_k\|\\\nonumber
		& \leq  L\|\hat z_k- z_0\|+\|\nabla f(z_0)\|+\frac{6\sqrt{3\kappa_\theta R_0}}{\lambda}+ M_\theta(\alpha)+S(\alpha).
		\end{align}
		On the other hand, since $\lm_0=0$,    \eqref{def:wkpkhat} and \eqref{def:multiplier}  imply that   $\lm_k,\hat \lm_k \in {\rm  Im}\, A$, for every $k\geq 1$. Hence, it follows from Lemma~\ref{lem:sigma+} with $S=A$ that
		$$ 
		\|\hat \lm_k\| \leq \frac{1}{\sigma^+(A)}\|A^*\hat \lm_k\|,
		$$
		which combined with \eqref{ineq:auxbouded0} and  the last inequality in \eqref{ineq:zkz0} imply that
		\begin{align}
		\|\hat \lm_k\|&\leq \frac{1}{\sigma^+(A)}\left[L D(\alpha)+L\sqrt{3\kappa_\theta R_0}+\|\nabla f(z_0)\| + \frac{6\sqrt{3\kappa_\theta R_0}}{\lambda}+ M_\theta(\alpha)+S(\alpha)\right],\nonumber\\
		&=\frac{1}{\sigma^+(A)}\left[L D(\alpha)+\left(L+\frac{6}{\lambda}\right)\sqrt{3\kappa_\theta R_0}+\|\nabla f(z_0)\| + M_\theta(\alpha)+S(\alpha)\right]=:\hat N_0,
		 \qquad \forall k\geq 1.\label{def:Nhat}
		 	\end{align}
		Hence,  the first inequality in \eqref{ineqlmk} follows, in view of  the fact that $\hat N_0\leq N_0$ (see the definition of $N_0$ in \eqref{N0}).
		Now, subtracting  \eqref{def:wkpkhat} from  \eqref{def:multiplier}, using the triangle inequality for norms, 	 the first inequality in  \eqref{ineqzztilde}, the second inequality in \eqref{ineqs:Feas-rk-pk},	and the definitions of $L_c$  and $\sigma$ given  in \eqref{definition of sigma}, we obtain 
		\begin{align*}
		\|\lm_k\|&\leq \|\hat \lm_k\|+ \pen\| A(z_k-\hat z_k)\|\leq
		\|\hat \lm_k\| + \frac{\sigma c\|A\|\|r_k\|}{\sqrt{\lambda L_c+1}}\\
		&\leq \hat N_0+\frac{c\|A\|\|r_k\|}{\lambda(L+c\|A\|^2)+1}\leq \hat N_0+\frac{2\sqrt{3\kappa_\theta R_0 }}{\lambda\|A\|}.
		\end{align*}
		Hence, the  second inequality in \eqref{ineqlmk} follows by using the definitions of $N_0$ and $\hat N_0$ given in \eqref{N0} and \eqref{def:Nhat}, respectively, and that $\sigma^+(A)\leq \|A\|$. \end{proof}
	
	Now we are ready to proof of Theorem~\ref{mainprop2}.	
	
	\vspace{3mm}	
	\noindent{\bf Proof of Theorem~\ref{mainprop2}}: 
	The first statement follows from the first one in Theorem~\ref{mainprop1} c). 
	Now, using  \eqref{def:wkpkhat}, Proposition~\ref{prop:auxbounded}, the triangle inequality for norms and the fact that $\theta \in (0,1]$, we obtain
	\[
	c\|A\hat z_k-b\|\leq \|\hat \lm_k\| + (1-\theta)\|\lm_{k-1}\|\leq 2N_0,
	\]
	where $N_0$ is as in \eqref{N0}. Hence,  \eqref{eq:feasib3}  immediately follows.\qed

\section{Proof of Lemma~\ref{pr:aux-new2}}\label{Sec: proofAuxlemma}
	The main goal of this section is to prove Lemma~\ref{pr:aux-new2}.
	
	Before giving its proof, we state and prove two technical results. The first one essentially describes equivalent
	but useful ways of expressing the inclusion
	in \eqref{GIPPinc_ineq}.
	The  second one establishes an intermediate
	bound used in the proof of Lemma \ref{pr:aux-new2}.
	
% 	The aforementioned relations will be used to prove the  second lemma  which presents an inequality that is a first step to prove Lemma~\ref{pr:aux-new2}.
	
	\begin{lemma} \label{lem:appendix-akThetak} Let $\{(z_k,\lm_k,v_k,\varepsilon_k)\}$ be generated by  the static $\theta$-IPAAL method, let  $\Delta z_k$, $\Delta \lm_k$ and $r_k$ be as in  \eqref{def:Delta}, and define $\Delta v_{k+1}:=v_{k+1}-v_{k}$ and		
	\begin{equation}\label{def:ak}
		a_k(\cdot)  := \frac{1}2\|\cdot-(z_k+v_{k+1})\|^2 - \frac{1}{2}\|\cdot-(z_{k-1}+v_k)\|^2 + \lambda(1-\theta) \Inner{\Delta\lm_k}{A\cdot-b }.
		\end{equation}
		Then, for every $k\geq 1$, the following statements hold:
		\begin{itemize}
			\item[a)] $a_k$ is an affine function whose gradient $\nabla a_k$ is given by
			\begin{equation}\label{eq:nablaqk}
			\nabla a_k  = -  \Delta z_{k} -\Delta v_{k+1} + \lambda(1-\theta)A^*\Delta\lm_k; 
			\end{equation}
			
			\item[b)] the following inclusions hold
			\[
			0 \in \partial_{\varepsilon_{k}}  
			\left[\lam  {\cal L}^\theta_\pen(\cdot,\lm_{k-1} ) +  \frac{1}2\|\cdot-(z_{k-1}+v_k)\|^2\right](z_{k}),
			\]
			\[
			-\nabla a_k  \in \partial_{\varepsilon_{k+1}}  \left[\lam  {\cal L}^\theta_\pen(\cdot,\lm_{k-1} ) +  \frac{1}2\|\cdot-(z_{k-1}+v_k)\|^2\right](z_{k+1}).
			\]
	\end{itemize}
		
		%		v_{k} \in \partial_{\varepsilon_k} \psi_k (z_k),\quad	
	\end{lemma}
	
	\begin{proof}
		a)	This statement follows trivially from the definition of $a_k$ given in  \eqref{def:ak} and the relations in  \eqref{def:Delta}.
		
		b) In view of \eqref{def:epsSubdiff}, it is immediate to see that the inclusion in \eqref{GIPPinc_ineq} is equivalent to the first one in b). From the definitions of $ {\cal L}^\theta_\pen$ and $a_k$ given in \eqref{lagrangian2} and \eqref{def:ak}, respectively, we see that
		\[
		\lam  {\cal L}^\theta_\pen(\cdot,\lm_{k}) +  \frac{1}2\|\cdot-(z_{k}+v_{k+1})\|^2 =
		\lam  {\cal L}^\theta_\pen(\cdot,\lm_{k-1} ) +  \frac{1}2\|\cdot-(z_{k-1}+v_k) \|^2 + a_k(\cdot).
		\]
		It follows from this relation and the first inclusion of b) with $k=k+1$ that 
		\begin{equation*}
		0  \in \partial_{\varepsilon_{k+1}} \left[ \lam  {\cal L}^\theta_\pen(\cdot,\lm_{k-1} ) +  \frac{1}2\|\cdot-(z_{k-1}+v_k)\|^2 + a_k(\cdot)
		\right] (z_{k+1}).
		\end{equation*}
		Since $a_k$ is an affine function, the latter inclusion easily implies that the second one of b) holds.\end{proof}
	
%	(b)	 From the definitions of $ {\cal L}^\theta_\pen$ and $a_k$ given in \eqref{lagrangian2} and \eqref{def:ak}, respectively, we see that
%		\[
%		\lam  {\cal L}^\theta_\pen(\cdot,\lm_{k}) +  \frac{1}2\|\cdot-z_{k})\|^2 =
%		\lam  {\cal L}^\theta_\pen(\cdot,p_{k-1} ) +  \frac{1}2\|\cdot-z_{k-1} \|^2 + a_k(\cdot).
%		\]
%		Hence, the inclusion in \eqref{GIPPinc_ineq} implies that
%		\begin{equation*}
%		0  \in \partial_{\varepsilon_{k+1}} \left[ \lam  {\cal L}^\theta_\pen(\cdot,p_{k-1} ) +  \frac{1}2\|\cdot-(z_{k-1}+v_k)\|^2 + a_k(\cdot)
%		\right] (z_{k+1}).
%		\end{equation*}
%		Since $a_k$ is an affine function, the latter inclusion easily implies that the ones  in (b) hold.

	\begin{lemma}Let $\{(z_k,\lm_k,v_k,\varepsilon_k)\}$ be generated by  the static $\theta$-IPAAL method and consider   $\Delta z_k$, $\Delta \lm_k$ and $r_k$  as in  \eqref{def:Delta}. Then, for every $k\geq 1$, the following inequality holds 
			\begin{align}\label{def:Thetak}
			\Theta_k:=& \frac{c \lambda}{1+\tau_\theta}  \|A\Delta z_{k+1}\|^2+\lambda(1-\theta) \inner{\Delta \lm_{k}}{A\Delta z_{k+1}}
			\leq \left[ \frac{2\tau_\theta(1+\sigma)^2}{\tau_\theta+1} + 2\sigma(1+\sigma) +\frac{(\tau_\theta+1)\sigma^2}{\tau_\theta}\right] \|r_{k+1}\|^2\nonumber \\[3mm]
			&+\left[\frac{\sigma(1+\sigma)}{2} +\frac{(\tau_\theta+1)\sigma^2}{2\tau_\theta}\right]\left[\|r_{k}\|^2-\|r_{k+1}\|^2\right]+\frac{1}{2}\left[\|\Delta z_{k}\|^2-\|\Delta z_{k+1}\|^2\right].
			\end{align}
	\end{lemma}

		\begin{proof}
				 For every $k$, let us  consider the function 
		\begin{align}\label{def:psik}
		\psi_k&:=\lam  {\cal L}^\theta_\pen(\cdot,\lm_{k-1} ) + \frac{1}2\|\cdot-(z_{k-1}+v_k)\|^2.
		\end{align} Note that, in view of {\bf (A2)},    \eqref{lowerCurvature-m},  \eqref{def:psik},  and the definition of $ {\cal L}^\theta_\pen$  in \eqref{lagrangian2}, we obtain   $\psi_k  - (1/2)\|\cdot\|_Q^2$ is convex where $Q:= (1-\lambda m)I + c\lambda A^*A\in S^n_ {++}$. Hence, in view of  Lemma~\ref{lem:appendix-akThetak} b),  using  Lemma~\ref{lem:proxinclusion} twice, first with $\psi=\psi_k$, $\xi=1$, $(y,v,\eta)=(z_k,0,\varepsilon_k)$, and   $\tau=\tau_\theta$, and second with the same
		$\psi$, $\xi$, and $\tau$ but 
		$(y,v,\eta)=(z_{k+1},-\nabla a_k,\varepsilon_{k+1})$
		where $a_k$ is as in \eqref{def:ak}, we obtain
		\begin{align}
		\psi_k(u)&\geq \psi_k(z_{k})+\frac{1}{2(1+\tau_\theta)}\|u-z_{k}\|_Q^2-(1+\tau_\theta^{-1})\varepsilon_{k},\\ \psi_k(u')&\geq  \psi_k(z_{k+1}) +\inner{-\nabla a_k}{u'-z_{k+1}}+\frac{1}{2(1+\tau_\theta)}\| u'-z_{k+1}\|_Q^2-(1+\tau_\theta^{-1})\varepsilon_{k+1},
		\end{align}
		for all $u, u'\in \Re^n$.
		Adding both inequalities  with $u=z_{k+1}$ and $u'=z_k$, and using the  relation in \eqref{eq:nablaqk} and the definition of $Q$ given above, we conclude that 
		\begin{align*}
		&(1+\tau_\theta^{-1})(\varepsilon_k+\varepsilon_{k+1}) \geq \inner{\nabla a_k}{\Delta z_{k+1}}+ \frac{\|\Delta z_{k+1}\|_Q^2}{1+\tau_\theta} \\[3mm]
		= &-\inner{\Delta z_{k}}{\Delta z_{k+1}}- \inner{\Delta v_{k+1}}{\Delta z_{k+1}} +\lambda(1-\theta) \inner{\Delta \lm_{k}}{A\Delta z_{k+1}} 
		+ \frac{1-\lam m }{1+\tau_\theta}\|\Delta z_{k+1}\|^2
		+ \frac{c \lambda}{1+\tau_\theta}  \|A\Delta z_{k+1}\|^2. 
		\end{align*}
	Rewriting this inequality and using the definition of $\Theta_k$ in \eqref{def:Thetak}, we obtain
		\begin{align}
		\Theta_k&= \frac{c \lambda}{1+\tau_\theta}  \|A\Delta z_{k+1}\|^2+\lambda(1-\theta) \inner{\Delta \lm_{k}}{A\Delta z_{k+1}} \nonumber\\[3mm]
		&\leq \inner{\Delta z_{k}}{\Delta z_{k+1}}-\frac{1-\lam m }{\tau_\theta+1}\|\Delta z_{k+1}\|^2+ \inner{\Delta v_{k+1}}{\Delta z_{k+1}}+\frac{\tau_\theta+1}{\tau_\theta}(\varepsilon_k+\varepsilon_{k+1}).\label{def:Thetak2}
		\end{align}
		We now proceed to estimate the right-hand side of the last inequality. Using that $\inner{u}{\tilde u}\leq (\|u\|^2+\|\tilde u\|^2)/2$ for every $u,\tilde u\in \Re^n$,  \eqref{ineq:deltazkrk} and that $\tau_\theta=\lambda m$, we have 
		\begin{align*}
		\inner{\Delta z_{k}}{\Delta z_{k+1}}-\frac{1-\lam m }{\tau_\theta+1}\|\Delta z_{k+1}\|^2&\leq \frac{\|\Delta z_k\|^2}{2}+\frac{\|\Delta z_{k+1}\|^2}{2}-\frac{1-\tau_\theta }{\tau_\theta+1}\|\Delta z_{k+1}\|^2\\[3mm] 
		&=\frac{\|\Delta z_k\|^2-\|\Delta z_{k+1}\|^2}{2}+\frac{2\tau_\theta}{\tau_\theta+1}\|\Delta z_{k+1}\|^2\\[3mm]
		&\leq\frac{\|\Delta z_k\|^2-\|\Delta z_{k+1}\|^2}{2}+\frac{2\tau_\theta(1+\sigma)^2}{\tau_\theta+1}\|r_{k+1}\|^2.
		\end{align*}
		Using Cauchy-Schwarz and triangle inequalities,   the inequality in \eqref{GIPPinc_ineq} combined with the definitions of $\Delta z_k$ and $r_k$ in \eqref{def:Delta}, and \eqref{ineq:deltazkrk}, we obtain
		\begin{align*}
		\inner{\Delta v_{k+1}}{\Delta z_{k+1}} &\leq ( \|v_{k+1}\| + \|v_k\| ) \|\Delta z_{k+1}\|
		\le \sigma  (\|r_{k+1}\| + \|r_k\| ) \|\Delta z_{k+1}\|\\[3mm]
		&\le \sigma(1+\sigma)  (\|r_{k+1}\| + \|r_k\| ) \|r_{k+1}\|\leq \frac{\sigma(1+\sigma)\|r_k\|^2}{2}+\frac{3\sigma(1+\sigma)\|r_{k+1}\|^2}{2}\\[3mm]
		&=\frac{\sigma(1+\sigma)}{2}\left[\|r_k\|^2-\|r_{k+1}\|^2\right]+2\sigma(1+\sigma)\|r_{k+1}\|^2,
		\end{align*}
where the last inequality is due to $ab\leq  (a^2+b^2)/2$, $\forall a,b \in \Re$.
		The inequality in \eqref{GIPPinc_ineq} combined with the definitions of $\Delta z_k$ and $r_k$ in \eqref{def:Delta}, also yields
		\begin{align*}
		\varepsilon_k+\varepsilon_{k+1}\leq \frac{\sigma^2\|r_k\|^2}{2}+\frac{\sigma^2\|r_{k+1}\|^2}{2}= \sigma^2\|r_{k+1}\|^2+\frac{\sigma^2}{2}\left[\|r_{k}\|^2-\|r_{k+1}\|^2\right].
		\end{align*}
		Hence, in view of  \eqref{def:Thetak2}  and  the above inequalities, we obtain
		\begin{align*}
		\Theta_k&\leq\frac{1}{2}\left[\|\Delta z_k\|^2-\|\Delta z_{k+1}\|^2\right]+\frac{2\tau_\theta (1+\sigma)^2}{\tau_\theta+1}\|r_{k+1}\|^2+\frac{\sigma(1+\sigma)}{2}\left[\|r_k\|^2-\|r_{k+1}\|^2\right]+2\sigma(1+\sigma)\|r_{k+1}\|^2\\[3mm]
		&+\frac{(\tau_\theta+1)\sigma^2}{\tau_\theta}\|r_{k+1}\|^2+\frac{(\tau_\theta+1)\sigma^2}{2\tau_\theta}\left[\|r_{k}\|^2-\|r_{k+1}\|^2\right],
		    \end{align*}
		which after simple algebraic manipulations proves the desired inequality.\end{proof}
\vspace{2mm}		
		
	We are now ready to prove Lemma~\ref{pr:aux-new2}.
	
	\vgap
	
	\noindent{\bf Proof of Lemma~\ref{pr:aux-new2}:}
	In view of \eqref{eq:deltalambda} and the definition of $\Theta_k$ in \eqref{def:Thetak}, we obtain  
	   \begin{align*}
	   \Theta_k &= \frac{c\lambda}{1+\tau_\theta}  \|A\Delta z_{k+1}\|^2+\lambda(1-\theta) \inner{\Delta \lm_{k}}{A\Delta z_{k+1}}\\[3mm]
		&= \frac{\lambda}{c} \left(\frac{1}{1+\tau_\theta}\|\Delta \lm_{k+1} - (1-\theta) \Delta \lm_{k} \|^2 
		+ (1-\theta)\inner{\Delta \lm_{k}} {\Delta \lm_{k+1} - (1-\theta) \Delta \lm_{k}}\right) \\[3mm]
		&=\frac{\lambda}{c(1+\tau_\theta)} \left(\|\Delta \lm_{k+1}\|^2 
		-\tau_\theta(1-\theta)^2 \|\Delta \lm_k\|^2 
		+ (\tau_\theta-1)(1-\theta)\inner{\Delta\lm_k}{\Delta \lm_{k+1}}\right).
		\end{align*}
		Hence, using Cauchy-Schwarz inequality and the facts that $\tau_\theta \in (0,1)$ and $ab\geq -(a^2+b^2)/2$ for all $a,b \in \Re$, we obtain
		\begin{align*}
		\Theta_k &\geq\frac{\lambda}{c(1+\tau_\theta)} \left( \|\Delta \lm_{k+1}\|^2 -\tau_\theta(1-\theta)^2\|\Delta \lm_{k}\|^2-
		\frac{(1-\tau_\theta)}{2} \|\Delta \lm_{k+1}\|^2-\frac{ (1-\tau_\theta) (1-\theta)^2 }{2} \|\Delta \lm_{k}\|^2  \right) \\[3mm]
		&= \frac{\lambda}{c} \left(
		\frac{1}{2} \|\Delta \lm_{k+1}\|^2-\frac{(1-\theta)^2 }{2} \|\Delta \lm_{k}\|^2  \right) = \frac{\lambda}{c} \left(
		\frac{\theta(2-\theta)}{2} \|\Delta \lm_{k+1}\|^2+\frac{(1-\theta)^2}{2}\left[ \|\Delta \lm_{k+1}\|^2-\|\Delta \lm_{k}\|^2 \right]  \right). \end{align*}
		The conclusion of the lemma now follows by combining  the above inequality with  \eqref{def:Thetak}. \qed

\section{Numerical experiments\label{sec:numerical}}

This section presents computational results to
illustrate  the performance
of the $\theta$-IPAAL method for different values
of the parameter $\theta$. The computational results are limited
only to a class of linearly constrained
quadratic matrix (LCQM) problems
but they should provide a good
indication of how promising the $\theta$-IPAAL
is as the parameter $\theta$  decreases towards
zero.
%one with a large parameter $\theta$.
% Even though the results are quite limited, we believe that they provide how promising the $\theta$-IPAAL method is. %The main goal is to illustrate  its numerical performance under  different values of the parameter $\theta$. % which as we will see affects considerably its  performance. 

%??for solving  linearly constrained quadratic matrix (LCQM) problems.The main goal of this section  is to show that  IPAAL method with  $\theta\neq 1$ outperforms the latter variant. %The performance of  the accelerated gradient (AG) method  \cite[Algorithm~2]{nonconv_lan16} . IPAAL method is also For comparisons purpose, I show that t compares twoto variants of the IPAAL method and against one state-of-the-art nonconvex composite optimization algorithm. The problems solved for this section are of the form given in 	\eqref{optl0}.
%the $\theta$-IPAAL variant considered. 
 In order to describe  the  LCQM  problems considered here, let  $l,n\in\mathcal{N}$,  $\alpha_{1},\alpha_{2}\in\Re_{++}$,  $b,d \in\Re^{l}$, matrices $\{A_i\}_{i = 1}^{l},\{B_j\}_{j = 1}^{n},\{C_i\}_{i = 1}^{l} \subseteq \Re^{n \times n}$, and a positive diagonal matrix $D\in\Re^{n\times n}$ be given and define the  
linear operators ${\cal A}:S_{+}^{n}\mapsto\Re^{l}$ , ${\cal B}:S_{+}^{n}\mapsto\Re^{n}$,
and ${\cal C}:S_{+}^{n}\mapsto\Re^{l}$  by
\begin{align*}
\left[{\cal A}(z)\right]_{i}  =\left\langle A_{i},z\right\rangle _{F},\quad 
\left[{\cal B}(z)\right]_{j}  =\left\langle B_{j},z\right\rangle _{F},\quad 
\left[{\cal C}(z)\right]_{i}  =\left\langle C_{i},z\right\rangle _{F},\qquad \forall i=1,\ldots,l, \; \forall  j=1,\ldots, n.
\end{align*}
The LCQM problem is
\begin{align}\label{LCQP1}
\min_{z}\quad  & \frac{\alpha_{1}}{2}\|{\cal C}(z)-d\|^{2}-\frac{\alpha_{2}}{2}\|D{\cal B}(z)\|^{2}\\
\text{s.t.}  \quad & {\cal A}(z)=b, \quad  z\in P_{n}\nonumber,
\end{align}
where $P_{n}=\{z\in S_{+}^{n}:{\rm tr}\, z=1\}$ denotes the $n$--dimensional
spectraplex.
Note that \eqref{LCQP1} can be put in the setting of  \eqref{optl0} by considering
\begin{align*}
  f(z)=\frac{\alpha_{1}}{2}\|{\cal C}(z)-d\|^{2}-\frac{\alpha_{2}}{2}\|D{\cal B}(z)\|^{2},\quad h(z)=\delta_{P_{n}}(z).
\end{align*}

 In the numerical experiments, the entries of  $A_{i},B_{j},C_{i},b$, and $d$ (resp.,
$D$) were generated by sampling from the uniform distribution ${\cal U}[0,1]$
(resp., ${\cal U}[1,1000]$) with  5.0\% and 1.0\% of the entries of the matrices $A_{i},B_{j},$ and $C_{i}$ being nonzero when $(l,n) = (5,20)$ and $(l,n)=(25,100)$, respectively.
The scalars $\alpha_{1},\alpha_{2}\in\Re_{++}$ are selected such that the  pair of lower and upper curvatures $(m,L)$  satisfies 
$L=\lambda_{\max}(\nabla^{2}f)$ and $-m=\lambda_{\min}(\nabla^{2}f)$. In particular, the inequalities \eqref{gradLips} and \eqref{lowerCurvature-m} are satisfied.

% The initial point $z_{0} \in S_{+}^{n}$ is
% randomly generated, namely, $z_{0}=\sum_{i=1}^{3}e_{i}\nu_{i}\nu_{i}^{T}$
% where $\nu_{i}:=\tilde{\nu}_{i}/\|\tilde{\nu}_{i}\|$, $e_{i}:=\tilde{e}_{i}/(\sum_{j=1}^{3}\tilde{e}_{i})$, $\tilde{\nu}_{i}\sim {\cal U}^{n}[0,1]$ and $\tilde{e}_{i}\sim{\cal U}[0,1]$, for $i=1,2,3$.

The numerical  experiments were performed using  MATLAB 2019b and a MacOS 64-bit machine with  an Intel Core i5 processor and  8 GB of memory.

This section reports computational results for the ``theoretical'' version of the $\theta$-IPAAL method studied in
Subsection~\ref{subsec:dynamicIPAAL} as well as of a more aggressive version
whose motivation is as follows.  Note that the parameters $\tau_\theta$ and $\sigma_\theta$ defined in \eqref{def:tau} and \eqref{def:sigmatheta}, respectively,
and hence the prox stepsize $\lambda$ and the parameter $\sigma$ in
\eqref{definition of sigma} become too small as $\theta$ approaches $0$
(see the two columns under the ``theoretical version'' part of Table~1).
Since the smallness of $\sigma$ and $\lambda$ directly affects the number of ACG and outer iterations performed by
the method, respectively, it is natural to consider its more aggressive variant (referred to here as the constant version of the $\theta$-IPAAL method)
which simply sets these parameters to $\lam=0.5/m$ and $\sigma^2=0.5$.
Even though our theoretical results derived in
Sections~\ref{sec-algorithms} and \ref{sec:Technical-Results} do not apply to this constant version, the computational results reported
in Table~2  and 3 show that
it performs considerably better than its theoretical
counterpart.

The implementation of the theoretical version of
the $\theta$-IPAAL method essentially follows its
description in Subsection~\ref{subsec:dynamicIPAAL} with the exception that
$\sigma$ is set to be $\sigma_\theta$ instead of
the value in \eqref{definition of sigma}.
It is straightforward to see that this choice of
$\sigma$ would still be covered by our analysis if we modify the formula in \eqref{definition of sigma} for $\sigma$ to
$$\sigma= \min\left\{\frac{\tau}{\sqrt{\lambda L_c + 1}}, \sigma_\theta\right\}$$
where $\tau>0$ is a fixed constant.
Choosing $\tau>0$ large enough, it clearly follows
that $\sigma=\sigma_\theta$, and hence that
the latter choice for $\sigma$ is still under the scope
of our analysis.

% For the theoretical version of the $\theta$-IPAAL method, we decided to set $\sigma=\sigma_\theta$. This is motivated by the fact that one can use the following definition of $\sigma$  instead of the one in \eqref{definition of sigma}: $$\sigma= \min\left\{\frac{\tau}{\sqrt{\lambda L_c + 1}}, \sigma_\theta\right\}$$ for some $\tau>0$,  and  basically the same theoretical analysis presented in this paper can be considered to this case which  shows that the latter definition of $\sigma$  does not compromise  the overall complexity bounds established in Section~3. Hence, the choice $\sigma=\sigma_\theta$ is natural if $\tau$ is chosen to be sufficiently large.

% ????
% and, as a result,
% the performance of the theoretical version of $\theta$-IPAAL method. Since the updating rules for  $\tau_\theta$ and $\sigma_\theta$  are  very conservative,  we  also consider the performance of the $\theta$-IPAAL method with  less conservative choices for these parameters, displayed as ``constant version" in tables~2 and 3, namely,

We now discuss implementation details which are common to
both versions.
The initial multiplier $\lm_0$ is set to be zero and the initial point $z_{0} \in S_{+}^{n}$ is
randomly generated, namely, $z_{0}=\nu\nu^{\top }$
where $\nu:=\tilde{\nu}/\|\tilde{\nu}\|$, $\tilde{\nu}\sim {\cal U}^{n}[0,1]$ with 10.0\% of the entries being nonzero. 
For a given tolerance pair $(\hat\rho,\hat\eta)\in \Re_{++}^2$, the $\theta$-IPAAL method stops when it obtains a point $(\hat{z},\hat{v},\hat{p})$ satisfying 
\begin{equation}
\begingroup\addtolength{\jot}{0.5em}
\begin{aligned}  \hat{v}\in\nabla f(\hat{z})+ \partial h(\hat{z})+A^{*}\hat{p},\quad\frac{\|\hat{v}\|}{\|\nabla f(z_{0})\|+1}\leq\hat{\rho},\quad \frac{\|A\hat{z}-b\|}{\|A z_0-b\|+1}\le\hat{\eta}.
\end{aligned}
\endgroup\label{eq:term_lin_constr}
\end{equation}
The penalty parameter  is chosen as $c=c_{1}:=10^{-5} L/(\|A\|^{2}+1)$ and is  updated according to
$c \leftarrow 5c$,
instead of $c \leftarrow 2c$ as in Step~2 of the $\theta$-IPAAL method.  It can be easily seen that such a choice does not  affect the overall iteration complexity of the $\theta$-IPAAL method.
Both versions use the following 
warm start strategy: instead of starting the call to the static $\theta$-IPAAL method in step~1 of the $\theta$-IPALL method
%in step 1... 
from the iterate-Lagrangian multiplier pair $(z_0,p_0)=(z_0,0)$ where $z_0$ is as in step 0 of the $\theta$-IPALL method, it initializes from $(z_0,p_0)=(\hat z,\hat \lm)$ where $(\hat z,\hat \lm)$ is the  pair obtained in the previous call to the static $\theta$-IPALL method.

% its executing
% the static $\theta$-IPAAL method with input ???? in step 1 of
% of $\theta$-IPAAL method, the warm strategey version changes the input
% to be the output $(\hat z,\hat \lm)$ obtained on its last
% callin each execution of the static $\theta$-IPALL method as in step~1 of the $\theta$-IPALL method,  the  initial pair $(z_0,\lm_0)$ 
% is set  to be the output $(\hat z,\hat \lm)$ obtained in the previous execution of it. %loop when $c$ was $c/5$.

%It should be mentioned that the $\theta$-IPAAL method will still work efficiently if the penalty parameter is updated exactly as described in its  Step~???, however it would require  more cycles which in turn would reflect in the overall number of ACG iterations. 

%=====================

%Hence, by considering %large $\tau$,   we may assume that $\sigma=\sigma_\theta$  in the numerical experiments. 
% The number of cycles also diminishes when $\theta$ decreases, specially when compared to $\theta=1$.

% Recall that  for $\theta=1$, the multiplier sequence is not updated in view of \eqref{def:multiplier}, and hence  the   $\theta$-IPAAL method with $\theta=1$ is a penalty scheme which  can be seen as a variant of the QP-AIPP of \cite{WJRproxmet1}.  

% Although Table~1 does not report the parameter $\sigma$ defined in \eqref{definition of sigma}, it can be easily seen that it becomes very small if  the penalty parameter $c$ is large or, equivalently, the number of cfycle increases.

%when $\theta$ decreases, the performance of  the $\theta$-IPAAL method improves in all but the last instance.

	\begin{table}[H]
\begin{centering}
\begin{tabular}{|>{\centering}m{2cm}|>{\centering}p{2cm}>{\centering}p{2cm}|>{\centering}p{2cm}>{\centering}p{2cm}|}
\hline 
\multirow{2}{2cm}{\centering{}{\footnotesize{}$\theta$}} &
\multicolumn{2}{c|}{{\small{}Theoretical version}} & \multicolumn{2}{c|}{{\small{}Constant version}} \tabularnewline
&{\footnotesize{}$\tau_\theta$} & {\footnotesize{}$\sigma^2$}&{\footnotesize{}$\tau_\theta$} & {\footnotesize{}$\sigma^2$}\tabularnewline
\hline
{\footnotesize{}$1$} & {\footnotesize{}0.5} &  {\footnotesize{}3.75e-02}& {\footnotesize{}0.5}& {\footnotesize{}0.5} \tabularnewline
%\hdashline
{\footnotesize{}$0.5$} & {\footnotesize{}0.067} &  {\footnotesize{}5.44e-04}& {\footnotesize{}0.5}& {\footnotesize{}0.5} \tabularnewline
%\hdashline
{\footnotesize{}$0.1$} & {\footnotesize{}0.0070} &  {\footnotesize{}8.08e-06}& {\footnotesize{}0.5}& {\footnotesize{}0.5} \tabularnewline
%\hdashline
{\footnotesize{}$0$} & {\footnotesize{}*} &  {\footnotesize{}*}& {\footnotesize{}0.5}& {\footnotesize{}0.5} \tabularnewline
\hline
\end{tabular}
\par\end{centering}
\caption{ The values of $\tau_\theta$ and $\sigma_\theta$.}
\end{table}

%	\begin{table}[H]
%\begin{centering}
%\begin{tabular}{|>{\centering}m{2cm}|>{\centering}p{2cm}>{\centering}p{2cm}|>{\centering}p{2cm}>{\centering}p{2cm}|}
%\hline 
%\multirow{2}{2cm}{\centering{}{\footnotesize{}$\theta$}} &
%\multicolumn{2}{c|}{{\small{}Theoretical version}} & \multicolumn{2}{c|}{{\small{}Constant version}} \tabularnewline
%&{\footnotesize{}$\tau_\theta$} & %{\footnotesize{}$\sigma_\theta^2$}&{\footnotesize{}$\tau_\theta$} & %{\footnotesize{}$\sigma_\theta^2$}\tabularnewline
%\hline
%{\footnotesize{}$1$} & {\footnotesize{}0.5} &  {\footnotesize{}3.75e-02}& %{\footnotesize{}0.5}& {\footnotesize{}0.5} \tabularnewline
%%\hdashline
%{\footnotesize{}$0.5$} & {\footnotesize{}0.067} &  {\footnotesize{}5.44e-04}& %{\footnotesize{}0.5}& {\footnotesize{}0.5} \tabularnewline
%\hdashline
%{\footnotesize{}$0.1$} & {\footnotesize{}0.0070} &  {\footnotesize{}8.08e-06}& {\footnotesize{}0.5}& {\footnotesize{}0.5} \tabularnewline
%\hdashline
%{\footnotesize{}$0$} & {\footnotesize{}*} &  {\footnotesize{}*}& {\footnotesize{}0.5}& %{\footnotesize{}0.5} \tabularnewline
%\hline
%\end{tabular}
%\par\end{centering}
%\caption{ The values of $\tau_\theta$ and $\sigma_\theta$.}
%\end{table}

%$\theta = 1, 0.5, 0.1,0$
Table~2 and 3 illustrate the performance of the $\theta$-IPAAL method with different values of the parameter $\theta$  for solving some instances of \eqref{LCQP1}. Table~2 displays the results for  instances with  $(l,n)=(5,20)$ whereas table~3 considers instances with  $(l,n)=(25,100)$.  In these two tables, ``ACG iter" denotes the total number of ACG iterations, ``outer iter" denotes the total number of outer iterations, ``cycle" is the number of cycles, i.e., the number of times that the penalty parameter is updated
(hence, every outer iteration within a cycle uses the same penalty
parameter $c$), and runtime is  in seconds.

\begin{table}[H]
\begin{centering}
\makebox[\textwidth][c]{%
\begin{tabular}{|>{\centering}m{1.4cm}>{\centering}m{0.6cm}|>{\centering}p{1.6cm}
>{\centering}p{1.6cm}>{\centering}p{1cm}>{\centering}p{1.2cm}|>{\centering}p{1.6cm}>{\centering}p{1.6cm}>{\centering}p{1cm}>{\centering}p{1.2cm}|}
\hline 
\multirow{2}{1.4cm}{\centering{}{\footnotesize{}$(L,m)$}} & \multirow{2}{0.6cm}{\centering{}{\footnotesize{}$\theta$}} & 
\multicolumn{4}{c|}{{\small{}Theoretical version}} & \multicolumn{4}{c|}{{\small{}Constant version}} \tabularnewline
& &
%\multicolumn{4}{c|}{{\footnotesize{} ($\tau_\theta$ = 0.5, $\sigma_\theta$ = 0.194)}} & \multicolumn{4}{c|}{{\footnotesize{}($\tau_\theta$ = 0.067, $ \sigma_\theta$= 0.023)}}  \tabularnewline
%& &
{\footnotesize{}ACG iter} & {\footnotesize{}Outer iter} & {\footnotesize{}Cycle} & {\footnotesize{}Runtime} & {\footnotesize{}ACG iter} & {\footnotesize{}Outer iter} & {\footnotesize{}Cycle} & {\footnotesize{}Runtime} \tabularnewline
\hline

{\footnotesize{}} & {\footnotesize{}$1$} & {\footnotesize{}25704} &  {\footnotesize{}16}& {\footnotesize{}13}& {\footnotesize{}37.91}  & {\footnotesize{}6606}& {\footnotesize{}16} & {\footnotesize{}13}  & {\footnotesize{}8.71} \tabularnewline
{\footnotesize{}$(10^{4},1)$} & {\footnotesize{}$0.5$} & {\footnotesize{}7404} &  {\footnotesize{}27}& {\footnotesize{}12}& {\footnotesize{}12.37}  & {\footnotesize{}2639}& {\footnotesize{}15} & {\footnotesize{}12}  & {\footnotesize{}4.14} \tabularnewline
{\footnotesize{}} & {\footnotesize{}$0.1$} & {\footnotesize{}5188} &  {\footnotesize{}129}& {\footnotesize{}11}& {\footnotesize{}7.44}  & {\footnotesize{}1323}& {\footnotesize{}14} & {\footnotesize{}11}  & {\footnotesize{}2.34} \tabularnewline
{\footnotesize{}} & {\footnotesize{}$0$} & {\footnotesize{}*} &  {\footnotesize{}*}& {\footnotesize{}*}& {\footnotesize{}*}  & {\footnotesize{}756}& {\footnotesize{}13} & {\footnotesize{}10}  & {\footnotesize{}1.44} \tabularnewline
\hline
{\footnotesize{}} & {\footnotesize{}$1$} & {\footnotesize{}93443} &  {\footnotesize{}13}& {\footnotesize{}13}& {\footnotesize{}145.78}  & {\footnotesize{}25697}& {\footnotesize{}14} & {\footnotesize{}13}  & {\footnotesize{}45.01} \tabularnewline
{\footnotesize{}$(10^{5},1)$} & {\footnotesize{}$0.5$} & {\footnotesize{}24337} &  {\footnotesize{}13}& {\footnotesize{}12}& {\footnotesize{}33.75}  & {\footnotesize{}10092}& {\footnotesize{}13} & {\footnotesize{}12}  & {\footnotesize{}15.78} \tabularnewline
{\footnotesize{}} & {\footnotesize{}$0.1$} & {\footnotesize{}7662} &  {\footnotesize{}28}& {\footnotesize{}11}& {\footnotesize{}11.92}  & {\footnotesize{}4057}& {\footnotesize{}12} & {\footnotesize{}11}  & {\footnotesize{}6.45} \tabularnewline
{\footnotesize{}} & {\footnotesize{}$0$} & {\footnotesize{}*} &  {\footnotesize{}*}& {\footnotesize{}*}& {\footnotesize{}*}  & {\footnotesize{}2226}& {\footnotesize{}11} & {\footnotesize{}10}  & {\footnotesize{}3.52} \tabularnewline
\hline
{\footnotesize{}} & {\footnotesize{}$1$} & {\footnotesize{}328146} &  {\footnotesize{}13}& {\footnotesize{}13}& {\footnotesize{}458.28}  & {\footnotesize{}94579}& {\footnotesize{}13} & {\footnotesize{}13}  & {\footnotesize{}179.83} \tabularnewline
{\footnotesize{}$(10^{6},1)$} & {\footnotesize{}$0.5$} & {\footnotesize{}89737} &  {\footnotesize{}12}& {\footnotesize{}12}& {\footnotesize{}149.06}  & {\footnotesize{}40578}& {\footnotesize{}12} & {\footnotesize{}12}  & {\footnotesize{}75.18} \tabularnewline
{\footnotesize{}} & {\footnotesize{}$0.1$} & {\footnotesize{}19568} &  {\footnotesize{}12}& {\footnotesize{}11}& {\footnotesize{}32.65}  & {\footnotesize{}17491}& {\footnotesize{}11} & {\footnotesize{}11}  & {\footnotesize{}32.39} \tabularnewline
{\footnotesize{}} & {\footnotesize{}$0$} & {\footnotesize{}*} &  {\footnotesize{}*}& {\footnotesize{}*}& {\footnotesize{}*}  & {\footnotesize{}8005}& {\footnotesize{}10} & {\footnotesize{}10}  & {\footnotesize{}13.23} \tabularnewline
\hline
{\footnotesize{}} & {\footnotesize{}$1$} & {\footnotesize{}327119} &  {\footnotesize{}13}& {\footnotesize{}13}& {\footnotesize{}517.19}  & {\footnotesize{}94613}& {\footnotesize{}13} & {\footnotesize{}13}  & {\footnotesize{}143.58} \tabularnewline
{\footnotesize{}$(10^{7},10)$} & {\footnotesize{}$0.5$} & {\footnotesize{}89983} &  {\footnotesize{}12}& {\footnotesize{}12}& {\footnotesize{}142.12}  & {\footnotesize{}40719}& {\footnotesize{}12} & {\footnotesize{}12}  & {\footnotesize{}69.29} \tabularnewline
{\footnotesize{}} & {\footnotesize{}$0.1$} & {\footnotesize{}19791} &  {\footnotesize{}12}& {\footnotesize{}11}& {\footnotesize{}29.80}  & {\footnotesize{}17977}& {\footnotesize{}11} & {\footnotesize{}11}  & {\footnotesize{}25.74} \tabularnewline
{\footnotesize{}} & {\footnotesize{}$0$} & {\footnotesize{}*} &  {\footnotesize{}*}& {\footnotesize{}*}& {\footnotesize{}*}  & {\footnotesize{}7942}& {\footnotesize{}10} & {\footnotesize{}10}  & {\footnotesize{}11.34} \tabularnewline
\hline
{\footnotesize{}} & {\footnotesize{}$1$} & {\footnotesize{}93835} &  {\footnotesize{}13}& {\footnotesize{}13}& {\footnotesize{}150.35}  & {\footnotesize{}25791}& {\footnotesize{}14} & {\footnotesize{}13}  & {\footnotesize{}38.33} \tabularnewline
{\footnotesize{}$(10^{7},10^2)$} & {\footnotesize{}$0.5$} & {\footnotesize{}24160} &  {\footnotesize{}13}& {\footnotesize{}12}& {\footnotesize{}32.84}  & {\footnotesize{}10113}& {\footnotesize{}13} & {\footnotesize{}12}  & {\footnotesize{}13.98} \tabularnewline
{\footnotesize{}} & {\footnotesize{}$0.1$} & {\footnotesize{}7548} &  {\footnotesize{}28}& {\footnotesize{}11}& {\footnotesize{}10.22}  & {\footnotesize{}4189}& {\footnotesize{}12} & {\footnotesize{}11}  & {\footnotesize{}6.36} \tabularnewline
{\footnotesize{}} & {\footnotesize{}$0$} & {\footnotesize{}*} &  {\footnotesize{}*}& {\footnotesize{}*}& {\footnotesize{}*}  & {\footnotesize{}2226}& {\footnotesize{}11} & {\footnotesize{}10}  & {\footnotesize{}3.91} \tabularnewline
\hline
{\footnotesize{}} & {\footnotesize{}$1$} & {\footnotesize{}26061} &  {\footnotesize{}16}& {\footnotesize{}13}& {\footnotesize{}44.04}  & {\footnotesize{}6552}& {\footnotesize{}16} & {\footnotesize{}13}  & {\footnotesize{}10.24} \tabularnewline
{\footnotesize{}$(10^{7},10^3)$} & {\footnotesize{}$0.5$} & {\footnotesize{}7424} &  {\footnotesize{}27}& {\footnotesize{}12}& {\footnotesize{}11.28}  & {\footnotesize{}2639}& {\footnotesize{}15} & {\footnotesize{}12}  & {\footnotesize{}4.43} \tabularnewline
{\footnotesize{}} & {\footnotesize{}$0.1$} & {\footnotesize{}5208} &  {\footnotesize{}129}& {\footnotesize{}11}& {\footnotesize{}7.20}  & {\footnotesize{}1323}& {\footnotesize{}14} & {\footnotesize{}11}  & {\footnotesize{}1.97} \tabularnewline
{\footnotesize{}} & {\footnotesize{}$0$} & {\footnotesize{}*} &  {\footnotesize{}*}& {\footnotesize{}*}& {\footnotesize{}*}  & {\footnotesize{}756}& {\footnotesize{}13} & {\footnotesize{}10}  & {\footnotesize{}1.43} \tabularnewline
\hline
\end{tabular}}
\par\end{centering}
\caption{Performance of the $\theta$-IPAAL method with $(l,n) = (5,20),\hat{\rho}=\hat{\eta}=10^{-4}$.}
\end{table}

\begin{table}[H]
\begin{centering}
\makebox[\textwidth][c]{%
\begin{tabular}{|>{\centering}m{1.4cm}>{\centering}m{0.6cm}|>{\centering}p{1.6cm}
>{\centering}p{1.6cm}>{\centering}p{1cm}>{\centering}p{1.2cm}|>{\centering}p{1.6cm}>{\centering}p{1.6cm}>{\centering}p{1cm}>{\centering}p{1.2cm}|}
\hline 
\multirow{2}{1.4cm}{\centering{}{\footnotesize{}$(L,m)$}} & \multirow{2}{0.6cm}{\centering{}{\footnotesize{}$\theta$}} & 
\multicolumn{4}{c|}{{\small{}Theoretical version}} & \multicolumn{4}{c|}{{\small{}Constant version}} \tabularnewline
& &
%\multicolumn{4}{c|}{{\footnotesize{} ($\tau_\theta$ = 0.5, $\sigma_\theta$ = 0.194)}} & \multicolumn{4}{c|}{{\footnotesize{}($\tau_\theta$ = 0.067, $ \sigma_\theta$= 0.023)}}  \tabularnewline
%& &
{\footnotesize{}ACG iter} & {\footnotesize{}Outer iter} & {\footnotesize{}Cycle} & {\footnotesize{}Runtime} & {\footnotesize{}ACG iter} & {\footnotesize{}Outer iter} & {\footnotesize{}Cycle} & {\footnotesize{}Runtime} \tabularnewline
\hline

{\footnotesize{}} & {\footnotesize{}$1$} & {\footnotesize{}140390} &  {\footnotesize{}20}& {\footnotesize{}15}& {\footnotesize{}1037.44}  & {\footnotesize{}35221}& {\footnotesize{}20} & {\footnotesize{}15}  & {\footnotesize{}288.64} \tabularnewline
{\footnotesize{}$(10^{4},1)$} & {\footnotesize{}$0.5$} & {\footnotesize{}33794} &  {\footnotesize{}27}& {\footnotesize{}14}& {\footnotesize{}250.59}  & {\footnotesize{}11787}& {\footnotesize{}19} & {\footnotesize{}14}  & {\footnotesize{}106.09} \tabularnewline
{\footnotesize{}} & {\footnotesize{}$0.1$} & {\footnotesize{}10218} &  {\footnotesize{}103}& {\footnotesize{}13}& {\footnotesize{}92.99}  & {\footnotesize{}4054}& {\footnotesize{}19} & {\footnotesize{}13}  & {\footnotesize{}34.83} \tabularnewline
{\footnotesize{}} & {\footnotesize{}$0$} & {\footnotesize{}*} &  {\footnotesize{}*}& {\footnotesize{}*}& {\footnotesize{}*}  & {\footnotesize{}1753}& {\footnotesize{}18} & {\footnotesize{}12}  & {\footnotesize{}15.26} \tabularnewline
\hline
{\footnotesize{}} & {\footnotesize{}$1$} & {\footnotesize{}496375} &  {\footnotesize{}15}& {\footnotesize{}15}& {\footnotesize{}3739.52}  & {\footnotesize{}139589}& {\footnotesize{}16} & {\footnotesize{}15}  & {\footnotesize{}1420.47} \tabularnewline
{\footnotesize{}$(10^{5},1)$} & {\footnotesize{}$0.5$} & {\footnotesize{}125599} &  {\footnotesize{}17}& {\footnotesize{}14}& {\footnotesize{}1223.06}  & {\footnotesize{}55049}& {\footnotesize{}15} & {\footnotesize{}14}  & {\footnotesize{}437.98} \tabularnewline
{\footnotesize{}} & {\footnotesize{}$0.1$} & {\footnotesize{}27887} &  {\footnotesize{}35}& {\footnotesize{}13}& {\footnotesize{}264.72}  & {\footnotesize{}20826}& {\footnotesize{}14} & {\footnotesize{}13}  & {\footnotesize{}174.20} \tabularnewline
{\footnotesize{}} & {\footnotesize{}$0$} & {\footnotesize{}*} &  {\footnotesize{}*}& {\footnotesize{}*}& {\footnotesize{}*}  & {\footnotesize{}6044}& {\footnotesize{}13} & {\footnotesize{}12}  & {\footnotesize{}54.30} \tabularnewline
\hline
{\footnotesize{}} & {\footnotesize{}$1$} & {\footnotesize{}1662990} &  {\footnotesize{}15}& {\footnotesize{}15}& {\footnotesize{}11582.70}  & {\footnotesize{}507773}& {\footnotesize{}15} & {\footnotesize{}15}  & {\footnotesize{}3525.21} \tabularnewline
{\footnotesize{}$(10^{6},1)$} & {\footnotesize{}$0.5$} & {\footnotesize{}431788} &  {\footnotesize{}14}& {\footnotesize{}14}& {\footnotesize{}3383.33}  & {\footnotesize{}212601}& {\footnotesize{}14} & {\footnotesize{}14}  & {\footnotesize{}1500.66} \tabularnewline
{\footnotesize{}} & {\footnotesize{}$0.1$} & {\footnotesize{}90949} &  {\footnotesize{}16}& {\footnotesize{}13}& {\footnotesize{}635.13}  & {\footnotesize{}88880}& {\footnotesize{}13} & {\footnotesize{}13}  & {\footnotesize{}646.18} \tabularnewline
{\footnotesize{}} & {\footnotesize{}$0$} & {\footnotesize{}*} &  {\footnotesize{}*}& {\footnotesize{}*}& {\footnotesize{}*}  & {\footnotesize{}33908}& {\footnotesize{}12} & {\footnotesize{}12}  & {\footnotesize{}224.12} \tabularnewline
\hline
{\footnotesize{}} & {\footnotesize{}$1$} & {\footnotesize{}1705212} &  {\footnotesize{}15}& {\footnotesize{}15}& {\footnotesize{}12073.87}  & {\footnotesize{}508098}& {\footnotesize{}15} & {\footnotesize{}15}  & {\footnotesize{}4126.89} \tabularnewline
{\footnotesize{}$(10^{7},10)$} & {\footnotesize{}$0.5$} & {\footnotesize{}435710} &  {\footnotesize{}14}& {\footnotesize{}14}& {\footnotesize{}3957.91}  & {\footnotesize{}212679}& {\footnotesize{}14} & {\footnotesize{}14}  & {\footnotesize{}1981.36} \tabularnewline
{\footnotesize{}} & {\footnotesize{}$0.1$} & {\footnotesize{}90586} &  {\footnotesize{}16}& {\footnotesize{}13}& {\footnotesize{}1012.90}  & {\footnotesize{}88970}& {\footnotesize{}13} & {\footnotesize{}13}  & {\footnotesize{}869.04} \tabularnewline
{\footnotesize{}} & {\footnotesize{}$0$} & {\footnotesize{}*} &  {\footnotesize{}*}& {\footnotesize{}*}& {\footnotesize{}*}  & {\footnotesize{}34858}& {\footnotesize{}12} & {\footnotesize{}12}  & {\footnotesize{}249.89} \tabularnewline
\hline
{\footnotesize{}} & {\footnotesize{}$1$} & {\footnotesize{}493067} &  {\footnotesize{}15}& {\footnotesize{}15}& {\footnotesize{}4793.98}  & {\footnotesize{}139926}& {\footnotesize{}16} & {\footnotesize{}15}  & {\footnotesize{}978.11} \tabularnewline
{\footnotesize{}$(10^{7},10^2)$} & {\footnotesize{}$0.5$} & {\footnotesize{}125408} &  {\footnotesize{}17}& {\footnotesize{}14}& {\footnotesize{}894.30}  & {\footnotesize{}52874}& {\footnotesize{}15} & {\footnotesize{}14}  & {\footnotesize{}377.65} \tabularnewline
{\footnotesize{}} & {\footnotesize{}$0.1$} & {\footnotesize{}27941} &  {\footnotesize{}35}& {\footnotesize{}13}& {\footnotesize{}194.83}  & {\footnotesize{}20739}& {\footnotesize{}14} & {\footnotesize{}13}  & {\footnotesize{}165.19} \tabularnewline
{\footnotesize{}} & {\footnotesize{}$0$} & {\footnotesize{}*} &  {\footnotesize{}*}& {\footnotesize{}*}& {\footnotesize{}*}  & {\footnotesize{}5952}& {\footnotesize{}13} & {\footnotesize{}12}  & {\footnotesize{}42.42} \tabularnewline
\hline
{\footnotesize{}} & {\footnotesize{}$1$} & {\footnotesize{}139745} &  {\footnotesize{}20}& {\footnotesize{}15}& {\footnotesize{}977.19}  & {\footnotesize{}35227}& {\footnotesize{}20} & {\footnotesize{}15}  & {\footnotesize{}246.07} \tabularnewline
{\footnotesize{}$(10^{7},10^3)$} & {\footnotesize{}$0.5$} & {\footnotesize{}33462} &  {\footnotesize{}27}& {\footnotesize{}14}& {\footnotesize{}236.89}  & {\footnotesize{}11634}& {\footnotesize{}19} & {\footnotesize{}14}  & {\footnotesize{}84.47} \tabularnewline
{\footnotesize{}} & {\footnotesize{}$0.1$} & {\footnotesize{}10075} &  {\footnotesize{}103}& {\footnotesize{}13}& {\footnotesize{}72.31}  & {\footnotesize{}4041}& {\footnotesize{}19} & {\footnotesize{}13}  & {\footnotesize{}29.63} \tabularnewline
{\footnotesize{}} & {\footnotesize{}$0$} & {\footnotesize{}*} &  {\footnotesize{}*}& {\footnotesize{}*}& {\footnotesize{}*}  & {\footnotesize{}1753}& {\footnotesize{}18} & {\footnotesize{}12}  & {\footnotesize{}13.55} \tabularnewline
\hline
\end{tabular}}
\par\end{centering}
\caption{Performance of the $\theta$-IPAAL method with $(l,n) = (25,100),\hat{\rho}=\hat{\eta}=10^{-4}$.}
\end{table}

We conclude from Table~2 and 3  that  the total number of ACG iterations for both versions of the $\theta$-IPAAL method as well as the runtime decrease as $\theta$ approaches zero.
%Even though our theoretical results derived in
% Sections~\ref{sec-algorithms} and \ref{sec:Technical-Results} do not apply to the constant version of the
% $\theta$-IPAAL method, t
The computational results reported
in Table~2  and 3 show that the constant version
performs considerably better than its theoretical
counterpart. This phenomenon can be attributed to the
sizes of the scalars $\sigma$ and $\lambda$
which directly affect the number of ACG and outer iterations,
respectively.
% Both versions of the $\theta$-IPAAL method can be viewed
% as a variant of the method 
%  Note that for $\theta=1$, both methods have the same value of prox stepsize but different values of $\sigma$, see Table~1. % of 'sigma changes with, 
% Moreover, the  theoretical version with $\theta=0.1$  outperforms both versions with $\theta=1$, it should be noted that both versions with  $\theta=1$ can be seen as variants of  the QP-AIPP of \cite{WJRproxmet1}. F
Finally, even though the theoretical version is undefined for $\theta=0$, the above computational results indicate that its  constant version counterpart with $\theta = 0$ is quite promising.  
\section{Concluding remarks}\label{sec:ConcludingRemarks}

This paper has presented  an inexact proximal accelerated augmented Lagrangian (IPAAL) method, based
on the $\theta$-AL function \eqref{lagrangian2}, for finding
an approximate  stationary point of
the linearly constrained smooth  nonconvex composite optimization problem \eqref{optl0}
where the prox subproblems are inexactly solved by an accelerated composite gradient (ACG) scheme.
It is shown that the $\theta$-IPAAL obtains a
$(\hat\rho,\hat\eta)$-approximate stationary point
(see Definition~\ref{def:stationarypoint})
of \eqref{optl0} in ${\cal O}([1/(\hat\eta\hat\rho^2)]\log(1/\hat\eta))$ ACG iterations.
Moreover, it is also shown that the
previous bound can be improved to ${\cal O}([1/(\sqrt{\hat \eta} \hat \rho^2)]\log(1/\hat\eta))$ under the additional mildly stronger conditions {\bf (B1)--(B3)}.
The above bounds are derived assuming  that the initial point is neither feasible nor the domain of the composite term of the objective function is bounded.

We now make some remarks about the derived complexity bounds
in light of the choice of the parameter $\theta$.
Our complexity bounds are derived under the assumption that
$\theta \in (0,1]$, and hence does not apply to the case
in which $\theta=0$, i.e., the $\theta$-AL function
reduces to the classical quadratic AL function.
It turns out that as $\theta$ approaches zero, the
constant involved in the ${\cal O}(\cdot)$
iteration-complexity bounds explodes to infinity
and the stepsize $\lam$ given by \eqref{definition of sigma} approaches zero.
Hence, the case $\theta=0$ is still an open case
whose resolution possibly requires an insight 
different
than the one used in this paper.

Our analysis has assumed that the pair
$(z_0,\lm_0)$ used as part of the input in step 1 of
$\theta$-IPAAL is always the same pair.
In practice, this pair can be chosen using the
following simple warm
strategy, namely, set it to be the output
$(\hat z,\hat \lm)$ obtained in step 2 of the
$\theta$-IPAAL during its previous loop
when $c$ was $c/2$. An interesting topic for future
research is to analyze the ACG iteration-complexity of the
$\theta$-IPAAL method endowed with this warm strategy.

% The considered augmented Lagrangian (AL) function depends on a parameter $\theta\in (0,1]$.  When $\theta=1$, it  corresponds to the usual quadratic penalty function and the IPAAL method is closely related to the one studied in \cite{MJR}. The classical quadratic AL function corresponds to the  singular case in which  $\theta=0$. It is an open problem to know whether or not the iteration-complexity bounds established here still hold for this singular case, under similar assumptions. The authors are also interested in analyzing the IPAAL method when the subproblemas are solved by using a warm start strategy, i.e., the best approximate solution obtained in a subproblem of the IPAAL method is used as an initial point to compute the solution of the next subproblem for which the penalty parameter is updated.

	\appendix
	%\section{Appendix}\label{sec:appendix}
	\section{Two Technical Results} \label{techresults}
	
	This section contains two technical results concerning some properties of the $\varepsilon$-subdifferential of a convex function perturbed by a prox-term.
	
	The following result is used in the proof of Lemma~\ref{pr:aux-new2}.
	
	\begin{lemma}\label{lem:proxinclusion}
		Assume that $\xi>0$, $\psi \in \bConv{n}$ and $Q \in {\cal S}^n_{++}$ are such that
		$\psi - (\xi/2) \|\cdot\|^2_Q$ is convex and let
		$(y,v,\eta) \in \Re^n \times \Re^n \times  \Re$ be such that $v\in \partial_\eta \psi(y)$.
		Then, %for any $ u \in \dom \psi$,
		for any $\tau>0$,
		
		\[
		\psi(u) \ge \psi(y) + \inner{v}{u-y} - (1+\tau^{-1})\eta+ \frac{\xi}{2(1+\tau)} \|u-y\|_Q^2  \quad \forall u \in \Re^n.
		\]
	\end{lemma}
	\begin{proof}
		Let $\psi_v := \psi-\inner{v}{\cdot}$.
		The assumptions imply that $\psi_v$ has a unique global minimum $\bar y $
		and that
		\begin{equation}\label{ineq:u}
		\psi_v(u)\ge \psi_v(\bar y)+\frac{\xi}{2}\|u-\bar y\|_Q^2
		\end{equation}
		for every $u \in \Re^n$. Moreover, Since  $ v \in \partial_\eta \psi(y) $, we have 
		$\psi_v(u) \ge \psi_v(y) - \eta$ for every $u \in \Re^n$, and hence that
		\begin{equation}\label{ineq:bar y}
		\psi_v(\bar y) \ge \psi_v(y) - \eta.
		\end{equation}
		The relations in \eqref{ineq:u} and \eqref{ineq:bar y} then imply that for every $u \in \Re^n$,
		\begin{align}
		\psi_v(u)&
		%\ge \psi(\bar y)+\frac{\xi}{2}\|u-\bar y\|^2
		\ge \psi_v(y)-\eta+\frac{\xi}{2}\|u-\bar y\|_Q^2 = \psi_v(y)-\left (\eta+\frac{\xi}{2\tau}\|\bar y-y\|_Q^2 \right) +\frac{\xi}{2}  \left[  \frac{1}{\tau} \|y-\bar y\|_Q^2+\|u-\bar y\|_Q^2 \right] \nonumber \\
		&\ge \psi_v(y)-\eta'+\frac{\xi}{2(1+\tau)}\|u-y\|_Q^2 \label{ineq:u y},
		\end{align}
		where the last inequality is due to the following relations
		\[
		\frac{1}{1+\tau} \|\tilde u+ u'\|^2 \le  \frac{1}{\tau}\|\tilde u\|^2+  \|u'\|^2, \qquad  
		\eta':=\eta+\frac{\xi}{2\tau}\|\bar y-y\|_Q^2.
		\]
		Also, inequality \eqref{ineq:u} with $u=y$
			and the relation in \eqref{ineq:bar y} imply that $(\xi/2) \|\bar y-y\|_Q^2 \le \eta$ and hence that
			$ \eta' \le (1+\tau^{-1})\eta$. 
			The conclusion now follows from \eqref{ineq:u y}, definition of $\psi_v$ and the latter conclusion.	\end{proof}
		
		The following result is used in the proof of Proposition~\ref{prop:constante0}.

		\begin{lemma} \label{lem:auxNewNest2}
			Let proper function $\tilde \phi: \Re^n \to (-\infty,\infty]$,
			scalar $\sigma \in (0,1)$  and
			$(z_0,z_1) \in \Re^n \times     \dom \tilde \phi$ be given, and assume that there exists
			$(v_1,\varepsilon_1)$ such that
			\begin{gather}
			v_1 \in \partial_{\varepsilon_1} \left(\tilde\phi+\frac{1}{2}\|\cdot-z_0\|^2\right) (z_1), 
			\quad  \|v_1\|^2 + 2 \varepsilon_1 \leq\sigma^2 \|v+z_{0}-z_1\|^{2}. \label{Auxeq:prox_incl}  
			\end{gather}
			Then, for every $z\in \Re^n$ and $s>0$, we have
			\[
			\tilde \phi(z_1)+\frac{1}{2} \left[ 1 - \sigma^2 ( 1 + s^{-1}) \right]\|v_1+z_0 - z_1\|^{2}\le \tilde \phi(z) +\frac{s+1}{2} \|z-z_0\|^2.
			\]
		\end{lemma}
		
		\begin{proof}
			Using the inclusion in \eqref{Auxeq:prox_incl}, the definition of
			$\varepsilon$-subdifferential in \eqref{def:epsSubdiff}, and the fact that $|\inner{u}{\tilde u}| \le [s\|u\|^2+s^{-1}\|\tilde u\|^2]/2$
			for every $u, \tilde u \in \Re^n$  and $s>0$, we conclude that for every $z\in \Re^n$,
			\begin{align*}
			\tilde\phi(z)  + \frac{\|z-z_0\|^2}2- \tilde \phi(z_1) &\ge \frac{\|z_1-z_0\|^2}2 + \inner{v_1}{z-z_1} - \varepsilon_1 \\
			&= \frac{\|z_1-z_0\|^2}2 + \inner{v_1}{z_0-z_1} + \inner{v_1}{z-z_0} - \varepsilon_1 \\
			&\ge \frac{\|z_1-z_0\|^2}2 + \inner{v_1}{z_0-z_1} - \varepsilon_1 - \frac{\|v_1\|^2}{2s} - \frac{s \|z-z_0\|^2}{2}\\
			&\geq \frac{\|v_1+z_0-z_1\|^2}2 - \frac12 ( 1 + s^{-1}) \left[  \|v_1\|^2+2 \varepsilon_1  \right] - \frac{s \|z-z_0\|^2}{2},
			%&\ge \frac{1}{2} [ 1 - \sigma^2 ( 1 + s^{-1}) ] \|v_1+z_0-z_1\|^2- \frac{s \|z-z_0\|^2}{2},
			\end{align*}
			which immediately implies the conclusion of the lemma in view of
			the inequality in \eqref{Auxeq:prox_incl}.		\end{proof}

		\section{Proof of the first statement of Lemma~\ref{lem:auxdecL2}}\label{sec:proof of Ctheta is positive}
		
		In this section we prove the  first statement of Lemma~\ref{lem:auxdecL2}, i.e.,   
		$C_\theta \geq 1/8$, where $C_\theta$ is as in \eqref{def:BCtheta}.
		
		\begin{proof} First note that the definition of $\tau_\theta$ in \eqref{def:tau} immediately yields
		$\tau_\theta(16-17\theta)/\theta\leq 1,$
		which can be easily seen to be equivalent to 
	    \begin{equation}\label{auxeq9900}
        \frac{2(1-\theta)}{\theta}\left[ \frac{2\tau_\theta}{\tau_\theta+1}\right]\leq \frac{1}{4}.
	    \end{equation}
		Now, in view of \eqref{def:BCtheta},   $C_\theta\geq 1/8$ if and only if $\sigma$ as in \eqref{definition of sigma} satisfies
		$$
		\frac{\sigma^2}{2}+\frac{2(1-\theta)}{\theta}
		\left[ \frac{2\tau_\theta(1+\sigma)^2}{\tau_\theta+1} + 2\sigma(1+\sigma) +\frac{(\tau_\theta+1)\sigma^2}{\tau_\theta}\right]\leq \frac{3}{8}.
		$$
		Hence,  in view of \eqref{auxeq9900}, in order  to show that $C_\theta\geq 1/8$, it is sufficient to prove that
		$$
		\frac{\sigma^2}{2}+ \frac{(1+\sigma)^2}{4} + \frac{2(1-\theta)}{\theta}
		\left[2\sigma(1+\sigma) +\frac{(\tau_\theta+1)\sigma^2}{\tau_\theta}\right]\leq \frac{3}{8},
		$$
		or equivalently
		$$
		\left(\frac{3}{4} + \frac{2(1-\theta)\left(3\tau_\theta+1\right)}{\theta\tau_\theta}\right) \sigma^2 + \left(\frac{8-7\theta}{2\theta}\right)\sigma - \frac{1}{8}\leq 0.
		$$
		Since $\theta \in (0,1]$ and $0<\sigma \leq \sigma_\theta$ in view of  \eqref{definition of sigma}, the above inequality holds immediately from the fact that $\sigma_\theta$ is  the only positive solution of equation  \eqref{def:sigmatheta}  associated to the above quadratic inequality.
		Hence, from the above conclusions, we obtain  $C_\theta \geq 1/8$.\end{proof}

		\bibliographystyle{plain}
		\bibliography{IPAAL-SIAM-04-06-20}
		
	\end{document}

%% file: def.tex
\usepackage{amsmath}
\usepackage{amsfonts}
\usepackage{latexsym}
\usepackage{amssymb}
\usepackage{framed}

\newcommand{\beq}{\begin{equation}}
\newcommand{\eeq}{\end{equation}}
\newcommand{\beqa}{\begin{eqnarray}}
\newcommand{\eeqa}{\end{eqnarray}}
\newcommand{\beqas}{\begin{eqnarray*}}
\newcommand{\eeqas}{\end{eqnarray*}}
\newcommand{\bi}{\begin{itemize}}
\newcommand{\ei}{\end{itemize}}

\newcommand{\vgap}{\vspace{.1in}}

\newcommand{\lam}{{\lambda}}

\newcommand{\inner}[2]{\langle #1,#2\rangle}
\newcommand{\Inner}[2]{\left \langle #1\,,#2 \right \rangle}

\newcommand{\cl}{\mathrm{cl}\,}

\newcommand{\argmin}{\mathrm{argmin}}

\newcommand{\inte}{\mathrm{int}\,}

\newcommand{\dist}{\mathrm{dist}\,}
\newcommand{\dom}{\mathrm{dom}\,}

\newcommand{\cball}[2]{\mbox{$\bar {\it B}$}(#1;#2)}

\newcommand{\bConv}[1]{\mbox{\rm C}\overline{\mbox{\rm onv}}\,(\Re^{#1})}

%% file: Teste_shared.tex
\usepackage{lipsum}
\usepackage{amsfonts}
\usepackage{graphicx}
\usepackage{epstopdf}
\usepackage{algorithmic}
\ifpdf
  \DeclareGraphicsExtensions{.eps,.pdf,.png,.jpg}
\else
  \DeclareGraphicsExtensions{.eps}
\fi

\newcommand{\TheTitle}{Iteration-complexity of an inexact proximal accelerated augmented Lagrangian method for solving linearly constrained smooth nonconvex composite optimization problems}

\headers{inexact prox accelerated augmented Lagrangian method}{J. G. Melo, R. D. C. Monteiro, and H. Wang}

\title{{\TheTitle}}

\author{
    Jefferson G. Melo \thanks{Instituto de Matem\'atica e Estat\'istica, Universidade Federal de Goi\'as, Campus II- Caixa Postal 131, CEP 74001-970, Goi\^ania-GO, Brazil. (E-mail: {\tt jefferson@ufg.br}). The work of this author was partially supported by CNPq grant 312559/2019-4 and FAPEG/GO.}
      \and 
    Renato D.C. Monteiro \thanks{School of Industrial and Systems Engineering, Georgia Institute of Technology, Atlanta, GA, 30332-0205. (E-mails: {\tt monteiro@isye.gatech.edu} and  {\tt hairongwhr@gatech.edu}). The work of Renato D.C. Monteiro
    was partially supported by ONR Grant N00014-18-1-2077.}
    \and
    Hairong Wang \footnotemark[2]
}
\usepackage{amsopn}

%% file: IPAAL-SIAM-04-06-20.bbl
\def\cprime{$'$}
\begin{thebibliography}{10}

\bibitem{Attouch2016}
H.~Attouch and J.~Peypouquet.
\newblock The rate of convergence of {N}esterov{\textquotesingle}s accelerated
  forward-backward method is actually faster than $1/k^{2}$.
\newblock {\em SIAM J. Optim.}, 26(3):1824--1834, 2016.

\bibitem{Aybatpenalty}
N.S. Aybat and G.~Iyengar.
\newblock A first-order smoothed penalty method for compressed sensing.
\newblock {\em SIAM J. Optim.}, 21(1):287--313, 2011.

\bibitem{AybatAugLag}
N.S. Aybat and G.~Iyengar.
\newblock A first-order augmented {Lagrangian} method for compressed sensing.
\newblock {\em SIAM J. Optim.}, 22(2):429--459, 2012.

\bibitem{beck2009fast}
A.~Beck and M.~Teboulle.
\newblock A fast iterative shrinkage-thresholding algorithm for linear inverse
  problems.
\newblock {\em SIAM Journal on Imaging Sciences}, 2(1):183--202, 2009.

\bibitem{Ber1}
D.~P. Bertsekas.
\newblock {\em Constrained optimization and Lagrange multiplier methods}.
\newblock Academic Press, New York, 1982.

\bibitem{Lan-ConstrainedStocasticProxMetNonconvex2019}
D.~Boob, Q.~Deng, and G.~Lan.
\newblock Stochastic first-order methods for convex and nonconvex functional
  constrained optimization.
\newblock {\em Available on arxiv:1908.02734}, 2019.

\bibitem{Aaronetal2017}
Yair Carmon, John~C. Duchi, Oliver Hinder, and Aaron Sidford.
\newblock Accelerated methods for nonconvex optimization.
\newblock {\em SIAM Journal on Optimization}, 28(2):1751--1772, 2018.

\bibitem{CartToint}
C.~Cartis, N.~Gould, and P.~Toint.
\newblock On the complexity of steepest descent, {Newton}'s and regularized
  {Newton}'s methods for nonconvex unconstrained optimization problems.
\newblock {\em SIAM J. Optim.}, 20(6):2833--2852, 2010.

\bibitem{GlanPDaccel2014}
Y.~Chen, G.~Lan, and Y.~Ouyang.
\newblock Optimal primal-dual methods for a class of saddle point problems.
\newblock {\em SIAM J. Optim.}, 24(4):1779--1814, 2014.

\bibitem{Paquette2017}
D.~Drusvyatskiy and C.~Paquette.
\newblock Efficiency of minimizing compositions of convex functions and smooth
  maps.
\newblock {\em Math. Programming}, 178:503--558, 2019.

\bibitem{nonconv_lan16}
S.~Ghadimi and G.~Lan.
\newblock Accelerated gradient methods for nonconvex nonlinear and stochastic
  programming.
\newblock {\em Math. Programming}, 156:59--99, 2016.

\bibitem{Ghadimi2019}
S.~Ghadimi, G.~Lan, and H.~Zhang.
\newblock Generalized uniformly optimal methods for nonlinear programming.
\newblock {\em Journal of Scientific Computing}, 79(3):1854--1881, Jun 2019.

\bibitem{MaxJeffRen-admm}
M.L.N. Gonçalves, J.G. Melo, and R.D.C. Monteiro.
\newblock Convergence rate bounds for a proximal admm with over-relaxation
  stepsize parameter for solving nonconvex linearly constrained problems.
\newblock {\em Pac. J. Optim.}, 15(3):379--398, 2019.

\bibitem{HongPertAugLag}
D.~Hajinezhad1 and M.~Hong.
\newblock Perturbed proximal primal–dual algorithm for nonconvex nonsmooth
  optimization.
\newblock {\em Math. Programming}, 176:207--245, 2019.

\bibitem{YheMoneiroNash}
Y.~He and R.~D.~C. Monteiro.
\newblock Accelerating block-decomposition first-order methods for solving
  composite saddle-point and two-player {N}ash equilibrium problems.
\newblock {\em SIAM J. Optim.}, 25(4):2182--2211, 2015.

\bibitem{YHe2}
Y.~He and R.~D.~C. Monteiro.
\newblock An accelerated {HPE}-type algorithm for a class of composite
  convex-concave saddle-point problems.
\newblock {\em SIAM J. Optim.}, 26(1):29--56, 2016.

\bibitem{ProxAugLag_Ming}
M.~Hong.
\newblock Decomposing linearly constrained nonconvex problems by a proximal
  primal dual approach: algorithms, convergence, and applications.
\newblock {\em available on arXiv:1604.00543}, 2016.

\bibitem{SZhang-Pen-admm}
B.~Jiang, T.~Lin, S.~Ma, and S.~Zhang.
\newblock Structured nonconvex and nonsmooth optimization algorithms and
  iteration complexity analysis.
\newblock {\em Comput. Optim. Appl.}, 72(3):115–157, 2019.

\bibitem{OliverMonteiro}
O.~Kolossoski and R.~D.~C. Monteiro.
\newblock An accelerated non-euclidean hybrid proximal extragradient-type
  algorithm for convex-concave saddle-point problems.
\newblock {\em Optim. Methods Softw.}, 32(6):1244--1272, 2017.

\bibitem{WJRproxmet1}
W.~Kong, J.G. Melo, and R.D.C. Monteiro.
\newblock Complexity of a quadratic penalty accelerated inexact proximal point
  method for solving linearly constrained nonconvex composite programs.
\newblock {\em SIAM Journal on Optimization}, 29(4):2566--2593, 2019.

\bibitem{MinMax-RenWilliam}
W~Kong and R.D.C. Monteiro.
\newblock An accelerated inexact proximal point method for solving
  nonconvex-concave min-max problems.
\newblock {\em Available on arXiv:1905.13433v2}, 2019.

\bibitem{WJRComputQPAIPP}
Weiwei Kong, Jefferson~G. Melo, and Renato D.~C. Monteiro.
\newblock An efficient adaptive accelerated inexact proximal point method for
  solving linearly constrained nonconvex composite problems.
\newblock {\em Comput. Optim. Appl.}, 76(2):305--346, 2019.

\bibitem{LanRen2013PenMet}
G.~Lan and R.~D.~C. Monteiro.
\newblock Iteration-complexity of first-order penalty methods for convex
  programming.
\newblock {\em Math. Programming}, 138(1):115--139, Apr 2013.

\bibitem{LanMonteiroAugLag}
G.~Lan and R.~D.~C. Monteiro.
\newblock Iteration-complexity of first-order augmented {L}agrangian methods
  for convex programming.
\newblock {\em Math. Programming}, 155(1):511--547, Jan 2016.

\bibitem{Li_Lin2015}
H.~Li and Z.~Lin.
\newblock Accelerated proximal gradient methods for nonconvex programming.
\newblock {\em Adv. Neural Inf. Process. Syst.}, 28:379--387, 2015.

\bibitem{HybridPenaltyAugLag19}
Z.~Li and Y.~Xu.
\newblock First-order inexact augmented {L}agrangian methods for convex and
  nonconvex programs: nonergodic convergence and iteration complexity.
\newblock {\em Personal technical report}, 2019.

\bibitem{PPmetNonconvex2019}
Q.~Lin, R.~Ma, and Y.~Xu.
\newblock Inexact proximal-point penalty methods for non-convex optimization
  with non-convex constraints.
\newblock {\em Available on https://arxiv.org/abs/1908.11518}, 2019.

\bibitem{ShiqiaMaAugLag16}
Y.F. Liu, X.~Liu, and S.~Ma.
\newblock On the nonergodic convergence rate of an inexact augmented
  {L}agrangian framework for composite convex programming.
\newblock {\em Math. Oper. Res.}, 44(2):632--650, 2019.

\bibitem{zhaosongAugLag18}
Z.~{Lu} and Z.~{Zhou}.
\newblock {Iteration-complexity of first-order augmented Lagrangian methods for
  convex conic programming}.
\newblock {\em Available on arXiv:1803.09941}, 2018.

\bibitem{MontSvaiter_fista}
R.~D.~C. Monteiro, Ortiz, and Benar~F. Svaiter.
\newblock An adaptive accelerated first-order method for convex optimization.
\newblock {\em Comput. Optim. Appl.}, 64:31--73, 2016.

\bibitem{MonteiroSvaiterAcceleration}
R.~D.~C. Monteiro and B.~F. Svaiter.
\newblock An accelerated hybrid proximal extragradient method for convex
  optimization and its implications to second-order methods.
\newblock {\em SIAM J. Optim.}, 23(2):1092--1125, 2013.

\bibitem{MonteiroSvaiterNewton}
R.D.C. Monteiro and B.F. Svaiter.
\newblock Iteration-complexity of a {N}ewton proximal extragradient method for
  monotone variational inequalities and inclusion problems.
\newblock {\em SIAM Journal on Optimization}, 22(3):914--935, 2012.

\bibitem{IterComplConicprog}
I.~Necoara, A.~Patrascu, and F.~Glineur.
\newblock Complexity of first-order inexact {L}agrangian and penalty methods
  for conic convex programming.
\newblock {\em Optim. Methods Softw.}, pages 1--31, 2017.

\bibitem{nesterov2012gradient}
Y.~Nesterov.
\newblock Gradient methods for minimizing composite functions.
\newblock {\em Math. Programming}, pages 1--37, 2012.

\bibitem{nesterov1983method}
Y.~E. Nesterov.
\newblock {\em Introductory lectures on convex optimization : a basic course}.
\newblock Kluwer Academic Publ., 2004.

\bibitem{NesterovSec_ord}
Y.E. Nesterov and B.T. Polyak.
\newblock Cubic regularization of newton method and its global performance.
\newblock {\em Math. Program.}, 108(1):177--205, 2006.

\bibitem{CatalystNC}
C.~Paquette, H.~Lin, D.~Drusvyatskiy, J.~Mairal, and Z.~Harchaoui.
\newblock Catalyst for gradient-based nonconvex optimization.
\newblock In {\em AISTATS 2018-21st International Conference on Artificial
  Intelligence and Statistics}, pages 1--10, 2018.

\bibitem{Patrascu2017}
A.~Patrascu, I.~Necoara, and Q.~Tran-Dinh.
\newblock Adaptive inexact fast augmented {L}agrangian methods for constrained
  convex optimization.
\newblock {\em Optim. Lett.}, 11(3):609--626, 2017.

\bibitem{MR0418919}
R.~T. Rockafellar.
\newblock Augmented {L}agrangians and applications of the proximal point
  algorithm in convex programming.
\newblock {\em Math. Oper. Res.}, 1(2):97--116, 1976.

\bibitem{tseng2008accmet}
P.~Tseng.
\newblock On accelerated proximal gradient methods for convex-concave
  optimization.
\newblock {\em
  \tt{http://www.mit.edu/$_{\widetilde{~}}$dimitrib/PTseng/papers.html}}, 2008.

\bibitem{YangyangAugLag17}
Yangyang Xu.
\newblock Iteration complexity of inexact augmented lagrangian methods for
  constrained convex programming.
\newblock {\em Math. Programming}, 2019.

\end{thebibliography}
